\pdfoutput=1

\documentclass[12pt]{article}
\usepackage{modnatbib}
\usepackage{amssymb,amsmath,amsthm,epsfig,verbatim,url}
\usepackage{pstricks}
\usepackage{ifpdf}  
\newtheorem{thm}{\sc Theorem.}[section]
\newtheorem{lem}[thm]{\sc Lemma.}
\newtheorem{rem}[thm]{\sc Remark.}
\renewcommand{\theequation}{\arabic{section}.\arabic{equation}}
{{\upshape\bfseries AMS subject classifications. }\ignorespaces}{}
\newenvironment{keywords}{{\upshape\bfseries Key words. }\ignorespaces}{}

\newcommand{\Rgeq}{{\mathbb R}_{\geq 0}}

\newcommand{\R}{{\mathbb R}}
\newcommand{\Ds}{\mathcal{D}_\Gamma}

\newcommand{\spa}{\operatorname{span}}

\newcommand{\vol}{\mathcal{L}^d} 

\DeclareMathOperator{\esssup}{ess\,sup}

\newcommand{\Mloss}{\mathcal{L}_{\rm loss}}
\newcommand{\GT}{{\mathcal{G}_T}}
\newcommand{\GhT}{{\mathcal{G}^h_T}}

\newcommand{\dH}[1]{\;{\rm d}{\cal H}^{#1}} 
\newcommand{\dL}[1]{\;{\rm d}{\cal L}^{#1}} 

\newcommand{\bigchi}{\ensuremath{\mathrm{\mathcal{X}}}}
\newcommand{\charfcn}[1]{\bigchi_{#1}} 
\newcommand{\Domain}{\Omega}

\newcommand{\Vhz}{\underline{V}(\Gamma^0)}

\newcommand{\Whp}{W(\Gamma^{m+1})}
\newcommand{\Vh}{\underline{V}(\Gamma^m)}

\newcommand{\Wh}{W(\Gamma^m)}
\newcommand{\Vht}{\underline{V}(\Gamma^h(t))}
\newcommand{\Wht}{W(\Gamma^h(t))}
\newcommand{\Whtz}{W(\Gamma^h(0))}
\newcommand{\uspace}{\mathbb{U}}
\newcommand{\utimespace}{\mathbb{V}}
\newcommand{\pspace}{\mathbb{P}}
\newcommand{\psispace}{\mathbb{S}}
\newcommand{\sigmaO}{o}
\newcommand{\nabs}{\nabla_{\!s}}
\newcommand{\Id}{I\!d}
\newcommand{\id}{\rm id}
\newcommand{\ddt}{\frac{\rm d}{{\rm d}t}}

\newcommand{\matpartu}{\partial_t^\bullet}
\newcommand{\matpartx}{\partial_t^\circ}
\newcommand{\matpartxh}{\partial_t^{\circ,h}}

\newcommand{\LerrorPsipsi}{\|\Psi - \psi\|_{L^2}}

\newcommand{\unitn}{\vec{\rm n}}
\newcommand{\unitt}{\vec{\rm t}}
\newcommand{\ek}{e}
\newcommand{\strikec}{c\!\!\!\!\:/}
\newcommand{\strikes}{s\!\!\!\!\:/}
\newcommand{\XFEMGAMMA}{XFEM$_\Gamma$}
\def\epsilon{\varepsilon} 
\newcommand{\mat}[1]{\underline{\underline{#1}}\rule{0pt}{0pt}}

\def\vL{L\kern-0.08cm\char39}
\begin{document}
\title{
On the Stable Numerical Approximation of 
Two-Phase Flow with Insoluble Surfactant}
\author{John W. Barrett\footnotemark[2] \and 
        Harald Garcke\footnotemark[3]\ \and 
        Robert N\"urnberg\footnotemark[2]}

\renewcommand{\thefootnote}{\fnsymbol{footnote}}
\footnotetext[2]{Department of Mathematics, 
Imperial College London, London, SW7 2AZ, UK}
\footnotetext[3]{Fakult{\"a}t f{\"u}r Mathematik, Universit{\"a}t Regensburg, 
93040 Regensburg, Germany}

\date{}

\maketitle

\begin{abstract}
We present a parametric finite element approximation of two-phase flow
with insoluble surfactant. This
free boundary problem is given by the Navier--Stokes equations 
for the two-phase flow in the bulk, which are coupled to the
transport equation for the insoluble surfactant on the interface that separates
the two phases. We combine the evolving surface finite element method
with an approach previously introduced by the authors for two-phase
Navier--Stokes flow, which maintains good mesh properties.
The derived finite element approximation of
two-phase flow with insoluble surfactant can be shown to be 
stable. Several numerical simulations demonstrate the 
practicality of our numerical method.
\end{abstract} 


\renewcommand{\thefootnote}{\arabic{footnote}}

\section{Introduction} \label{sec:intro}

The presence of surface active agents (surfactants) has a noticeable
effect on the deformation of fluid-fluid interfaces, because these
impurities lower the surface tension. In addition, surfactant
gradients along the fluid-fluid interface cause tangential stresses
leading to fluid motion (the Marangoni effect). As a result, the presence
of surfactants can have a dramatic effect on droplet shapes during their
evolution. Surfactants are applied in a wide range of
technologies to increase the efficiency of wetting agents, detergents,
foams and emulsion stabilisers.

In this paper we study the effect of an insoluble surfactant in a two-phase 
flow. The mathematical model consists of the Navier--Stokes
equations in the two phases, together with jump conditions at the free
boundary separating the two phases. In particular, the Laplace--Young
condition has to hold, which is a force balance involving forces
resulting from the two fluids. These forces are expressed with the help of
the stress tensor as well as surface tension forces and tangential
Marangoni forces, where the latter two involve the surfactant
concentration. The insoluble surfactant is transported on the
interface by advection and possibly by diffusion. The overall system
is quite complex, as a free boundary problem for the
Navier--Stokes equations and an advection-diffusion equation on the
evolving interface have to be solved simultaneously.

The mathematical analysis for the two-phase fluid flow problem with
surfactants is still in its early stages. 
We refer to \cite{GarckeW06}, who
showed a dissipation inequality for free surface flow with an
insoluble surfactant, 
and to \cite{BothePS05,BotheKP12preprint,BotheP10a},
where well-posedness and
stability of equilibria for two-phase flows with soluble
surfactants was shown. 
In particular, in \cite{BotheP10a} an energy inequality
was crucial in order to study the stability of equilibria. 
In this paper, it is our aim to develop a numerical
method that 
fulfills a discrete variant of this energy
inequality and, in addition, conserves the surfactant mass and
the volume of the two phases. Here we note that many of the existing
numerical methods for two-phase flow with insoluble surfactant may lose 
mass of one of the fluid phases, or may face stability issues.
In fact, to our knowledge, the numerical method presented in this paper is the
first approximation of two-phase flow with insoluble surfactant in the
literature that can be shown to satisfy a discrete energy law.

Different interface capturing and interface tracking methods have been
used to numerically compute two-phase flows with (in-)soluble
surfactants. Popular such approaches are volume of fluid methods,
\cite{RenardyRC02,JamesL04,Drumright-ClarkeR04,AlkeB09}; level set
methods, \cite{XuLLZ06,TeigenM10,GrossR11,XuYL12}; front tracking methods,
\cite{MuradogluT08,LaiTH08,KhatriT11,XuHLL14} and
arbitrary Lagrangian-Eulerian methods, \cite{Pozrikidis04,YangJ07,GanesanT09a}.
Another approach to model and numerically simulate two-phase fluids
involving surfactants involves diffuse interface approaches and we
refer to \cite{TeigenLLWV09,ElliottSSW11,GarckeLS13preprint} and
\cite{EngblomD-QAT13} for details.

In this work we use parametric finite elements to describe the fluid-fluid
interface 
with an unfitted coupling to the fluid flow in the bulk, 
which is also discretized with the help of finite elements. 
Unfitted in this context means that the mesh points used to describe the 
interface are not, in general, mesh points of the underlying bulk finite element mesh. 
Our approach is based on
earlier work by the authors on two-phase flow for incompressible
Stokes and Navier--Stokes flow involving surface tension effects, see
\cite{spurious,fluidfbp} for details. As mentioned above, apart from 
capturing the interface in a two-phase flow, one also has to accurately 
capture the advection and diffusion of the surfactant on the interface. 
Here we make
use of a variant of the evolving surface finite element method (ESFEM)
introduced by \cite{DziukE07,DziukE13}. In order to accurately discretize the
advection-diffusion equation on the evolving interface, it
is important to evolve the grid points representing the interface in
such a way, that the mesh does not degenerate. 
In particular, 
it is important to avoid the coalescence of vertices or a velocity induced 
coarsening at parts of the interface, see e.g.\ Figures~\ref{fig:dziuk} 
and \ref{fig:dziuknoni} in Section~\ref{sec:6}. 
It turns out that
moving vertices with the fluid velocity or with the normal part of the
fluid velocity typically leads to mesh degeneracies. Hence in this paper we
follow the approach from \cite{fluidfbp} and allow the grid points to have a
tangential velocity that is independent of the surrounding fluid motion.
We note that the idea to allow for an implicit, nonzero discrete tangential
velocity goes back to earlier work by the present authors,
who introduced novel numerical methods with excellent mesh properties
for curvature driven flows and moving boundary problems in e.g.\
\cite{triplej,gflows3d,dendritic}. In fact, we are able
to show that our semidiscrete continuous-in-time finite element approximations
lead to {\it equidistributed mesh points} on the interface in two space
dimensions, and to {\it conformal polyhedral surfaces},
which also have good mesh properties,
in three space dimensions.
Using this
approach also ensures that, due to the good mesh properties, the surface
partial differential equation for the insoluble surfactant can be
solved accurately. 

An important issue in surface tension driven flows is to compute
curvature quantities with the help of the chosen interface
representation. Our approach uses a parametric approximation of the
interface, and hence we use a variant of an idea by Dziuk to compute
the mean curvature. In fact \cite{Dziuk91} uses the identity 
\begin{equation} \label{eq:GDid}
\Delta_s\,\vec\id = \vec\varkappa\,,
\end{equation}
where $\Delta_s$ is the Laplace--Beltrami
operator and $\vec\varkappa$ is the mean curvature vector, in a discrete
setting to compute an approximation of the mean curvature. This idea
was used by \cite{Bansch01} for an approximation of free capillary flows, 
and by \cite{BaumlerB13} for two-phase flows.
A discretization of a variant of (\ref{eq:GDid}) was used by the 
present authors in \cite{spurious,fluidfbp}
to derive approximations of two-phase flow with better mesh properties. 
As mentioned above, this approach leads to tangential
motions for the mesh points on the interface that are independent of the fluid
motion. This has to be taken into account when solving the advection-diffusion
equation on the interface, and in our case we naturally obtain the so-called 
arbitrary Lagrangian Eulerian evolving surface finite element
method (ALE ESFEM), see \cite{ElliottS12}.

The structure of this article is as follows. In the next section we first state the
mathematical formulation of the problem and discuss the relevant
conserved quantities and an energy identity. In addition, different
weak formulations are introduced which form the basis for the finite
element approximations in Section~\ref{sec:3}. We state two different
finite element approximations in a semidiscrete and in a fully
discrete form. The first method uses the curvature discretization
of \cite{Dziuk91} and \cite{Bansch01}, while the second method
uses the curvature discretization introduced by the present authors in 
\cite{triplej,gflows3d,fluidfbp}. 

Both methods, in their semidiscrete form, conserve the total
surfactant concentration and allow for an energy inequality in two
space dimensions. In addition, the variant based on Dziuk's curvature
discretization allows for a discrete maximum principle for the surfactant
approximation. On the other hand, the approach that uses
the curvature discretization of the present authors leads to {\it good mesh
properties} and to {\it exactly conserved volumes} of the two fluids. For the
fully discrete approximations {\it existence and uniqueness} as well as
{\it conservation of the total surfactant concentration} can be
shown. Finally we present several numerical simulations in two and
three space dimensions in Section~\ref{sec:6}, which in particular show
the effect of surfactants on the interface evolution.

\section{Mathematical formulation} \label{sec:1}

\subsection{Governing equations} \label{sec:11}

Let $\Domain\subset\mathbb{R}^d$ be a given domain,  where $d=2$ or $d=3$. 
We now seek a time dependent interface $(\Gamma(t))_{t\in[0,T]}$,
$\Gamma(t)\subset\Domain$, 
which for all $t\in[0,T]$ separates
$\Domain$ into a domain $\Omega_+(t)$, occupied by one phase,
and a domain $\Omega_{-}(t):=\Domain\setminus\overline{\Omega_+(t)}$, 
which is occupied by the other phase. Here the phases could represent
two different liquids, or a liquid and a gas. Common examples are oil/water
or water/air interfaces.
See Figure~\ref{fig:sketch} for an illustration.
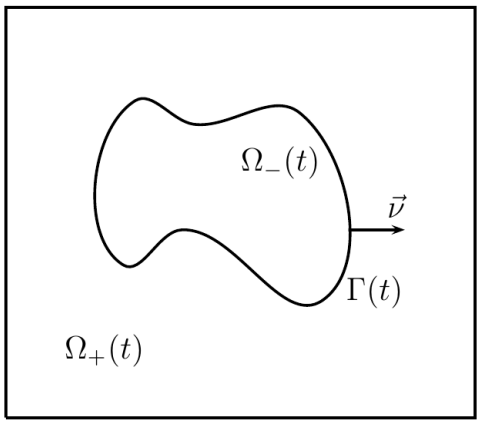
\begin{figure}
\begin{center}
\ifpdf
\includegraphics[angle=0,width=0.4\textwidth]{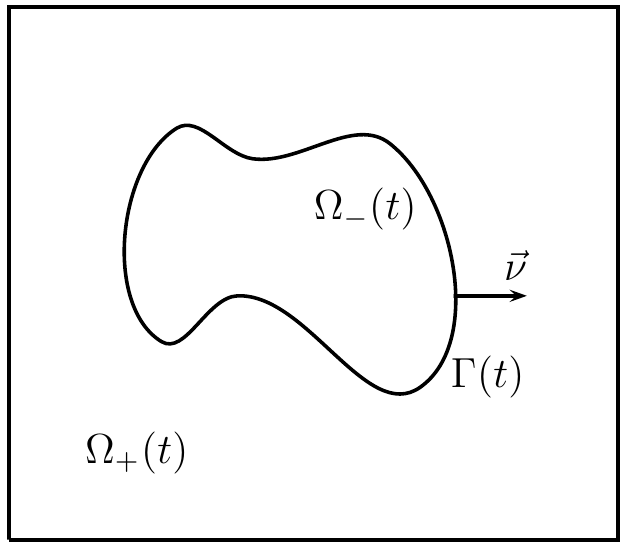}
\else
\unitlength15mm
\psset{unit=\unitlength,linewidth=1pt}
\begin{picture}(4,3.5)(0,0)
\psline[linestyle=solid]{-}(0,0)(4,0)(4,3.5)(0,3.5)(0,0)
\psccurve[showpoints=false,linestyle=solid] 
 (1,1.3)(1.5,1.6)(2.7,1.0)(2.5,2.6)(1.6,2.5)(1.1,2.7)
\psline[linestyle=solid]{->}(2.92,1.6)(3.4,1.6)
\put(3.25,1.7){{\black $\vec\nu$}}
\put(2.9,1.0){{$\Gamma(t)$}}
\put(2,2.1){{$\Omega_-(t)$}}
\put(0.5,0.5){{$\Omega_+(t)$}}
\end{picture}
\fi
\end{center}
\caption{The domain $\Omega$ in the case $d=2$.}
\label{fig:sketch}
\end{figure}%
For later use, we assume that
$(\Gamma(t))_{t\in [0,T]}$ 
is a sufficiently smooth evolving hypersurface without boundary that is
parameterized by $\vec{x}(\cdot,t):\Upsilon\to\R^d$,
where $\Upsilon\subset \R^d$ is a given reference manifold, i.e.\
$\Gamma(t) = \vec{x}(\Upsilon,t)$. Then
\begin{equation} \label{eq:V}
\vec{\mathcal{V}}(\vec z, t) := \vec x_t(\vec q, t)
\qquad \forall\ \vec z = \vec x(\vec q,t) \in \Gamma(t)
\end{equation}
defines the velocity of $\Gamma(t)$, and
$\vec{\mathcal{V}} \,.\,\vec{\nu}$ is
the normal velocity of the evolving hypersurface $\Gamma(t)$,
where $\vec\nu(t)$ is the unit normal on $\Gamma(t)$ pointing into 
$\Omega_+(t)$.
Moreover, we define the space-time surface
\begin{equation} \label{eq:GT}
\GT := \bigcup_{t \in [0,T]} \Gamma(t) \times \{t\}\,.
\end{equation}

Let $\rho(t) = \rho_+\,\charfcn{\Omega_+(t)} + \rho_-\,\charfcn{\Omega_-(t)}$,
with $\rho_\pm \in \R_{>0}$, 
denote the fluid densities, where here and
throughout $\charfcn{\mathcal{A}}$ defines the characteristic function for a
set $\mathcal{A}$.
Denoting by $\vec u : \Omega \times [0, T] \to \R^d$ the fluid velocity,
by $\mat\sigma : \Omega \times [0,T] \to \R^{d \times d}$ the stress tensor,
and by $\vec f : \Omega \times [0, T] \to \R^d$ a possible forcing,
the incompressible Navier--Stokes equations in the two phases are given by
\begin{subequations}
\begin{alignat}{2}
\rho\,(\vec u_t + (\vec u \,.\,\nabla)\,\vec u)
- \nabla\,.\,\mat\sigma & = \vec f := \rho\,\vec f_1 + \vec f_2
\qquad &&\mbox{in } 
\Omega_\pm(t)\,, \label{eq:NSa} \\
\nabla\,.\,\vec u & = 0 \qquad &&\mbox{in } \Omega_\pm(t)\,, \label{eq:NSb} \\
\vec u & = \vec 0 \qquad &&\mbox{on } \partial_1\Omega\,, \label{eq:NSc} \\
\vec u \,.\,\unitn = 0\,,\quad
\mat\sigma\,\unitn \,.\,\unitt& = 0 
\quad\forall\ \unitt \in \{\unitn\}^\perp 
\qquad &&\mbox{on } \partial_2\Omega\,, 
\label{eq:NSd} 
\end{alignat}
\end{subequations}
where $\partial\Domain = \partial_1\Omega \cup
\partial_2\Omega$, with $\partial_1\Omega \cap \partial_2\Omega
=\emptyset$, denotes the boundary of $\Domain$ with outer unit normal $\unitn$
and $\{\unitn\}^\perp := \{ \unitt \in \R^d : \unitt \,.\,\unitn = 0\}$.
Hence (\ref{eq:NSc}) prescribes a no-slip condition on 
$\partial_1\Omega$, while (\ref{eq:NSd}) prescribes a free-slip condition on 
$\partial_2\Omega$. 
In addition, the stress tensor in (\ref{eq:NSa}) is defined by
\begin{equation} \label{eq:sigma}
\mat\sigma = \mu \,(\nabla\,\vec u + (\nabla\,\vec u)^T) - p\,\mat\Id
= 2\,\mu\, \mat D(\vec u)-p\,\mat\Id\,,
\end{equation}
where $\mat\Id \in \R^{d \times d}$ denotes the identity matrix,
$\mat D(\vec u):=\frac12\, (\nabla\vec u+(\nabla\vec u)^T)$ is the
rate-of-deformation tensor,
$p : \Omega \times [0, T] \to \R$ is the pressure and
$\mu(t) = \mu_+\,\charfcn{\Omega_+(t)} + \mu_-\,\charfcn{\Omega_-(t)}$,
with $\mu_\pm \in \R_{>0}$, denotes the dynamic viscosities in the two 
phases.
On the free surface $\Gamma(t)$, the following conditions need to hold:
\begin{subequations}
\begin{alignat}{2}
[\vec u]_-^+ & = \vec 0 \qquad &&\mbox{on } \Gamma(t)\,, \label{eq:1a} \\ 
[\mat\sigma\,\vec \nu]_-^+ & = -\gamma(\psi)\,\varkappa\,\vec\nu 
- \nabs\,\gamma(\psi)
\qquad &&\mbox{on } \Gamma(t)\,, \label{eq:1b} \\ 
\vec{\mathcal{V}}\,.\,\vec\nu &= \vec u\,.\,\vec \nu 
\qquad &&\mbox{on } \Gamma(t)\,, \label{eq:1c} 
\end{alignat}
\end{subequations}
where $\gamma \in C^{1}([0,\psi_\infty))$, with $\psi_\infty \in
\R_{>0}\cup\{\infty\}$ and
\begin{equation} \label{eq:gammaprime}
\gamma'(r) \leq 0 \qquad \forall\ r \in [0,\psi_\infty)\,,
\end{equation}
denotes the surface tension which depends on 
the surfactant concentration $\psi : \GT \to [0,\psi_\infty)$, 
recall (\ref{eq:GT}), 
and $\nabs$ denotes the surface gradient on $\Gamma(t)$.
In addition, 
$\varkappa$ denotes the mean curvature of $\Gamma(t)$, i.e.\ the sum of
the principal curvatures of $\Gamma(t)$, where we have adopted the sign
convention that $\varkappa$ is negative where $\Omega_-(t)$ is locally convex.
In particular, on letting $\vec \id$ denote the identity function in $\R^d$,
it holds that
\begin{equation} \label{eq:LBop}
\Delta_s\, \vec \id = \varkappa\, \vec\nu =: \vec\varkappa
\qquad \mbox{on $\Gamma(t)$}\,,
\end{equation}
where $\Delta_s = \nabs\,.\,\nabs$ is the Laplace--Beltrami operator on 
$\Gamma(t)$, with $\nabs\,.\,$ denoting surface divergence on $\Gamma(t)$.
Moreover, as usual, $[\vec u]_-^+ := \vec u_+ - \vec u_-$ and
$[\mat\sigma\,\vec\nu]_-^+ := \mat\sigma_+\,\vec\nu - \mat\sigma_-\,\vec\nu$
denote the jumps in velocity and normal stress across the interface
$\Gamma(t)$. Here and throughout, we employ the shorthand notation
$\vec g_\pm := \vec g\!\mid_{\Omega_\pm(t)}$ for a function 
$\vec g : \Omega \times [0,T] \to \R^d$; and similarly for scalar and
matrix-valued functions. The surfactant transport (with diffusion) on
$\Gamma(t)$ is then given by
\begin{equation} \label{eq:1surf}
\matpartu\,\psi + \psi\,\nabs\,.\vec u - \nabs\,.\,(\Ds\,\nabs\,\psi) = 0
\qquad \mbox{on } \Gamma(t)\,,
\end{equation}
where $\Ds \geq 0$ is a diffusion coefficient, 
and where
\begin{equation} \label{eq:matpartu}
\matpartu\, \zeta = \zeta_t + \vec u \,.\,\nabla\,\zeta
\qquad \forall\ \zeta \in H^1(\GT)
\end{equation}
denotes the material time derivative of $\zeta$ on $\Gamma(t)$. Here we stress
that the derivative in (\ref{eq:matpartu}) is well-defined, and depends only on
the values of $\zeta$ on $\GT$, even though $\zeta_t$ and
$\nabla\,\zeta$ do not make sense separately; see e.g.\
\citet[p.\ 324]{DziukE13}. 
The system (\ref{eq:NSa}--d), (\ref{eq:sigma}), (\ref{eq:1a}--c),
(\ref{eq:1surf}) is closed with the initial conditions
\begin{equation} \label{eq:1d}
\Gamma(0) = \Gamma_0 \,, \qquad 
\psi(\cdot,0) = \psi_0 \qquad \mbox{on } \Gamma_0\,,\qquad
\vec u(\cdot,0) = 
\vec u_0 \qquad \mbox{in } \Omega\,,
\end{equation}
where $\Gamma_0 \subset \Omega$,
$\vec u_0 : \Omega \to \R^d$ 
and $\psi_0 : \Gamma_0 \to [0,\psi_\infty)$
are given initial data. 

For later purposes, we introduce the 
surface energy function $F$, which satisfies
\begin{subequations}
\begin{equation} \label{eq:F}
\gamma(r) = F(r) - r\,F'(r) \qquad \forall\ r \in (0,\psi_\infty) \,,
\end{equation}
and
\begin{equation} \label{eq:F0}
\lim_{r\to0} r\,F'(r) = F(0) - \gamma(0) = 0\,.
\end{equation}
\end{subequations}
This means in particular that 
\begin{equation} \label{eq:Fdd}
\gamma'(r) = - r\,F''(r) \qquad \forall\ r \in (0,\psi_\infty)\,. 
\end{equation}
It immediately follows from (\ref{eq:Fdd}) and (\ref{eq:gammaprime}) that
$F \in C([0,\psi_\infty)) \cap C^2(0,\psi_\infty)$ is convex.
Typical examples for $\gamma$ and $F$ are given by
\begin{subequations}
\begin{equation} \label{eq:gamma1}
\gamma(r) = \gamma_0 \,( 1 - \beta\,r) \,,\quad
F(r) = \gamma_0\left[1+\beta\,r\left(\ln r 
- 1 \right) \right]\,,\ \psi_\infty = \infty\,,
\end{equation}
which represents a linear equation of state, and by
\begin{equation} \label{eq:gamma2}
\gamma(r) = \gamma_0 \left[ 1 + \beta\,\psi_\infty\,\ln \left( 
1 - \tfrac{r}{\psi_\infty} \right) \right] , \quad
F(r) = \gamma_0\left[ 1 + 
\beta\left( r\,\ln \tfrac{r}{\psi_\infty-r} + \psi_\infty\,\ln
\tfrac{\psi_\infty-r}{\psi_\infty} \right) \right],
\end{equation}
\end{subequations}
the so-called Langmuir equation of state,
where $\gamma_0 \in \R_{>0}$ and $\beta \in \Rgeq$ are further given 
parameters, where we note that the special case $\beta = 0$ means that
(\ref{eq:gamma1},b) reduce to
\begin{equation} \label{eq:Fconst}
F(r) = \gamma(r) = \gamma_0 \in \R_{>0}\qquad \forall\ r \in \R\,.
\end{equation}
Moreover, we observe that (\ref{eq:gamma1}) can be viewed as a linearization of
(\ref{eq:gamma2}) in the sense that $\gamma$ in (\ref{eq:gamma1}) is 
affine, and $\gamma$ and $\gamma'$ agree at the origin with 
$\gamma$ and $\gamma'$ from (\ref{eq:gamma2}).

\subsection{Weak formulation}\label{sec:12}

Before introducing our finite element approximation, 
we will state an appropriate weak formulation. With this in mind, 
we introduce the function spaces
\begin{align*}
\uspace & := \{ \vec\varphi \in [H^1(\Omega)]^d : 
\vec\varphi = \vec0 \ \mbox{ on } 
\partial_1\Omega\,,\ \vec\varphi \,.\,\unitn = 0 \ \mbox{a.e. on } \partial_2\Omega
 \} \,,
\quad \pspace := L^2(\Omega)\,, \\ 
\widehat\pspace & := \{\eta \in \pspace : \int_\Omega\eta \dL{d}=0 \}\,,\quad
\utimespace :=  L^2(0,T; \uspace) \cap H^1(0,T;[L^2(\Omega)]^d)\,,\quad
\psispace := H^1(\GT)\,.
\end{align*}
Let $(\cdot,\cdot)$ and $\langle \cdot, \cdot \rangle_{\Gamma(t)}$
denote the $L^2$--inner products on $\Omega$ and $\Gamma(t)$, respectively.
We recall from \cite{fluidfbp} that it follows from (\ref{eq:NSb}--d) and
(\ref{eq:1c}) that
\begin{align}
( \rho\,(\vec u \,.\,\nabla)\,\vec u, \vec \xi)
&= \tfrac12 \left[ (\rho\,(\vec u\,.\,\nabla)\,\vec u, \vec \xi) -
(\rho\,(\vec u\,.\,\nabla)\,\vec \xi,\vec u)
-\left\langle [\rho]_-^+\,\vec u\,.\,\vec \nu, 
  \vec u\,.\,\vec \xi \right\rangle_{\Gamma(t)} \right]
\nonumber \\ & \hspace{9cm} 
\forall\ \vec \xi \in [H^1(\Omega)]^d
\label{eq:advect}
\end{align}
and
\begin{align*} 
\ddt(\rho \,\vec u, \vec \xi) 
&= (\rho\,\vec u_t, \vec \xi)
+ ( \rho\,\vec u, \vec \xi_t)
- \left\langle [\rho]_-^+\,\vec u\,.\,\vec \nu, 
\vec u \,.\,\vec \xi \right\rangle_{\Gamma(t)}
\qquad \forall \ 
\vec \xi \in \utimespace\,, 
\end{align*}
respectively.
Therefore, it holds that
\begin{equation*} 
(\rho\,\vec u_t, \vec \xi) = 
\tfrac{1}{2} \left[
\ddt (\rho\,\vec u,\vec \xi) + (\rho\,\vec u_t, \vec \xi)
- ( \rho\,\vec u, \vec \xi_t)
+ \left\langle [\rho]_-^+\,\vec u\,.\,\vec \nu, 
\vec u \,.\,\vec \xi \right\rangle_{\Gamma(t)} 
\right]
\qquad \forall\ \vec \xi \in \utimespace\,,
\end{equation*}
which on combining with (\ref{eq:advect}) yields that
\begin{align} \label{eq:rhot3}
& (\rho\,[\vec u_t + (\vec u \,.\,\nabla)\,\vec u], \vec \xi)
\nonumber \\ & \qquad
= \tfrac12\left[ \ddt (\rho\,\vec u, \vec \xi) + (\rho\,\vec u_t, \vec \xi)
- ( \rho\,\vec u, \vec \xi_t)
+ (\rho, [(\vec u\,.\,\nabla)\,\vec u]\,.\,\vec \xi
- [(\vec u\,.\,\nabla)\,\vec \xi]\,.\,\vec u) \right]
\quad \forall\ \vec \xi \in \utimespace\,.
\end{align}
Moreover, it holds on noting (\ref{eq:NSd}) and (\ref{eq:1b}) 
that for all $\vec \xi \in \uspace$ 
\begin{align}
& \int_{\Omega_+(t)\cup\Omega_-(t)} (\nabla\,.\,\mat\sigma)\,.\, \vec \xi 
\dL{d} 
= - 2\,(\mu\,\mat D(\vec u), \mat D(\vec \xi)) + (p, \nabla\,.\,\vec \xi)
+ \left\langle \gamma(\psi)\,\varkappa\,\vec \nu + \nabs\,\gamma(\psi), 
\vec \xi \right\rangle_{\Gamma(t)}.
\label{eq:sigmaibp}
\end{align}

Similarly to (\ref{eq:matpartu}) we define the following time derivative that
follows the parameterization $\vec x(\cdot, t)$ of $\Gamma(t)$, rather than
$\vec u$. In particular, we let
\begin{equation} \label{eq:matpartx}
\matpartx\, \zeta = \zeta_t + \vec{\mathcal{V}} \,.\,\nabla\,\zeta
\qquad \forall\ \zeta \in \psispace\,,
\end{equation}
recall (\ref{eq:V}). 
Here we stress once again that this definition is well-defined, even though
$\zeta_t$ and $\nabla\,\zeta$ do not make sense separately for a
function $\zeta \in \psispace$.
On recalling (\ref{eq:matpartu}) we obtain that
\begin{equation} \label{eq:matpartux}
\matpartx = \matpartu  \qquad\text{if}\qquad \vec{\mathcal{V}} = \vec u
\quad \text{on } \Gamma(t)\,.
\end{equation}
We note that the definition (\ref{eq:matpartx}) 
differs from the definition of $\partial^\circ$ in 
\citet[p.\ 327]{DziukE13}, where
$\partial^\circ\,\zeta = \zeta_t + (\vec{\mathcal{V}}\,.\,\vec \nu)\, 
\vec\nu\,.\,\nabla\,\zeta$
for the ``normal time derivative''. 
It holds that
\begin{equation} \label{eq:DElem5.2}
\ddt \left\langle \chi, \zeta \right\rangle_{\Gamma(t)}
 = \left\langle \matpartx\,\chi, \zeta \right\rangle_{\Gamma(t)}
 + \left\langle \chi, \matpartx\,\zeta \right\rangle_{\Gamma(t)}
+ \left\langle \chi\,\zeta, \nabs\,.\,\vec{\mathcal{V}} \right\rangle_{\Gamma(t)}
\qquad \forall\ \chi,\zeta \in \psispace\,,
\end{equation}
see \citet[Lem.\ 5.2]{DziukE13}, and that
\begin{equation} \label{eq:DEdef2.11}
\left\langle \zeta, \nabs\,.\, \vec\eta \right\rangle_{\Gamma(t)}
+ \left\langle \nabs\,\zeta , \vec\eta \right\rangle_{\Gamma(t)}
= - \left\langle \zeta\,\vec\eta , \vec\varkappa \right\rangle_{\Gamma(t)} 
\qquad \forall\ \zeta \in H^1(\Gamma(t)),\, \vec \eta \in [H^1(\Gamma(t))]^d\,,
\end{equation}
see \citet[Def.\ 2.11]{DziukE13}.
For later use
we remark that it follows from (\ref{eq:DEdef2.11}) that
\begin{equation} \label{eq:newGD}
\left\langle \gamma(\psi)\,\vec\varkappa + \nabs\,\gamma(\psi), 
\vec \xi \right\rangle_{\Gamma(t)} 
=
\left\langle \gamma(\psi)\,\varkappa\,\vec \nu + \nabs\,\gamma(\psi), 
\vec \xi \right\rangle_{\Gamma(t)}
=
- \left\langle \gamma(\psi), \nabs\,.\,\vec \xi \right\rangle_{\Gamma(t)}
\quad \forall\ \vec\xi \in \uspace\,.
\end{equation}

The natural weak formulation of
the system (\ref{eq:NSa}--d), (\ref{eq:sigma}), (\ref{eq:1a}--c),
(\ref{eq:1surf}) is then given as follows.
Find $\Gamma(t) = \vec x(\Upsilon, t)$ for $t\in[0,T]$ 
with $\vec{\mathcal{V}} \in L^2(0,T;[H^1(\Gamma(t))]^d)$, 
and functions $\vec u \in \utimespace$, 
$p \in L^2(0,T; \widehat\pspace)$, 
$\vec\varkappa \in L^2(0,T; [L^2(\Gamma(t))]^d)$ and
$\psi \in \psispace$ such that 
for almost all $t \in (0,T)$ it holds that
\begin{subequations}
\begin{align}
& \tfrac12\left[ \ddt (\rho\,\vec u, \vec \xi) + (\rho\,\vec u_t, \vec \xi)
- (\rho\,\vec u, \vec \xi_t)
+ (\rho, [(\vec u\,.\,\nabla)\,\vec u]\,.\,\vec \xi
- [(\vec u\,.\,\nabla)\,\vec \xi]\,.\,\vec u) \right] 
\nonumber \\ & \qquad
+ 2\,(\mu\,\mat D(\vec u), \mat D(\vec \xi)) 
- (p, \nabla\,.\,\vec \xi)
- \left\langle \gamma(\psi)\,\vec\varkappa + \nabs\,\gamma(\psi), 
  \vec \xi \right\rangle_{\Gamma(t)}
 = (\vec f, \vec \xi) 
\qquad \forall\ \vec \xi \in \utimespace\,,
\label{eq:weakGDa} \\
& (\nabla\,.\,\vec u, \varphi) = 0  \qquad \forall\ \varphi \in 
\widehat\pspace\,, \label{eq:weakGDb} \\
& \left\langle \vec{\mathcal{V}} - \vec u, \vec\chi \right\rangle_{\Gamma(t)} 
 = 0  \qquad\forall\ \vec\chi \in [L^2(\Gamma(t))]^d\,,
\label{eq:weakGDc} \\
& \left\langle \vec\varkappa, \vec\eta \right\rangle_{\Gamma(t)}
+ \left\langle \nabs\,\vec \id, \nabs\,\vec \eta \right\rangle_{\Gamma(t)}
 = 0  \qquad\forall\ \vec\eta \in [H^1(\Gamma(t))]^d\,, \label{eq:weakGDd} \\
& \ddt\left\langle \psi, \zeta \right\rangle_{\Gamma(t)}
+ \Ds \left\langle \nabs\,\psi, \nabs\,\zeta \right\rangle_{\Gamma(t)}
= \left\langle \psi, \matpartx\, \zeta \right\rangle_{\Gamma(t)}
\quad\forall\ \zeta \in \psispace\,,
\label{eq:weakGDe}
\end{align}
\end{subequations}
as well as the initial conditions (\ref{eq:1d}), where in (\ref{eq:weakGDc}) 
we have recalled (\ref{eq:V}). 
Here (\ref{eq:weakGDa}--d) can be derived analogously to the weak formulation
presented in \cite{fluidfbp}, recall (\ref{eq:rhot3}) and (\ref{eq:sigmaibp}), 
while (\ref{eq:weakGDe}) is a direct consequence
of (\ref{eq:DElem5.2}) 
and (\ref{eq:DEdef2.11}); see \cite{DziukE13}. Of course, it follows from
(\ref{eq:weakGDc}) and (\ref{eq:matpartux}) that $\matpartx$ in
(\ref{eq:weakGDe}) can be replaced by $\matpartu$.

\begin{rem} \label{rem:Stokes}
For ease of presentation, in this paper we restrict ourselves to the case of
two-phase Navier--Stokes flow, i.e.\ $\rho_\pm > 0$. However, it is a simple
matter to generalize the results in this paper to two-phase Stokes flow in the
bulk, i.e.\ to $\rho_+ = \rho_- = 0$. For example, the weak formulation
{\rm (\ref{eq:weakGDa}--e)} then holds with
$\rho=0$ and with $\utimespace$ replaced by $L^2(0,T;\uspace)$; and analogous
simplifications can be applied to the finite element approximations that will
be introduced later in this paper, see also \cite{spurious}.
For example, the presented fully discrete schemes in \S\ref{sec:32} are 
valid for arbitrary choices of $\rho_\pm \geq0$.
\end{rem}

\subsection{Energy bounds}\label{sec:13}

In what follows
we would like to derive an energy bound for a solution of
(\ref{eq:weakGDa}--e). All of the following considerations are formal, in the
sense that we make the appropriate assumptions about the existence, 
boundedness and regularity of a solution to \mbox{(\ref{eq:weakGDa}--e)}. 
In particular, we assume that $\psi \in [0,\psi_\infty)$.
Choosing $\vec\xi = \vec u$ in (\ref{eq:weakGDa}) and
$\varphi = p(\cdot, t)$ in (\ref{eq:weakGDb}) yields that
\begin{equation} \label{eq:d1}
\tfrac12\,\ddt \,\|\rho^\frac12\,\vec u \|_0^2 +
2\,\| \mu^\frac12\,\mat D(\vec u)\|_0^2 = (\vec f, \vec u)
+ \left\langle \gamma(\psi) \,\vec\varkappa + \nabs\,\gamma (\psi) , \vec u
\right\rangle_{\Gamma(t)}\,.
\end{equation}
In what follows, assuming that $\gamma$ is not constant, 
recall (\ref{eq:Fconst}), we would like to choose $\zeta = F'(\psi)$ in 
(\ref{eq:weakGDe}). As $F'$ in general is singular at the origin, 
recall (\ref{eq:Fdd}), we instead
choose $\zeta = F'(\psi + \alpha)$ for some $\alpha \in \R_{>0}$ with
$\psi + \alpha < \psi_\infty$.
Then we obtain, on recalling (\ref{eq:F}) and (\ref{eq:DElem5.2}), that
\begin{align} \label{eq:d4}
& \ddt \left\langle F(\psi + \alpha) - \gamma(\psi + \alpha), 1 
\right\rangle_{\Gamma(t)}
+ \Ds \left\langle \nabs\,(\psi + \alpha), 
\nabs\,F'(\psi + \alpha) \right\rangle_{\Gamma(t)} \nonumber \\ &
\qquad\qquad
= \left\langle \psi + \alpha, \matpartx\,F'(\psi + \alpha) \right\rangle_{\Gamma(t)}
+ \alpha  \left\langle F'(\psi + \alpha) , \nabs\,.\,\vec{\mathcal{V}} \right\rangle_{\Gamma(t)}
\,.
\end{align}
Moreover, choosing $\chi = \gamma(\psi+\alpha)$, $\zeta = 1$ in
(\ref{eq:DElem5.2}), 
and then choosing $\vec\eta = \vec{\mathcal{V}}$, $\zeta = \gamma(\psi+\alpha)$ in
(\ref{eq:DEdef2.11}) leads to
\begin{align} \label{eq:d2}
\ddt \left\langle \gamma(\psi+\alpha), 1 \right\rangle_{\Gamma(t)}
& = \left\langle \matpartx\,\gamma(\psi+\alpha), 1 \right\rangle_{\Gamma(t)}
+ \left\langle \gamma(\psi+\alpha), \nabs\,.\,\vec{\mathcal{V}} \right\rangle_{\Gamma(t)} 
\nonumber \\ &
= \left\langle \matpartx\,\gamma(\psi+\alpha), 1 \right\rangle_{\Gamma(t)}
- \left\langle \gamma(\psi+\alpha) \,\vec\varkappa + 
 \nabs\,\gamma (\psi+\alpha) , \vec{\mathcal{V}} \right\rangle_{\Gamma(t)}\,.
\end{align}
In addition, it follows from (\ref{eq:Fdd}) that
\begin{equation} \label{eq:d3}
\matpartx\,\gamma(\psi+\alpha) = \gamma'(\psi+\alpha) \, \matpartx\,\psi
= - (\psi+\alpha)\,F''(\psi+\alpha)\,\matpartx\,\psi 
= - (\psi+\alpha)\,\matpartx\,F'(\psi+\alpha)\,.
\end{equation}
Combining (\ref{eq:d4}), (\ref{eq:d2}) and (\ref{eq:d3}) yields that
\begin{align} \label{eq:d123}
& \ddt \left\langle F(\psi + \alpha) , 1 \right\rangle_{\Gamma(t)}
+ \Ds \left\langle \nabs\,\mathcal{F}(\psi + \alpha), 
\nabs\,\mathcal{F}(\psi + \alpha) \right\rangle_{\Gamma(t)} \nonumber \\ &
\qquad\qquad
= - \left\langle \gamma(\psi+\alpha) \,\vec\varkappa + 
 \nabs\,\gamma (\psi+\alpha) , \vec{\mathcal{V}} \right\rangle_{\Gamma(t)}
+ \alpha  \left\langle F'(\psi + \alpha) , \nabs\,.\,\vec{\mathcal{V}} \right\rangle_{\Gamma(t)}
\,,
\end{align}
where, on recalling (\ref{eq:Fdd}) and (\ref{eq:gammaprime}),
\begin{equation*}  
\mathcal{F}(r) = \int_0^r [F''(y)]^{\frac12} \;{\rm d}y\,.
\end{equation*}
Letting $\alpha \to 0$ in (\ref{eq:d123}) yields, on recalling (\ref{eq:F0}),
that
\begin{equation}  \label{eq:d45}
 \ddt \left\langle F(\psi) , 1 \right\rangle_{\Gamma(t)}
+ \Ds \left\langle \nabs\,\mathcal{F}(\psi), 
\nabs\,\mathcal{F}(\psi) \right\rangle_{\Gamma(t)} 
= - \left\langle \gamma(\psi) \,\vec\varkappa + 
 \nabs\,\gamma (\psi) , \vec{\mathcal{V}} \right\rangle_{\Gamma(t)} \,.
\end{equation}
We note that (\ref{eq:d45}) is still valid in the case (\ref{eq:Fconst}), on
noting (\ref{eq:DElem5.2}) and (\ref{eq:newGD}). 
Combining (\ref{eq:d45}) with (\ref{eq:d1}) implies the a priori energy 
equation
\begin{equation} \label{eq:d5}
\ddt \left( \tfrac12\,\|\rho^\frac12\,\vec u \|_0^2 
+ \left\langle F(\psi), 1 \right\rangle_{\Gamma(t)} \right)
+ 2\,\| \mu^\frac12\,\mat D(\vec u)\|_0^2 
+ \Ds \left\langle \nabs\,\mathcal{F}(\psi), \nabs\,\mathcal{F}(\psi) 
\right\rangle_{\Gamma(t)}
= (\vec f, \vec u)\,.
\end{equation}
Moreover, the volume
of $\Omega_-(t)$ is preserved in time, i.e.\ the mass of each phase is 
conserved. To see this, choose $\vec\chi = \vec\nu$ in
(\ref{eq:weakGDc}) and $\varphi = \charfcn{\Omega_-(t)}$ in (\ref{eq:weakGDb})
to obtain
\begin{equation}
\frac{\rm d}{{\rm d}t} \vol(\Omega_-(t)) = 
\left\langle \vec{\mathcal{V}}, \vec\nu \right\rangle_{\Gamma(t)}
= \left\langle \vec u, \vec\nu \right\rangle_{\Gamma(t)}
= \int_{\Omega_-(t)} \nabla\,.\,\vec u \dL{d} 
=0\,. \label{eq:conserved}
\end{equation}
In addition, we note that it immediately follows from 
choosing $\zeta = 1$ in (\ref{eq:weakGDe}) that the 
total amount of surfactant is preserved, i.e.\
\begin{equation} \label{eq:totalpsi}
\ddt \int_{\Gamma(t)} \psi \dH{d-1} = 0 \,.
\end{equation}

\subsection{Alternative weak formulation}\label{sec:14}

It will turn out that another weak formulation of the overall system
(\ref{eq:NSa}--d), (\ref{eq:sigma}), \mbox{(\ref{eq:1a}--c)}, (\ref{eq:1surf}) 
will lead to finite
element approximations with better mesh properties. In order to derive
the weak formulation, and on recalling (\ref{eq:matpartux}),
we note that if we relax $\vec{\mathcal{V}} = \vec u \!\mid_{\Gamma(t)}$ to
\begin{equation*} 
\vec {\mathcal{V}}\,.\,\vec \nu = \vec u\,.\,\vec \nu
\quad \text{on } \Gamma(t)\,,
\end{equation*}
then it holds that
\begin{equation} \label{eq:matpartxu}
\matpartx\, \zeta = \matpartu\,\zeta 
+ (\vec {\mathcal{V}} - \vec u)\,.\,\nabs\,\zeta \qquad \forall\ \zeta \in 
\psispace\,.
\end{equation}
Our preferred finite element approximation will then be based on the 
following weak formulation.
Find $\Gamma(t) = \vec x(\Upsilon, t)$ for $t\in[0,T]$ 
with $\vec{\mathcal{V}} \in L^2(0,T;[H^1(\Gamma(t))]^d)$, 
and functions $\vec u \in \utimespace$, 
$p \in L^2(0,T; \widehat\pspace)$, 
$\varkappa \in L^2(0,T; L^2(\Gamma(t)))$ and
$\psi \in \psispace$ such that 
for almost all $t \in (0,T)$ it holds that
\begin{subequations}
\begin{align}
& \tfrac12\left[ \ddt (\rho\,\vec u, \vec \xi) + (\rho\,\vec u_t, \vec \xi)
- (\rho\,\vec u, \vec \xi_t)
+ (\rho, [(\vec u\,.\,\nabla)\,\vec u]\,.\,\vec \xi
- [(\vec u\,.\,\nabla)\,\vec \xi]\,.\,\vec u) \right] 
\nonumber \\ & \qquad
+ 2\,(\mu\,\mat D(\vec u), \mat D(\vec \xi)) 
- (p, \nabla\,.\,\vec \xi)
- \left\langle \gamma(\psi)\,\varkappa\,\vec \nu + \nabs\,\gamma(\psi), 
  \vec \xi \right\rangle_{\Gamma(t)}
 = (\vec f, \vec \xi) 
\qquad \forall\ \vec \xi \in \utimespace\,,
\label{eq:weaka} \\
& (\nabla\,.\,\vec u, \varphi) = 0  \qquad \forall\ \varphi \in 
\widehat\pspace\,, \label{eq:weakb} \\
&  \left\langle \vec {\mathcal{V}} - \vec u, \chi\,\vec\nu \right\rangle_{\Gamma(t)} = 0
 \qquad\forall\ \chi \in L^2(\Gamma(t))\,,
\label{eq:weakc} \\
& \left\langle \varkappa\,\vec\nu, \vec\eta \right\rangle_{\Gamma(t)}
+ \left\langle \nabs\,\vec \id, \nabs\,\vec \eta \right\rangle_{\Gamma(t)}
 = 0  \qquad\forall\ \vec\eta \in [H^1(\Gamma(t))]^d\,, \label{eq:weakd} \\
& \ddt\left\langle \psi, \zeta \right\rangle_{\Gamma(t)}
+ \Ds \left\langle \nabs\,\psi, \nabs\,\zeta \right\rangle_{\Gamma(t)}
+ \left\langle \psi\,(\vec {\mathcal{V}} - \vec u), \nabs\,\zeta\right\rangle_{\Gamma(t)}
= \left\langle \psi, \matpartx\, \zeta \right\rangle_{\Gamma(t)}
\qquad\forall\ \zeta \in \psispace\,,
\label{eq:weake}
\end{align}
\end{subequations}
as well as the initial conditions (\ref{eq:1d}), where in (\ref{eq:weakc},e) 
we have recalled (\ref{eq:V}). 
The derivation of \mbox{(\ref{eq:weaka}--d)} is analogous to the derivation of
(\ref{eq:weakGDa}--d), while for the formulation (\ref{eq:weake}) we note
(\ref{eq:DElem5.2}) and, on recalling (\ref{eq:DEdef2.11}) 
and (\ref{eq:matpartxu}), the identity
\begin{align*}
 \left\langle \matpartx\, \psi + \psi\,\nabs\,.\,\vec {\mathcal{V}}, \zeta 
\right\rangle_{\Gamma(t)} &
=
 \left\langle \matpartu\, \psi + \psi\,\nabs\,.\,\vec u, \zeta 
\right\rangle_{\Gamma(t)}
 +  \left\langle (\vec{\mathcal{V}} - \vec u) \,.\,\nabs\, \psi
 + \psi\,\nabs\,.\,(\vec{\mathcal{V}} - \vec u), \zeta \right\rangle_{\Gamma(t)}
\nonumber \\ &
=
\left\langle \matpartu\, \psi + \psi\,\nabs\,.\,\vec u, \zeta 
\right\rangle_{\Gamma(t)}
 - \left\langle \psi\,(\vec{\mathcal{V}} - \vec u), \nabs\,\zeta 
 \right\rangle_{\Gamma(t)}\,,
\end{align*}
where we have used the fact that $\langle \vec{\mathcal{V}} - \vec u,
\psi\,\zeta\,\vec\varkappa \rangle_{\Gamma(t)} = 0$ due to (\ref{eq:weakc}). 
In fact, a simpler way of seeing that (\ref{eq:weake}) is consistent with
(\ref{eq:weakGDe}) is to recall that the latter holds with 
$\matpartx$ replaced by $\matpartu$, and so the desired result follows
immediately from (\ref{eq:matpartxu}).

The main differences between (\ref{eq:weakGDa}--e) and (\ref{eq:weaka}--e) are
that for the latter the scalar curvature $\varkappa$ is sought as part of the 
solution, rather than $\vec\varkappa$, that in the latter only the normal part 
of $\vec u$ affects the evolution of the parameterization $\vec x$, and that
as a consequence the weak formulation of the advection-diffusion has to account
for the additional freedom in the tangential velocity of the interface
parameterization.

Similarly to (\ref{eq:d1})--(\ref{eq:d5}), we can formally show that a solution
to (\ref{eq:weaka}--e) satisfies the a priori energy bound (\ref{eq:d5}). First
of all we note that since $\vec\varkappa=\varkappa\,\vec\nu$, a solution to
(\ref{eq:weaka}--e) satisfies (\ref{eq:d1}). Secondly we observe that the
analogue of (\ref{eq:d45}) has as right hand side
\begin{align} \label{eq:dbgn}
& - \left\langle \gamma(\psi)\,\vec\varkappa + \nabs\,\gamma(\psi), 
  \vec{\mathcal{V}} \right\rangle_{\Gamma(t)}
- \left\langle \psi\,(\vec {\mathcal{V}} - \vec u), \nabs\,F'(\psi)
\right\rangle_{\Gamma(t)} \nonumber \\ & \qquad
= - \left\langle \gamma(\psi)\,\varkappa\,\vec\nu+ \nabs\,\gamma(\psi), 
  \vec{\mathcal{V}} \right\rangle_{\Gamma(t)}
+ \left\langle \nabs\,\gamma(\psi), \vec {\mathcal{V}} - \vec u
\right\rangle_{\Gamma(t)} \nonumber \\ & \qquad
= -\left\langle \gamma(\psi)\,\varkappa\,\vec\nu+ \nabs\,\gamma(\psi), 
  \vec u \right\rangle_{\Gamma(t)}\,,
\end{align}
where we have used (\ref{eq:Fdd}) and (\ref{eq:weakc}) with 
$\chi = \gamma(\psi)\,\varkappa$.
Of course, (\ref{eq:dbgn}) now cancels with the last term in
(\ref{eq:d1}), and so we obtain (\ref{eq:d5}). Moreover, the properties
(\ref{eq:conserved}) and (\ref{eq:totalpsi}) also hold.

\setcounter{equation}{0}
\section{Finite element approximation} \label{sec:3}

\subsection{Semi-discrete approximation} \label{sec:31}
For simplicity we consider $\Omega$ to be a polyhedral domain. Then 
let  ${\cal T}^h$ 
be a regular partitioning of $\Omega$ into disjoint open simplices
$\sigmaO^h_j$, $j = 1 ,\ldots, J_\Omega^h$.
Associated with ${\cal T}^h$ are the finite element spaces
\begin{equation*} 
 S^h_k := \{\chi \in C(\overline{\Omega}) : \chi\!\mid_{\sigmaO} \in
\mathcal{P}_k(\sigmaO) \quad \forall\ \sigmaO \in {\cal T}^h\} 
\subset H^1(\Omega)\,, \qquad k \in \mathbb{N}\,,
\end{equation*}
where $\mathcal{P}_k(\sigmaO)$ denotes the space of polynomials of degree $k$
on $\sigmaO$. We also introduce $S^h_0$, the space of  
piecewise constant functions on ${\cal T}^h$.
Let $\{\varphi_{k,j}^h\}_{j=1}^{K_k^h}$ be the standard basis functions 
for $S^h_k$, $k\geq 0$.
We introduce $\vec I^h_k:[C(\overline{\Omega})]^d\to [S^h_k]^d$, $k\geq 1$, 
the standard interpolation
operators, such that $(\vec I^h_k\,\vec\eta)(\vec{p}_{k,j}^h) = 
\vec\eta(\vec{p}_{k,j}^h)$ for $j=1,\ldots, K_k^h$; where
$\{\vec p_{k,j}^h\}_{j=1}^{K_k^h}$ denotes the coordinates of the degrees of
freedom of $S^h_k$, $k\geq 1$. In addition we define the standard projection
operator $I^h_0:L^1(\Omega)\to S^h_0$, such that
\begin{equation*}
(I^h_0 \eta)\!\mid_{o} = \frac1{\mathcal{L}^d(o)}\,\int_{o}
\eta \dL{d} \qquad \forall\ o \in \mathcal{T}^h\,.
\end{equation*}
Our approximation to the velocity and pressure on ${\cal T}^h$
will be finite element spaces
$\uspace^h\subset\uspace$ and $\pspace^h(t)\subset\pspace$.
We require also the space $\widehat\pspace^h(t):= 
\pspace^h(t) \cap \widehat\pspace$. 
Based on the authors' earlier work in \cite{spurious,fluidfbp}, we will select 
velocity/pressure finite element spaces 
that satisfy the LBB inf-sup condition,
see e.g.\ \citet[p.~114]{GiraultR86}, and augment the pressure space by a 
single additional basis function, namely by the characteristic function of the
inner phase. 
For the obtained spaces $(\uspace^h,\pspace^h(t))$ we are unable to prove that
they satisfy an LBB condition.
The extension of the given pressure finite element space, which is an example
of an XFEM approach, leads to exact volume conservation of the two
phases within the finite element framework.
For the non-augmented spaces we may choose, for example, 
the lowest order Taylor--Hood element 
P2--P1, the P2--P0 element or the P2--(P1+P0) element on setting 
$\uspace^h=[S^h_2]^d\cap\uspace$, 
and $\pspace^h = S^h_1,\,S^h_0$ or $S^h_1+S^h_0$, respectively.
We refer to \cite{spurious,fluidfbp} for more details.

The parametric finite element spaces in order to approximate $\vec{x}$,
as well as $\vec\varkappa$ and $\varkappa$ in (\ref{eq:weakGDa}--e) and
(\ref{eq:weaka}--e), respectively, are defined as follows.
Similarly to \cite{gflows3d}, let
$\Gamma^h(t)\subset\R^d$ be a $(d-1)$-dimensional {\em polyhedral surface}, 
i.e.\ a union of non-degenerate $(d-1)$-simplices with no
hanging vertices (see \citet[p.~164]{DeckelnickDE05} for $d=3$),
approximating the closed surface $\Gamma(t)$. In
particular, let $\Gamma^h(t)=\bigcup_{j=1}^{J_\Gamma}
\overline{\sigma^h_j(t)}$, where $\{\sigma^h_j(t)\}_{j=1}^{J_\Gamma}$ is a
family of mutually disjoint open $(d-1)$-simplices with vertices
$\{\vec{q}^h_k(t)\}_{k=1}^{K_\Gamma}$.
Then let
\begin{align*} 
\Vht & := \{\vec\chi \in [C(\Gamma^h(t))]^d:\vec\chi\!\mid_{\sigma^h_j}
\mbox{ is linear}\ \forall\ j=1,\ldots, J_\Gamma\} \\ & 
=: [\Wht]^d \subset [H^1(\Gamma^h(t))]^d\,,
\end{align*}
where $\Wht \subset H^1(\Gamma^h(t))$ is the space of scalar continuous
piecewise linear functions on $\Gamma^h(t)$, with 
$\{\chi^h_k(\cdot,t)\}_{k=1}^{K_\Gamma}$ 
denoting the standard basis of $\Wht$, i.e.\
\begin{equation} \label{eq:bf}
\chi^h_k(\vec q^h_l(t),t) = \delta_{kl}\qquad
\forall\ k,l \in \{1,\ldots,K_\Gamma\}\,, t \in [0,T]\,.
\end{equation}
For later purposes, we also introduce 
$\pi^h(t): C(\Gamma^h(t))\to \Wht$, the standard interpolation operator
at the nodes $\{\vec{q}_k^h(t)\}_{k=1}^{K_\Gamma}$, and similarly
$\vec\pi^h(t): [C(\Gamma^h(t))]^d\to \Vht$.

For scalar and vector 
functions $\eta,\zeta$ on $\Gamma^h(t)$ 
we introduce the $L^2$--inner 
product $\langle\cdot,\cdot\rangle_{\Gamma^h(t)}$ over
the polyhedral surface $\Gamma^h(t)$ 
as follows
\begin{equation*} 
\left\langle \eta, \zeta \right\rangle_{\Gamma^h(t)} := 
\int_{\Gamma^h(t)} \eta\,.\,\zeta \dH{d-1}\,.
\end{equation*}
If $\eta,\zeta$ are piecewise continuous, with possible jumps
across the edges of $\{\sigma_j^h\}_{j=1}^{J_\Gamma}$,
we introduce the mass lumped inner product
$\langle\cdot,\cdot\rangle_{\Gamma^h(t)}^h$ as
\begin{equation} \label{eq:defNI}
\left\langle \eta, \zeta \right\rangle^h_{\Gamma^h(t)}  :=
\tfrac1d \sum_{j=1}^{J_\Gamma} \mathcal{H}^{d-1}(\sigma^h_j)\,\sum_{k=1}^{d} 
(\eta\,.\,\zeta)((\vec{q}^h_{j_k})^-),
\end{equation}
where $\{\vec{q}^h_{j_k}\}_{k=1}^{d}$ are the vertices of $\sigma^h_j$,
and where
we define $\eta((\vec{q}^h_{j_k})^-):=
\underset{\sigma^h_j\ni \vec{p}\to \vec{q}^h_{j_k}}{\lim}\, \eta(\vec{p})$.

Following \citet[(5.23)]{DziukE13}, we define the discrete material velocity
for $\vec z \in \Gamma^h(t)$ by
\begin{equation} \label{eq:Xht}
\vec{\mathcal{V}}^h(\vec z, t) := \sum_{k=1}^{K_\Gamma}
\left[\ddt\,\vec q^h_k(t)\right] \chi^h_k(\vec z, t) \,.
\end{equation}
Then, similarly to (\ref{eq:matpartx}), we define
\begin{equation} \label{eq:matpartxh}
\matpartxh\, \zeta = \zeta_t + \vec{\mathcal{V}}^h\,.\,\nabla\,\zeta
\qquad\forall\ \zeta \in H^1(\GhT)\,,
\end{equation}
where, similarly to (\ref{eq:GT}), we have defined the discrete
space-time surface
\begin{equation*} 
\GhT := \bigcup_{t \in [0,T]} \Gamma^h(t) \times \{t\}\,.
\end{equation*}
For later use, we also introduce the finite element space
\begin{equation*}
W(\GhT) := \{ \chi \in C(\GhT) : \matpartxh\,\chi \in C(\GhT) \text{ and } 
\chi(\cdot, t) \in \Wht \quad \forall\ t \in [0,T] \}\,.
\end{equation*}

On differentiating (\ref{eq:bf}) with respect to $t$, it immediately follows
that
\begin{equation} \label{eq:mpbf}
\matpartxh\, \chi^h_k = 0
\quad\forall\ k \in \{1,\ldots,K_\Gamma\}\,,
\end{equation}
see \citet[Lem.\ 5.5]{DziukE13}. 
It follows directly from (\ref{eq:mpbf}) that
\begin{equation} \label{eq:p1}
\matpartxh\,\zeta(\cdot,t) = \sum_{k=1}^{K_\Gamma} \chi^h_k(\cdot,t)\,
\ddt\,\zeta_k(t) \quad \text{on}\ \Gamma^h(t)
\end{equation}
for $\zeta(\cdot,t) = \sum_{k=1}^{K_\Gamma} \zeta_k(t)\,\chi^h_k(\cdot,t)
\in \Wht$, and hence
$\matpartxh\,\vec \id = \vec{\mathcal{V}}^h$ on $\Gamma^h(t)$.
Moreover, it holds that
\begin{equation} \label{eq:DEeq5.28}
\ddt\, \int_{\sigma^h_j(t)} \zeta \dH{d-1} 
= \int_{\sigma^h_j(t)} \matpartxh\,\zeta + \zeta\,\nabs\,.\,\vec{\mathcal{V}}^h 
\dH{d-1} \quad \forall\ \zeta \in H^1(\sigma^h_j(t))\,, j \in
\{1,\ldots,J_\Gamma\}\,,
\end{equation}
see \citet[Lem.\ 5.6]{DziukE13}. 
It immediately follows from (\ref{eq:DEeq5.28}) that
\begin{equation} \label{eq:DElem5.6}
\ddt \left\langle \eta, \zeta \right\rangle_{\Gamma^h(t)}
 = \left\langle \matpartxh\,\eta, \zeta \right\rangle_{\Gamma^h(t)}
 + \left\langle \eta, \matpartxh\,\zeta \right\rangle_{\Gamma^h(t)}
+ \left\langle \eta\,\zeta, \nabs\,.\,\vec{\mathcal{V}}^h \right\rangle_{\Gamma^h(t)}
\qquad \forall\ \eta,\zeta \in W(\GhT)\,, 
\end{equation}
which is a discrete analogue of (\ref{eq:DElem5.2}). 
It is not difficult to show that the
analogue of (\ref{eq:DElem5.6}) with numerical integration also holds. 
We establish
this result in the next lemma, together with a discrete variant of
(\ref{eq:DEdef2.11}), on recalling (\ref{eq:LBop}), for the case $d=2$. 

\begin{lem} \label{lem:DElem5.6NI}
It holds that
\begin{equation} \label{eq:DElem5.6NI}
\ddt \left\langle \eta, \zeta \right\rangle_{\Gamma^h(t)}^h
 = \left\langle \matpartxh\,\eta, \zeta \right\rangle_{\Gamma^h(t)}^h
 + \left\langle \eta, \matpartxh\,\zeta \right\rangle_{\Gamma^h(t)}^h
+ \left\langle \eta\,\zeta, \nabs\,.\,\vec{\mathcal{V}}^h \right\rangle_{\Gamma^h(t)}^h
\quad \forall\ \eta,\zeta \in W(\GhT)\,. 
\end{equation}
In addition, if $d=2$, it holds that
\begin{equation} \label{eq:JWB}
\left\langle \zeta, \nabs\,.\,\vec\eta \right\rangle_{\Gamma^h(t)}
+ \left\langle \nabs\,\zeta, \vec\eta \right\rangle_{\Gamma^h(t)} = 
\left\langle \nabs\,\vec \id, \nabs\,\vec\pi^h\,(\zeta\,\vec\eta) 
\right\rangle_{\Gamma^h(t)}
\quad \forall\ \zeta \in \Wht\,, \vec\eta \in \Vht\,.
\end{equation}
\end{lem}
\begin{proof}
Choosing $\zeta = 1$ in (\ref{eq:DEeq5.28}) yields that
\begin{equation} \label{eq:p3}
\ddt\, \mathcal{H}^{d-1}(\sigma^h_j(t))
= \mathcal{H}^{d-1}(\sigma^h_j(t))
\,\nabs\,.\,\vec{\mathcal{V}}^h(\cdot,t)\quad \text{on}\ \sigma^h_j(t)\,.
\end{equation}
Differentiating (\ref{eq:defNI}) with respect to $t$, and combining with
(\ref{eq:p3}) and (\ref{eq:p1}), yields the desired result 
(\ref{eq:DElem5.6NI}). 

For arbitrary $\zeta \in H^1(\Gamma^h(t))$ and 
$\vec\eta \in [H^1(\Gamma^h(t))]^2$ we have for $d=2$ that
\begin{align*}
\left\langle \nabs\,.\,(\zeta\,\vec\eta), 1 \right\rangle_{\Gamma^h(t)}
& = \left\langle \vec \id_s, (\zeta\,\vec\eta)_s \right\rangle_{\Gamma^h(t)}
= \left\langle \vec \id_s, (\vec\pi^h\,[\zeta\,\vec\eta])_s \right\rangle_{\Gamma^h(t)}
= \left\langle \nabs\,\vec \id, \nabs\,\vec\pi^h\,(\zeta\,\vec\eta) 
\right\rangle_{\Gamma^h(t)}\,,
\end{align*}
which yields the desired result (\ref{eq:JWB}) on noting that 
$\nabs\,.\,(\zeta\,\vec\eta) = \zeta\, \nabs\,.\,\vec\eta + 
\vec\eta\,.\,\nabs\,\zeta$.
\end{proof}

Given $\Gamma^h(t)$, we 
let $\Omega^h_+(t)$ denote the exterior of $\Gamma^h(t)$ and let
$\Omega^h_-(t)$ denote the interior of $\Gamma^h(t)$, so that
$\Gamma^h(t) = \partial \Omega^h_-(t) = \overline{\Omega^h_-(t)} \cap 
\overline{\Omega^h_+(t)}$. 
We then partition the elements of the bulk mesh 
$\mathcal{T}^h$ into interior, exterior and interfacial elements as follows.
Let
\begin{align*}
\mathcal{T}^h_-(t) & := \{ o \in \mathcal{T}^h : o \subset
\Omega^h_-(t) \} \,, \nonumber \\
\mathcal{T}^h_+(t) & := \{ o \in \mathcal{T}^h : o \subset
\Omega^h_+(t) \} \,, \nonumber \\
\mathcal{T}^h_{\Gamma^h}(t) & := \{ o \in \mathcal{T}^h : o \cap
\Gamma^h(t) \not = \emptyset \} \,. 
\end{align*}
Clearly $\mathcal{T}^h = \mathcal{T}^h_-(t) \cup \mathcal{T}^h_+(t) \cup
\mathcal{T}^h_{\Gamma}(t)$ is a disjoint partition.
In addition, we define the piecewise constant unit normal 
$\vec{\nu}^h(t)$ to $\Gamma^h(t)$ such that $\vec\nu^h(t)$ points into
$\Omega^h_+(t)$.
Moreover, we introduce the discrete density 
$\rho^h(t) \in S^h_0$ and the discrete viscosity $\mu^h(t) \in S^h_0$ as
\begin{equation*} 
\rho^h(t)\!\mid_{o} = \begin{cases}
\rho_- & o \in \mathcal{T}^h_-(t)\,, \\
\rho_+ & o \in \mathcal{T}^h_+(t)\,, \\
\tfrac12\,(\rho_- + \rho_+) & o \in \mathcal{T}^h_{\Gamma^h}(t)\,,
\end{cases}
\quad\text{and}\quad
\mu^h(t)\!\mid_{o} = \begin{cases}
\mu_- & o \in \mathcal{T}^h_-(t)\,, \\
\mu_+ & o \in \mathcal{T}^h_+(t)\,, \\
\tfrac12\,(\mu_- + \mu_+) & o \in \mathcal{T}^h_{\Gamma^h}(t)\,.
\end{cases}
\end{equation*}

In what follows we will introduce two different finite element approximations
for the free boundary problem
(\ref{eq:NSa}--d), (\ref{eq:sigma}), (\ref{eq:1a}--c), (\ref{eq:1surf}).
Here $\vec U^h(\cdot, t) \in \uspace^h$ will be an approximation to 
$\vec u(\cdot, t)$,
while $P^h(\cdot, t) \in \widehat\pspace^h(t)$ 
approximates $p(\cdot, t)$
and $\Psi^h(\cdot, t) \in \Wht$ approximates $\psi(\cdot, t)$.
When designing such a finite element approximation, a
careful decision has to be made about the {\em discrete tangential velocity} of
$\Gamma^h(t)$. The most natural choice is to select the velocity of the fluid,
i.e.\ (\ref{eq:weakGDc}) is appropriately discretized.
This then gives a natural discretization of the surfactant transport equation
(\ref{eq:1surf}). Note also that the approximation of curvature, 
recall (\ref{eq:LBop}), where 
now $\vec\varkappa= \varkappa\,\vec\nu$ is discretized
directly, goes back to the seminal paper
\cite{Dziuk91}. Overall, we then obtain the following semidiscrete
continuous-in-time finite element approximation, which is the semidiscrete
analogue of the weak formulation (\ref{eq:weakGDa}--e).
Given $\Gamma^h(0)$, $\vec U^h(\cdot,0) \in \uspace^h$
and $\Psi^h(\cdot, 0) \in \Whtz$, 
find $\Gamma^h(t)$ such that $\vec \id \!\mid_{\Gamma^h(t)} \in \Vht$
for $t \in [0,T]$, and functions 
$\vec U^h \in H^1(0,T; \uspace^h)$, 
$P^h \in L^2(0,T; \widehat\pspace^h(t))$, 
$\vec\kappa^h \in L^2(0,T; \Vht)$ and $\Psi^h \in W(\GhT)$ such that
for almost all $t \in (0,T)$
it holds that
\begin{subequations}
\begin{align}
&
\tfrac12 \left[ \ddt \left( \rho^h\,\vec U^h , \vec \xi \right)
+ \left( \rho^h\,\vec U^h_t , \vec \xi \right)
- (\rho^h\,\vec U^h, \vec \xi_t) \right]
 \nonumber \\ & \qquad
+ 2\left(\mu^h\,\mat D(\vec U^h), \mat D(\vec \xi) \right)
+ \tfrac12\left(\rho^h, 
 [(\vec I^h_2\,\vec U^h\,.\,\nabla)\,\vec U^h]\,.\,\vec \xi
- [(\vec I^h_2\,\vec U^h\,.\,\nabla)\,\vec \xi]\,.\,\vec U^h \right)
\nonumber \\ & \qquad
- \left(P^h, \nabla\,.\,\vec \xi\right)
= \left(\rho^h\,\vec f^h_1 + \vec f^h_2, \vec \xi\right)
 + \left\langle \gamma(\Psi^h)\,\vec \kappa^h + \nabs\,\pi^h\,[\gamma(\Psi^h)],
   \vec\xi\right\rangle_{\Gamma^h(t)}^h
\nonumber \\ & \hspace{9cm}
\qquad \forall\ \vec\xi \in H^1(0,T;\uspace^h) \,, \label{eq:sdGDa}\\
& \left(\nabla\,.\,\vec U^h, \varphi\right) = 0 
\qquad \forall\ \varphi \in \widehat\pspace^h(t)\,,
\label{eq:sdGDb} \\
& \left\langle \vec{\mathcal{V}}^h ,
\vec\chi \right\rangle_{\Gamma^h(t)}^h
= \left\langle \vec U^h, \vec\chi \right\rangle_{\Gamma^h(t)}^h
 \qquad\forall\ \vec\chi \in \Vht\,,
\label{eq:sdGDc} \\
& \left\langle \vec\kappa^h , \vec\eta \right\rangle_{\Gamma^h(t)}^h
+ \left\langle \nabs\,\vec \id, \nabs\,\vec \eta \right\rangle_{\Gamma^h(t)}
 = 0  \qquad\forall\ \vec\eta \in \Vht\,,\label{eq:sdGDd} \\
& \ddt
\left\langle \Psi^h, \chi \right\rangle_{\Gamma^h(t)}^h
+ \Ds\left\langle \nabs\, \Psi^h, \nabs\, \chi
\right\rangle_{\Gamma^h(t)}
= 
\left\langle \Psi^h, \matpartxh\, \chi \right\rangle_{\Gamma^h(t)}^h 
\qquad \forall\ \chi \in W(\GhT)\,,
\label{eq:sdGDe}
\end{align}
\end{subequations}
where we recall (\ref{eq:Xht}). 
Here we have defined 
$\vec f^h_i(\cdot,t) := \vec I^h_2\,\vec f_i(\cdot,t)$, $i= 1, 2$,
where here and throughout we assume that $\vec f_i \in
L^2(0,T;[C(\overline\Omega)]^d)$, $i=1,2$.
We observe that (\ref{eq:sdGDc}) collapses to
$\vec{\mathcal{V}}^h = \vec\pi^h\, \vec U^h \!\mid_{\Gamma^h(t)} \in \Vht$, 
which on recalling (\ref{eq:matpartxh}) turns out to be crucial for
the stability analysis for (\ref{eq:sdGDa}--e). It is for this reason that we
use mass lumping in (\ref{eq:sdGDc}), which then leads to mass lumping 
having to be used in the last term in (\ref{eq:sdGDa}), as well as for the
first term in (\ref{eq:sdGDd}).

We remark that the formulation (\ref{eq:sdGDe}) for the surfactant transport
equation (\ref{eq:1surf}) falls into the framework of ESFEM 
(evolving surface finite element method) as
coined by the authors in \cite{DziukE07}. In this particular instance, the
velocity of $\Gamma^h(t)$ is not a priori fixed, rather it arises
implicitly through the evolution of $\Gamma^h(t)$ as determined by
(\ref{eq:sdGDa}--e). Here we recall the important property 
(\ref{eq:mpbf}), 
which means that (\ref{eq:sdGDe}) simplifies if formulated
in terms of the basis functions $\{\chi^h_k(\cdot,t)\}_{k=1}^{K_\Gamma}$ of
$\Wht$.

In the following lemma we derive a discrete analogue of (\ref{eq:d1}).
\begin{lem} \label{lem:stabGD}
Let $\{(\Gamma^h, \vec U^h, P^h, \vec\kappa^h, \Psi^h)(t)\}_{t\in[0,T]}$ 
be a solution to {\rm (\ref{eq:sdGDa}--e)}. Then
\begin{align} \label{eq:lemGD}
& \tfrac12\,\ddt \,\|[\rho^h]^\frac12\,\vec U^h \|_0^2 +
2\,\| [\mu^h]^\frac12\,\mat D(\vec U^h)\|_0^2 \nonumber \\ & \hspace{3cm}
= (\rho^h\,\vec f_1^h + 
\vec f_2^h, \vec U^h)
+ \left\langle \gamma(\Psi^h) \,\vec\kappa^h + 
\nabs\,\pi^h\,[\gamma (\Psi^h)] , \vec U^h \right\rangle_{\Gamma^h(t)}^h\,.
\end{align}
\end{lem} 
\begin{proof}
The desired result (\ref{eq:lemGD}) follows immediately on choosing
$\vec \xi = \vec U^h$ in (\ref{eq:sdGDa}) and $\varphi = P^h$ in 
(\ref{eq:sdGDb}). 
\end{proof}

The next theorem derives a discrete analogue of the energy law (\ref{eq:d5}).
Here, similarly to (\ref{eq:d4}), it will be crucial to test (\ref{eq:sdGDe}) 
with an appropriate discrete variant of $F'(\Psi^h)$. It is for this reason
that we have to make the following well-posedness assumption. 
\begin{equation} \label{eq:assPsih}
\Psi^h(\cdot, t) < \psi_\infty \quad\text{on $\Gamma^h(t)$}\,,
\quad \forall\ t \in [0,T]\,.
\end{equation}
The theorem also establishes nonnegativity of $\Psi^h$ under the assumption
that
\begin{equation} \label{eq:chij}
\int_{\sigma^h_j(t)} \nabs \chi^h_i \,.\,\nabs \chi^h_k \dH{d-1} \leq 0 
\quad \forall\ i \neq k\,,\quad \forall\ t \in [0,T]\,,
\qquad j = 1,\ldots,J_\Gamma\,.
\end{equation}
We note that (\ref{eq:chij}) always holds for $d=2$, and it holds for $d=3$ if
all the triangles $\sigma^h_j(t)$ of $\Gamma^h(t)$ have no obtuse angles. A
direct consequence of (\ref{eq:chij}) is that for any monotonic
function $G \in C^{0,1}(\R)$ it holds that
\begin{align} \label{eq:LG}
& L_G\,
\left\langle \nabs\, \xi, \nabs\, \pi^{h}\,[G(\xi)] 
\right\rangle_{\Gamma^h(t)} \geq
\left\langle \nabs\, \pi^{h}\,[G(\xi)], \nabs\, \pi^{h}\,[G(\xi)] 
\right\rangle_{\Gamma^h(t)}  \quad \forall\ \xi \in \Wht\,,
\nonumber \\ & \hspace{12cm} \forall\ t \in [0,T]\,,
\end{align}
where $L_G \in \R_{>0}$ denotes its Lipschitz constant.
For example, (\ref{eq:LG}) holds for 
$G(r) = [r]_- := \min\{0, r\}$ with $L_G = 1$.

For the following theorem, we denote the $L^\infty$--norm on $\Gamma^h(t)$ by
$\| \cdot \|_{\infty, \Gamma^h(t)}$, i.e.\
$\| z \|_{\infty, \Gamma^h(t)} := \esssup_{\Gamma^h(t)} |z|$ for
$z : \Gamma^h(t) \to \R$.

\begin{thm} \label{thm:stabGD}
Let $\{(\Gamma^h, \vec U^h, P^h,\vec\kappa^h, \Psi^h)(t)\}_{t\in[0,T]}$ 
be a solution to {\rm (\ref{eq:sdGDa}--e)}. Then
\begin{equation} \label{eq:totalpsih}
\ddt \left\langle \Psi^h, 1 \right\rangle_{\Gamma^h(t)} = 0\,.
\end{equation}
In addition, if $\Ds=0$ or if {\rm (\ref{eq:chij})} and
\begin{equation} \label{eq:Xinfty}
\max_{0\leq t \leq T} \| \nabs\,.\,\vec{\mathcal{V}}^h \|_{\infty,\Gamma^h(t)} <
\infty
\end{equation}
hold, then
\begin{equation} \label{eq:sddmp}
\Psi^h(\cdot,t) \geq 0 \quad \forall\ t \in (0,T]
 \qquad \text{if}\quad \Psi^h(\cdot,0) \geq 0\,.
\end{equation}
Moreover, if $d=2$ and if {\rm (\ref{eq:sddmp})} and
{\rm (\ref{eq:assPsih})} hold,
then
\begin{align}
& \ddt\left(\tfrac12\,\|[\rho^h]^\frac12\,\vec U^h\|^2_{0} + 
 \left\langle F(\Psi^h) , 1 \right\rangle_{\Gamma^h(t)}^h \right) 
+ 2\,\|[\mu^h]^\frac12\,\mat D(\vec U^h)\|^2_{0}
\leq \left(\rho^h\,\vec f^h_1 + \vec f^h_2, \vec U^h\right) \,.
\label{eq:stabGD}
\end{align}
\end{thm}
\begin{proof}
The conservation property (\ref{eq:totalpsih}) follows immediately from
choosing $\chi = 1$ in (\ref{eq:sdGDe}).

If $\Ds=0$ then it immediately follows from (\ref{eq:sdGDe}), on recalling
(\ref{eq:mpbf}), that 
$$\ddt\left\langle \Psi^h, \chi_k^h
\right\rangle_{\Gamma^h(t)}^h
=\ddt \left[ \left\langle 1, \chi^h_k \right\rangle_{\Gamma^h(t)}
\Psi^h(\vec q^h_k(t), t) \right] = 0\,,$$ 
for $k = 1,\ldots,K_\Gamma$, which yields the desired result (\ref{eq:sddmp}) 
if $\Ds=0$. 
If $\Ds > 0$, then choosing $\chi = \pi^h\,[\Psi^h]_-$ in
(\ref{eq:sdGDe}) yields, on noting (\ref{eq:LG}) with $G = [\cdot]_-$
and (\ref{eq:DElem5.6NI}), that
\begin{align*}
\ddt \left\langle [\Psi^h]_-^2, 1 \right\rangle_{\Gamma^h(t)}^h 
& = \ddt \left\langle \Psi^h, [\Psi^h]_- \right\rangle_{\Gamma^h(t)}^h
\leq \left\langle \Psi^h, \matpartxh\,\pi^h\,[\Psi^h]_- 
\right\rangle_{\Gamma^h(t)}^h \nonumber \\ &
= \left\langle [\Psi^h]_-, \matpartxh\,\pi^h\,[\Psi^h]_- 
\right\rangle_{\Gamma^h(t)}^h 
= \tfrac12 \left\langle \matpartxh\, \pi^h\,[\Psi^h]_-^2 , 1 
\right\rangle_{\Gamma^h(t)}^h \nonumber \\ &
= \tfrac12\,\ddt \left\langle \pi^h\,[\Psi^h]_-^2 , 1 
\right\rangle_{\Gamma^h(t)}^h - \tfrac12 \left\langle \pi^h\,[\Psi^h]_-^2 , 
\nabs\,.\,\vec{\mathcal{V}}^h \right\rangle_{\Gamma^h(t)}^h
\nonumber \\ &
\leq - \left\langle \pi^h\,[\Psi^h]_-^2 , 
\nabs\,.\,\vec{\mathcal{V}}^h \right\rangle_{\Gamma^h(t)}^h 
\leq \| \nabs\,.\,\vec{\mathcal{V}}^h \|_{\infty,\Gamma^h(t)}
\left\langle \pi^h\,[\Psi^h]_-^2 , 1 \right\rangle_{\Gamma^h(t)}^h\,.
\end{align*}
A Gronwall inequality, together with (\ref{eq:Xinfty}),
now yields our desired result (\ref{eq:sddmp}). 

For the proof of (\ref{eq:stabGD}) we note that
the assumption (\ref{eq:assPsih}) means that we can
choose $\chi = \pi^h\,[F'(\Psi^h + \alpha)]$ in (\ref{eq:sdGDe}), with
$\alpha \in \R_{>0}$ such that $\Psi^h + \alpha < \psi_\infty$, to yield,
on recalling (\ref{eq:F}) and (\ref{eq:DElem5.6NI}), that
\begin{align}
& \ddt\, \left\langle F(\Psi^h+\alpha) - \gamma(\Psi^h+\alpha), 1 
\right\rangle_{\Gamma^h(t)}^h +
\Ds \left\langle \nabs\, (\Psi^h+\alpha), \nabs\, \pi^h\,[F'(\Psi^h+\alpha)]
\right\rangle_{\Gamma^h(t)} 
\nonumber \\ & \qquad
= \left\langle \Psi^h + \alpha, 
\matpartxh\,\pi^h\,[F'(\Psi^h + \alpha)] \right\rangle_{\Gamma^h(t)}^h
+ \alpha  \left\langle F'(\Psi^h + \alpha) , \nabs\,.\,\vec{\mathcal{V}}^h 
\right\rangle_{\Gamma^h(t)}^h\,,
\label{eq:sp1}
\end{align}
similarly to (\ref{eq:d4}).
For the remainder of the proof we assume that $d=2$.
It follows from (\ref{eq:Fdd}), (\ref{eq:defNI}) and (\ref{eq:p1}) that
we have a discrete analogue of (\ref{eq:d3}), i.e.\
\begin{equation} \label{eq:sp11} 
\left\langle \Psi^h + \alpha, \matpartxh\,\pi^h\,[F'(\Psi^h+\alpha)] 
\right\rangle_{\Gamma^h(t)}^h
= - \left\langle \matpartxh\,\pi^h\,[\gamma(\Psi^h+\alpha)], 1 
\right\rangle_{\Gamma^h(t)}^h\,,
\end{equation}
which means that (\ref{eq:sp1}), together with (\ref{eq:DElem5.6NI}), 
(\ref{eq:JWB}) and (\ref{eq:sdGDc},d), implies that
\begin{align}
& \ddt\, \left\langle F(\Psi^h + \alpha), 1 \right\rangle_{\Gamma^h(t)}^h +
\Ds \left\langle \nabs\, (\Psi^h+\alpha), \nabs\, \pi^h\,[F'(\Psi^h+\alpha)]
\right\rangle_{\Gamma^h(t)} 
\nonumber \\ & \qquad\qquad
= \left\langle \pi^h\,[\gamma(\Psi^h+\alpha)], 
\nabs\,.\,\vec{\mathcal{V}}^h \right\rangle_{\Gamma^h(t)} 
+ \alpha  \left\langle F'(\Psi^h + \alpha) , \nabs\,.\,\vec{\mathcal{V}}^h 
\right\rangle_{\Gamma^h(t)}^h
\nonumber \\ & \qquad\qquad
= \left\langle \nabs\,\vec \id, \nabs\,\pi^h\,[\gamma(\Psi^h+\alpha)\,\vec{\mathcal{V}}^h]
  \right\rangle_{\Gamma^h(t)}
- \left\langle \nabs\,\pi^h\,[\gamma(\Psi^h+\alpha)], 
\vec{\mathcal{V}}^h \right\rangle_{\Gamma^h(t)}
\nonumber \\ & \qquad\qquad\qquad
+ \alpha \left\langle F'(\Psi^h + \alpha) , \nabs\,.\,\vec{\mathcal{V}}^h 
\right\rangle_{\Gamma^h(t)}^h
\nonumber \\ & \qquad\qquad
= - \left\langle \vec\kappa^h , \gamma(\Psi^h+\alpha)\,\vec U^h 
\right\rangle_{\Gamma^h(t)}^h
- \left\langle \nabs\,\pi^h\,[\gamma(\Psi^h+\alpha)], 
\vec U^h \right\rangle_{\Gamma^h(t)}^h
\nonumber \\ & \qquad\qquad\qquad
+ \alpha \left\langle F'(\Psi^h + \alpha) , \nabs\,.\,\vec{\mathcal{V}}^h 
\right\rangle_{\Gamma^h(t)}^h\,.
\label{eq:sp2}
\end{align}
Next, on noting for $\Ds > 0$
that $G(\cdot) = F'(\cdot+\alpha)$ is monotonic, as $F$ is convex, 
and has a finite Lipschitz constant, on noting (\ref{eq:sddmp}),
it follows from our assumptions and (\ref{eq:LG}) that
\begin{equation} \label{eq:sp3}
\Ds \left\langle \nabs\, (\Psi^h+\alpha), \nabs\, \pi^h\,[F'(\Psi^h+\alpha)]
\right\rangle_{\Gamma^h(t)} \geq 0\,,
\end{equation}
and so we obtain that
\begin{align} \label{eq:sp4}
\ddt\, \left\langle F(\Psi^h + \alpha), 1 \right\rangle_{\Gamma^h(t)}^h & \leq
- \left\langle \vec\kappa^h , \gamma(\Psi^h+\alpha)\,\vec U^h 
\right\rangle_{\Gamma^h(t)}^h
- \left\langle \nabs\,\pi^h\,[\gamma(\Psi^h+\alpha)], 
\vec U^h \right\rangle_{\Gamma^h(t)}^h
\nonumber \\ & \qquad
+ \alpha \left\langle F'(\Psi^h + \alpha) , \nabs\,.\,\vec{\mathcal{V}}^h 
\right\rangle_{\Gamma^h(t)}^h\,.
\end{align}
Passing to the limit $\alpha\to 0$ in (\ref{eq:sp4}), 
noting (\ref{eq:F0}), and combining with
(\ref{eq:lemGD}), yields the desired result (\ref{eq:stabGD}).
\end{proof}

Clearly, (\ref{eq:totalpsih}) and (\ref{eq:stabGD}) are natural discrete 
analogues of (\ref{eq:totalpsi}) and (\ref{eq:d5}), respectively.

We note that while (\ref{eq:sdGDa}--e) is a very natural approximation, in
particular (\ref{eq:sdGDe}) for the surfactant transport, see also
\cite{DziukE07}, a drawback in practice is that the finitely many vertices of 
the triangulations $\Gamma^h(t)$ are moved with the flow, which can lead to
coalescence. If a remeshing procedure is applied to $\Gamma^h(t)$, then
theoretical results like stability are no longer valid.

It is with this in mind that we would like to introduce an alternative finite
element approximation. It will be based on the weak formulation 
(\ref{eq:weaka}--e), and on the schemes from \cite{spurious,fluidfbp} for the
two-phase flow problem in the bulk. Of course, the discretization of
(\ref{eq:weake}) is going to be more complicated than (\ref{eq:sdGDe}), but
the advantage is that good mesh properties can be shown for $\Gamma^h(t)$. In
practice this means that no remeshings or reparameterizations need to be
performed for $\Gamma^h(t)$.

The main difference to (\ref{eq:sdGDa}--e) is that (\ref{eq:sdGDc}) is replaced
with a discrete variant of (\ref{eq:weakc}). In particular, the discrete
tangential velocity of $\Gamma^h(t)$ is not defined via $\vec U^h(\cdot,t)$, 
but it
is chosen totally independent from the surrounding fluid. In fact, the discrete
tangential velocity is not prescribed directly, but it is implicitly
introduced via the novel approximation of curvature which was first introduced
by the authors in \cite{triplej} for the case $d=2$, and in \cite{gflows3d} for
the case $d=3$. This discrete tangential velocity is such that, 
in the case $d=2$, $\Gamma^h(t)$
will remain equidistributed for all times $t \in (0,T]$. For $d=3$, a weaker 
property can be shown, which still guarantees good meshes in practice.
We refer to \cite{triplej,gflows3d} for more details.

For this new finite element approximation we are unable to guarantee the
nonnegativity of $\Psi^h(\cdot,t)$, which is in contrast to the result
(\ref{eq:sddmp}) for the scheme (\ref{eq:sdGDa}--e). 
It is for this reason that, 
following similar ideas in \cite{surf,surf2d}, we introduce regularizations
$F_\epsilon \in C^2(-\infty,\psi_\infty)$ of $F\in C^2(0,\psi_\infty)$, where
$\epsilon > 0$ is a regularization parameter. In particular, we set
\begin{subequations}
\begin{equation} \label{eq:Freg}
F_\epsilon(r) = \begin{cases}
F(r) & r \geq \epsilon\,, \\
F(\epsilon) + F'(\epsilon)\,(r-\epsilon) +
\frac12\,F''(\epsilon)\,(r-\epsilon)^2 & r \leq \epsilon\,,
\end{cases}
\end{equation}
which in view of (\ref{eq:F}) leads to
\begin{equation} \label{eq:geps}
\gamma_\epsilon(r) = \begin{cases}
\gamma(r) & r \geq \epsilon\,, \\
\gamma(\epsilon) + 
\frac12\,F''(\epsilon)\,(\epsilon^2 - r^2) & r \leq \epsilon\,,
\end{cases}
\end{equation}
\end{subequations}
so that
\begin{equation} \label{eq:Feps}
\gamma_\epsilon(r) = F_\epsilon(r) - r\,F'_\epsilon(r) 
\quad\text{and}\quad \gamma_\epsilon'(r) = - r\,F_\epsilon''(r) 
\qquad \forall\ r < \psi_\infty \,.
\end{equation}

We propose the following semidiscrete
continuous-in-time finite element approximation, which is the semidiscrete
analogue of the weak formulation (\ref{eq:weaka}--e).
Given $\Gamma^h(0)$, $\vec U^h(\cdot,0) \in \uspace^h$
and $\Psi^h(\cdot, 0) \in \Whtz$, 
find $\Gamma^h(t)$ such that $\vec \id \!\mid_{\Gamma^h(t)} \in \Vht$
for $t \in [0,T]$, and functions 
$\vec U^h \in H^1(0,T; \uspace^h)$, 
$P^h \in L^2(0,T; \widehat\pspace^h(t))$, 
$\kappa^h \in L^2(0,T; \Wht)$ and $\Psi^h \in W(\GhT)$ such that
for almost all $t \in (0,T)$
it holds that
\begin{subequations}
\begin{align}
&
\tfrac12 \left[ \ddt \left( \rho^h\,\vec U^h , \vec \xi \right)
+ \left( \rho^h\,\vec U^h_t , \vec \xi \right)
- (\rho^h\,\vec U^h, \vec \xi_t) \right]
 \nonumber \\ & \qquad
+ 2\left(\mu^h\,\mat D(\vec U^h), \mat D(\vec \xi) \right)
+ \tfrac12\left(\rho^h, 
 [(\vec I^h_2\,\vec U^h\,.\,\nabla)\,\vec U^h]\,.\,\vec \xi
- [(\vec I^h_2\,\vec U^h\,.\,\nabla)\,\vec \xi]\,.\,\vec U^h \right)
\nonumber \\ & \qquad
- \left(P^h, \nabla\,.\,\vec \xi\right)
= \left(\rho^h\,\vec f^h_1 + \vec f^h_2, \vec \xi\right)
 + \left\langle \pi^h\,[\gamma_\epsilon(\Psi^h)\,\kappa^h]\,\vec\nu^h,
   \vec\xi\right\rangle_{\Gamma^h(t)}
\nonumber \\ & \qquad\qquad\qquad\qquad\qquad
 + \left\langle \nabs\,\pi^h\,[\gamma_\epsilon(\Psi^h)], 
   \vec\xi\right\rangle_{\Gamma^h(t)}^h
\qquad \forall\ \vec\xi \in H^1(0,T;\uspace^h)
\,, \label{eq:sdHGa}\\
& \left(\nabla\,.\,\vec U^h, \varphi\right)  = 0 
\quad \forall\ \varphi \in \widehat\pspace^h(t)\,,
\label{eq:sdHGb} \\
& \left\langle \vec{\mathcal{V}}^h ,
\chi\,\vec\nu^h \right\rangle_{\Gamma^h(t)}^h
= \left\langle \vec U^h, 
\chi\,\vec\nu^h \right\rangle_{\Gamma^h(t)} 
 \quad\forall\ \chi \in \Wht\,,
\label{eq:sdHGc} \\
& \left\langle \kappa^h\,\vec\nu^h, \vec\eta \right\rangle_{\Gamma^h(t)}^h
+ \left\langle \nabs\,\vec \id, \nabs\,\vec \eta \right\rangle_{\Gamma^h(t)}
 = 0  \quad\forall\ \vec\eta \in \Vht\,, \label{eq:sdHGd} \\
& \ddt
\left\langle \Psi^h, \chi \right\rangle_{\Gamma^h(t)}^h 
+ \Ds\left\langle \nabs\, \Psi^h, \nabs\, \chi
\right\rangle_{\Gamma^h(t)}
= \left\langle \Psi^h, \matpartxh\, \chi \right\rangle_{\Gamma^h(t)}^h 
- \left\langle \Psi^h_{\star,\epsilon} \left( \vec{\mathcal{V}}^h - \vec U^h \right) ,
\nabs\,\chi \right\rangle_{\Gamma^h(t)}^h
\nonumber \\ & \hspace{9cm}
\qquad \forall\ \chi \in W(\GhT)\,,
\label{eq:sdHGe}
\end{align}
\end{subequations}
where we recall (\ref{eq:Xht}). 
Here $\Psi^h_{\star,\epsilon} = \Psi^h$ for $d=3$ and, on recalling 
(\ref{eq:Feps}),
\begin{equation} \label{eq:Psihstar}
\Psi^h_{\star,\epsilon} = \begin{cases}
- \frac{\gamma_\epsilon(\Psi^h_k) - \gamma_\epsilon(\Psi^h_{k-1})}
{F'_\epsilon(\Psi^h_k) - F'_\epsilon(\Psi^h_{k-1})} & 
F'_\epsilon(\Psi^h_{k-1}) \not= F'_\epsilon(\Psi^h_k)\,, \\
\frac12\,(\Psi^h_{k-1} + \Psi^h_k) 
&  F'_\epsilon(\Psi^h_{k-1}) = F'_\epsilon(\Psi^h_k)\,,
\end{cases}
\quad\text{on}\quad [\vec q^h_{k-1}, \vec q^h_{k}]
\quad\forall\ k \in \{1,\ldots,K_\Gamma\}
\end{equation}
for $d=2$. Here we have introduced the shorthand notation
$\Psi^h_k(t) = \Psi^h(\vec q^h_k(t), t)$, for $k=1,\ldots,K_\Gamma$,
and for notational convenience we have
dropped the dependence on $t$ in (\ref{eq:Psihstar}). 
The definition in (\ref{eq:Psihstar}) is chosen such that for $d=2$
it holds that
\begin{align} \label{eq:doctored}
& \left\langle \Psi^h_{\star,\epsilon}\, \vec\eta ,
\nabs\,\pi^h\,[F'_\epsilon(\Psi^h)] \right\rangle_{\Gamma^h(t)}^h
= \left\langle \Psi^h_{\star,\epsilon}\, \vec\eta ,
\nabs\,\pi^h\,[F'_\epsilon(\Psi^h)] \right\rangle_{\Gamma^h(t)}
= - \left\langle \vec\eta, \nabs\,\pi^h\,[\gamma_\epsilon(\Psi^h)]
\right\rangle_{\Gamma^h(t)} \nonumber\\ & \hspace{11cm}\forall\ \vec\eta \in \Vht\,,
\end{align}
which will be crucial for the stability proof for (\ref{eq:sdHGa}--e). Note
that here the regularization (\ref{eq:Freg},b) is required in order to make
the definition (\ref{eq:Psihstar}) well-defined, 
where we recall from (\ref{eq:F}) that $F'$ in general is only well-defined
on the positive real line.
We observe that (\ref{eq:doctored}) for 
$\vec\eta = \vec{\mathcal{V}}^h - \vec\pi^h\,\vec U^h \!\mid_{\Gamma^h(t)}$ 
mimics (\ref{eq:dbgn}) on the discrete level.

Similarly to
Theorem~\ref{thm:stabGD} we are only able to prove stability for the scheme
(\ref{eq:sdHGa}--e) in the case $d=2$. Hence in the case $d=3$ the definition
(\ref{eq:Psihstar}) is not required, and so $\gamma_\epsilon$ in 
(\ref{eq:sdHGa}) may also be replaced by $\gamma$.

We remark that the formulation (\ref{eq:sdHGe}) for the surfactant transport
equation (\ref{eq:1surf}) falls into the framework of ALE ESFEM 
(arbitrary Lagrangian Eulerian evolving surface finite element method) as
coined by the authors in \cite{ElliottS12}. In this particular instance, the
tangential velocity of $\Gamma^h(t)$ is not a priori fixed, rather it arises
implicitly through the evolution of $\Gamma^h(t)$ as determined by
(\ref{eq:sdHGa}--e). 

Similarly to Lemma~\ref{lem:stabGD}, in the following lemma 
we derive a discrete analogue of (\ref{eq:d1}).
\begin{lem} \label{lem:stabHG}
Let $\{(\Gamma^h, \vec U^h, P^h, \kappa^h, \Psi^h)(t)\}_{t\in[0,T]}$ 
be a solution to {\rm (\ref{eq:sdHGa}--e)}. Then
\begin{align} \label{eq:lemHG}
& \tfrac12\,\ddt \,\|[\rho^h]^\frac12\,\vec U^h \|_0^2 +
2\,\| [\mu^h]^\frac12\,\mat D(\vec U^h)\|_0^2 \nonumber \\ & \qquad
= (\rho^h\,\vec f_1^h + 
\vec f_2^h, \vec U^h)
+ \left\langle \pi^h\,[\gamma_\epsilon(\Psi^h) \,\kappa^h]\,\vec\nu^h, \vec U^h
\right\rangle_{\Gamma^h(t)} + \left\langle
\nabs\,\pi^h\,[\gamma_\epsilon(\Psi^h)] , \vec U^h \right\rangle_{\Gamma^h(t)}^h\,.
\end{align}
\end{lem} 
\begin{proof}
The desired result (\ref{eq:lemHG}) follows immediately on choosing
$\vec \xi = \vec U^h$ in (\ref{eq:sdHGa}) and $\varphi = P^h$ in 
(\ref{eq:sdHGb}). 
\end{proof}

The next theorem derives a discrete analogue of the energy law (\ref{eq:d5}),
similarly to Theorem~\ref{thm:stabGD}, together with an exact volume
conservation property.

\begin{thm} \label{thm:stabHG}
Let $\{(\Gamma^h, \vec U^h, P^h, \kappa^h, \Psi^h)(t)\}_{t\in[0,T]}$ 
be a solution to {\rm (\ref{eq:sdHGa}--e)}. Then
\begin{equation} \label{eq:totalPsih}
\ddt \left\langle \Psi^h, 1 \right\rangle_{\Gamma^h(t)} = 0\,.
\end{equation}
Moreover, if $\charfcn{\Omega_-^h(t)} \in \pspace^h(t)$ then 
\begin{equation}
\ddt\, \vol(\Omega_-^h(t)) = 0\,. \label{eq:cons}
\end{equation}
In addition, if $d=2$ and if the assumption {\rm (\ref{eq:assPsih})} holds,
then
\begin{align}
& \ddt\left(\tfrac12\,\|[\rho^h]^\frac12\,\vec U^h\|^2_{0} + 
 \left\langle F_\epsilon(\Psi^h) , 1 \right\rangle_{\Gamma^h(t)}^h \right) 
+ 2\,\|[\mu^h]^\frac12\,\mat D(\vec U^h)\|^2_{0}
\leq \left(\rho^h\,\vec f^h_1 + \vec f^h_2, \vec U^h\right) \,.
\label{eq:stabHG}
\end{align}
\end{thm}
\begin{proof}
The conservation property (\ref{eq:totalPsih}) follows immediately from
choosing $\chi = 1$ in (\ref{eq:sdHGe}). Moreover, 
choosing $\chi = 1$ in (\ref{eq:sdHGc}) and
$\varphi= (\charfcn{\Omega_-^h(t)} -
\frac{\mathcal{L}^d(\Omega_-^h(t))}{\mathcal{L}^d(\Omega)})
\in \widehat\pspace^h(t)$ in (\ref{eq:sdHGb}), we obtain that
\begin{equation*}
\frac{\rm d}{{\rm d}t} \vol(\Omega_-^h(t)) = 
\left\langle \vec{\mathcal{V}}^h , \vec\nu^h \right\rangle_{\Gamma^h(t)}
= \left\langle \vec{\mathcal{V}}^h , \vec\nu^h \right\rangle^h_{\Gamma^h(t)}
= \left\langle \vec U^h, \vec\nu^h \right\rangle_{\Gamma^h(t)}
= \int_{\Omega_-^h(t)} \nabla\,.\,\vec U^h \dL{d} 
=0\,, 
\end{equation*}
which proves the desired result (\ref{eq:cons}). 
For the remainder of the proof we assume that $d=2$.

The assumption (\ref{eq:assPsih}) means that we can
choose $\chi = \pi^h\,[F'_\epsilon(\Psi^h)]$ in (\ref{eq:sdHGe}) to yield,
similarly to (\ref{eq:sp1})--(\ref{eq:sp2}), with $\alpha=0$ and $F$ replaced
by $F_\epsilon$, 
on recalling (\ref{eq:Feps}), (\ref{eq:DElem5.6NI}), (\ref{eq:JWB}), 
(\ref{eq:doctored}) and (\ref{eq:sdHGc},d), that
\begin{align}
& \ddt\, \left\langle F_\epsilon(\Psi^h), 1 \right\rangle_{\Gamma^h(t)}^h +
\Ds \left\langle \nabs\, \Psi^h, \nabs\, \pi^h\,[F'_\epsilon(\Psi^h)]
\right\rangle_{\Gamma^h(t)} 
\nonumber \\ & \qquad\qquad
= 
\left\langle \nabs\,\vec \id, \nabs\,\pi^h\,[\gamma_\epsilon(\Psi^h)\,\vec{\mathcal{V}}^h]
  \right\rangle_{\Gamma^h(t)}
- \left\langle \nabs\,\pi^h\,[\gamma_\epsilon(\Psi^h)], 
 \vec{\mathcal{V}}^h \right\rangle_{\Gamma^h(t)}
\nonumber \\ & \qquad\qquad\qquad
+ \left\langle \vec{\mathcal{V}}^h - \vec\pi^h\,\vec U^h, 
\nabs\,\pi^h\,[\gamma_\epsilon(\Psi^h)] \right\rangle_{\Gamma^h(t)}
\nonumber \\ & \qquad\qquad
= - \left\langle \kappa^h\,\vec\nu^h, \gamma_\epsilon(\Psi^h)\,\vec{\mathcal{V}}^h
  \right\rangle_{\Gamma^h(t)}^h
- \left\langle \vec U^h, \nabs\,\pi^h\,[\gamma_\epsilon(\Psi^h)] \right\rangle_{\Gamma^h(t)}^h 
\nonumber \\ & \qquad\qquad
= - \left\langle \pi^h\,[\gamma_\epsilon(\Psi^h)\,\kappa^h]\,\vec\nu^h, \vec U^h
  \right\rangle_{\Gamma^h(t)}
- \left\langle \nabs\,\pi^h\,[\gamma_\epsilon(\Psi^h)], \vec U^h \right\rangle_{\Gamma^h(t)}^h 
\,.
\label{eq:ps1}
\end{align}
Since $d=2$ we can apply (\ref{eq:LG}) to the function $G = F'_\epsilon$,
where we recall (\ref{eq:assPsih}) and
that $F_\epsilon\in C^2(-\infty,\psi_\infty)$ is convex, and
obtain that the second term on the left hand side of (\ref{eq:ps1}) is 
nonnegative. Hence
the desired result (\ref{eq:stabHG}) follows from combining (\ref{eq:ps1}) with
(\ref{eq:lemHG}).
\end{proof}

Clearly, (\ref{eq:totalPsih}), (\ref{eq:cons}) and
(\ref{eq:stabHG}) are natural discrete analogues of
(\ref{eq:totalpsi}), (\ref{eq:conserved}) and (\ref{eq:d5}), respectively.
We remark that the condition $\charfcn{\Omega_-^h(t)} \in \pspace^h(t)$ is
always satisfied for the \XFEMGAMMA\ approach as introduced in
\cite{spurious,fluidfbp}. 

In addition, it is possible to prove that the vertices of the solution 
$\Gamma^h(t)$ to (\ref{eq:sdHGa}--e) are
well distributed. As this follows already from the equations 
{\rm (\ref{eq:sdHGd})}, we
refer to our earlier work in \cite{triplej,gflows3d} for further details. In
particular, we observe that in the case $d=2$, i.e.\ for the planar two-phase
problem, an equidistribution property for the vertices of $\Gamma^h(t)$ can be
shown. These good mesh properties mean that for fully discrete schemes based on
(\ref{eq:sdHGa}--e) no remeshings are required in practice for either $d=2$ or
$d=3$.

We remark that for the scheme (\ref{eq:sdGDa}--e) it is not possible to prove
(\ref{eq:cons}), even if mass lumping was to be dropped from the right hand
side of (\ref{eq:sdGDc}),
because $\vec\chi = \vec\nu^h$ is not a valid test function in
(\ref{eq:sdGDc}). As a consequence, the volume of the two phases will
in general not be conserved in practice.
This is an additional advantage of the formulation
(\ref{eq:sdHGa}--e) over (\ref{eq:sdGDa}--e).
A disadvantage is the fact that it does not appear possible to
derive a discrete maximum principle similarly to (\ref{eq:sddmp}). However, the
following remark demonstrates that also for the scheme (\ref{eq:sdHGa}--e) 
the negative part of $\Psi^h$ can be controlled.
Moreover, in practice we observe that for a fully discrete variant of
(\ref{eq:sdHGa}--e) the fully discrete analogues of $\Psi^h(\cdot,t)$ remain
positive for positive initial data.

\begin{rem} \label{rem:Psi-}
The convex nature of $F$, together with the fact that $F'$ is 
singular at the origin, allows us to derive upper bounds on the negative part
of $\Psi^h$ for the two cases {\rm (\ref{eq:gamma1},b)}. 
On recalling {\rm (\ref{eq:Freg})} and {\rm (\ref{eq:F})}, it holds that 
\begin{equation*}
F_\epsilon(r) 
= \gamma(\epsilon) + F'(\epsilon)\,r + \tfrac12\,F''(\epsilon)\,(r-\epsilon)^2
\geq \tfrac12\,F''(\epsilon)\,r^2 \geq 
\tfrac12\,\epsilon^{-1}\,\gamma_0\,\beta\,r^2
\qquad \forall\ r \leq 0\,,
\end{equation*}
provided that $\epsilon$ is sufficiently small. 
Hence the bound {\rm (\ref{eq:stabHG})}, via a Korn's inequality, 
implies that
\begin{equation*} 
\left\langle [\Psi^h]_-^2, 1 \right\rangle_{\Gamma^h(t)}^h \leq C\,\epsilon
\qquad\forall\ t \in (0,T]
\qquad \text{if}\quad \Psi^h(\cdot,0) \geq 0\,,
\end{equation*}
for some positive constant $C$, and for $\epsilon$ sufficiently small.
\end{rem}

We recall that the stability proofs in Theorems~\ref{thm:stabGD} and
\ref{thm:stabHG} are restricted to the case $d=2$. However, it is possible to
prove stability for $d = 2$ and $d=3$ for a variant of (\ref{eq:sdGDa}--e),
which, on recalling (\ref{eq:newGD}), is given by
\begin{align}
&
\tfrac12 \left[ \ddt \left( \rho^h\,\vec U^h , \vec \xi \right)
+ \left( \rho^h\,\vec U^h_t , \vec \xi \right)
- (\rho^h\,\vec U^h, \vec \xi_t) \right]
 \nonumber \\ & \quad
+ 2\left(\mu^h\,\mat D(\vec U^h), \mat D(\vec \xi) \right)
+ \tfrac12\left(\rho^h, 
 [(\vec I^h_2\,\vec U^h\,.\,\nabla)\,\vec U^h]\,.\,\vec \xi
- [(\vec I^h_2\,\vec U^h\,.\,\nabla)\,\vec \xi]\,.\,\vec U^h \right)
\nonumber \\ & \quad
- \left(P^h, \nabla\,.\,\vec \xi\right)
= \left(\rho^h\,\vec f^h_1 + \vec f^h_2, \vec \xi\right)
 - \left\langle \gamma(\Psi^h), 
 \nabs\,.\,\vec\pi^h\,\vec\xi\right\rangle_{\Gamma^h(t)}^h
\qquad \forall\ \vec\xi \in H^1(0,T;\uspace^h) \,, \label{eq:sdGDa2}
\end{align}
together with (\ref{eq:sdGDb},c,e). Here we observe that in this new
discretization it is no longer necessary to compute the discrete curvature
vector $\vec\kappa^h$. It is then not difficult to prove the following theorem.

\begin{thm} \label{thm:stabGD2}
Let $\{(\Gamma^h, \vec U^h, P^h, \Psi^h)(t)\}_{t\in[0,T]}$ 
be a solution to {\rm (\ref{eq:sdGDa2})}, {\rm (\ref{eq:sdGDb},c,e)}. Then
{\rm (\ref{eq:totalpsih})} and
\begin{equation} \label{eq:lemGD2}
 \tfrac12\,\ddt \,\|[\rho^h]^\frac12\,\vec U^h \|_0^2 +
2\,\| [\mu^h]^\frac12\,\mat D(\vec U^h)\|_0^2 
= (\rho^h\,\vec f_1^h + 
\vec f_2^h, \vec U^h)
- \left\langle \gamma(\Psi^h) , \nabs\,.\,\vec\pi^h\,\vec U^h 
\right\rangle_{\Gamma^h(t)}^h
\end{equation}
hold. 
In addition, if $\Ds=0$ or if {\rm (\ref{eq:chij})} and
{\rm (\ref{eq:Xinfty})} hold, then we have {\rm (\ref{eq:sddmp})}.
Moreover, if {\rm (\ref{eq:assPsih})} and {\rm (\ref{eq:sddmp})} hold, and
$\Ds=0$ or {\rm (\ref{eq:chij})} holds, then
\begin{equation} 
 \ddt\left(\tfrac12\,\|[\rho^h]^\frac12\,\vec U^h\|^2_{0} + 
 \left\langle F(\Psi^h) , 1 \right\rangle_{\Gamma^h(t)}^h \right) 
+ 2\,\|[\mu^h]^\frac12\,\mat D(\vec U^h)\|^2_{0}
\leq \left(\rho^h\,\vec f^h_1 + \vec f^h_2, \vec U^h\right) \,.
\label{eq:stabGD2}
\end{equation}
\end{thm}
\begin{proof}
The desired results (\ref{eq:totalpsih}) and 
(\ref{eq:lemGD2}) follow immediately on choosing
$\chi = 1$ in (\ref{eq:sdGDe}) and on choosing
$\vec \xi = \vec U^h$ in (\ref{eq:sdGDa2}) and $\varphi = P^h$ in 
(\ref{eq:sdGDb}), respectively. The nonnegativity result (\ref{eq:sddmp})
can be shown as in the proof of Theorem~\ref{thm:stabGD}.
The stability bound (\ref{eq:stabGD2}) follows as in the proof of
Theorem~\ref{thm:stabGD}, on combining the first equation in (\ref{eq:sp2})
with (\ref{eq:lemGD2}) and 
$\vec{\mathcal{V}}^h = \vec\pi^h\,\vec U^h\!\mid_{\Gamma^h(t)}$, 
and on recalling that (\ref{eq:sp3}) holds if
our assumptions are satisfied. We note that this proof is valid for $d=3$, as
we do not use (\ref{eq:JWB}).
\end{proof}

We recall that the assumption (\ref{eq:chij}) always holds for $d=2$, but
for $d=3$ it will in general only be satisfied if all the triangles 
$\sigma^h_j(t)$ of $\Gamma^h(t)$ have no obtuse angles. 
Unfortunately, in practice this will in general not be the case.
Finally, 
we remark that it does not seem possible to derive a stability result for the
scheme (\ref{eq:sdGDa2}), (\ref{eq:sdHGb}--e) in the case $d=2$ or $d=3$.

\begin{rem} \label{rem:Fconst}
We note that in the special case of constant surface tension, i.e.\ when
{\rm (\ref{eq:Fconst})} holds, then, similarly to {\rm (\ref{eq:d45})},
the stability results {\rm (\ref{eq:stabGD})}, {\rm (\ref{eq:stabHG})} 
and {\rm (\ref{eq:stabGD2})} remain valid and reduce to
\begin{align}
& \ddt\left(\tfrac12\,\|[\rho^h]^\frac12\,\vec U^h\|^2_{0} + 
 \gamma_0\,\mathcal{H}^{d-1}(\Gamma^h(t)) \right) 
+ 2\,\|[\mu^h]^\frac12\,\mat D(\vec U^h)\|^2_{0}
\leq \left(\rho^h\,\vec f^h_1 + \vec f^h_2, \vec U^h\right) \,,
\label{eq:stabFconst}
\end{align}
where we note that $F_\epsilon = F = \gamma_0$ in {\rm (\ref{eq:stabHG})}.
The bound {\rm (\ref{eq:stabFconst})} recovers the stability results for
the semidiscrete variants of the fully discrete schemes from \cite{fluidfbp}
for two-phase Navier--Stokes flow.
\end{rem}

\subsection{Fully discrete approximation} \label{sec:32}

In this section we consider fully discrete variants of the schemes
(\ref{eq:sdGDa}--e) and (\ref{eq:sdHGa}--e) from \S\ref{sec:31}. Here we will
choose the time discretization such that existence and uniqueness of the
discrete solutions can be guaranteed, and such that we inherit as much of the
structure of the stable schemes in \cite{spurious,fluidfbp} as possible, see
below for details.

We consider the partitioning $t_m = m\,\tau$, $m = 0,\ldots, M$, 
of $[0,T]$ into uniform time steps $\tau = T / M$.
The time discrete spatial discretizations then directly follow from the finite
element spaces introduced in \S\ref{sec:31}, where here in order to allow for
local mesh refinements 
we consider bulk finite element spaces that change in time.

For all $m\ge 0$, let $\mathcal{T}^m$ 
be a regular partitioning of $\Omega$ into disjoint open simplices
$\sigmaO^m_j$, $j = 1 ,\ldots, J_\Omega^m$. 
We set $h^m:= \max_{j=1,\ldots,J_\Omega^m}\mbox{diam}( \sigmaO^m_j)$.
Associated with ${\cal T}^m$ are the finite element spaces
$S^m_k$ for $k\geq 0$.
We introduce also $\vec I^m_k: [C(\overline{\Omega})]^d \to [S^m_k]^d$, 
$k\geq 1$, the standard interpolation operators, and the standard projection
operator $I^m_0:L^1(\Omega)\to S^m_0$.
For the approximation to the velocity and pressure on ${\cal T}^m$
will use the finite element spaces
$\uspace^m\subset\uspace$ and $\pspace^m\subset\pspace$, which are the direct
time discrete analogues of $\uspace^h$ and $\pspace^h(t_m)$,
as well as $\widehat\pspace^m \subset \widehat\pspace$.
We recall that $(\uspace^m, \pspace^m)$ are said to satisfy 
the LBB inf-sup condition if
there exists a constant $C_0 \in \R_{>0}$ independent of $h^m$ such that
\begin{equation} \label{eq:LBB}
\inf_{\varphi \in \widehat\pspace^m} \sup_{\vec \xi \in \uspace^m}
\frac{( \varphi, \nabla \,.\,\vec \xi)}
{\|\varphi\|_0\,\|\vec \xi\|_1} \geq C_0\,.
\end{equation}

Similarly, the parametric finite element spaces are given by
\begin{equation*} 
\Vh := \{\vec\chi \in [C(\Gamma^m)]^d:\vec\chi\!\mid_{\sigma^m_j}
\mbox{ is linear}\ \forall\ j=1,\ldots, J_\Gamma\} 
=: [\Wh]^d \,,
\end{equation*}
for $m=0 ,\ldots, M-1$. Here
$\Gamma^m=\bigcup_{j=1}^{J_\Gamma} 
\overline{\sigma^m_j}$,
where $\{\sigma^m_j\}_{j=1}^{J_\Gamma}$ is a family of mutually disjoint open 
$(d-1)$-simplices 
with vertices $\{\vec{q}^m_k\}_{k=1}^{K_\Gamma}$. 
We denote the standard basis of $\Wh$ by
$\{\chi^m_k(\cdot,t)\}_{k=1}^{K_\Gamma}$.
We also introduce 
$\pi^m: C(\Gamma^m)\to \Wh$, the standard interpolation operator
at the nodes $\{\vec{q}_k^m\}_{k=1}^{K_\Gamma}$,
and similarly $\vec\pi^m: [C(\Gamma^m)]^d \to \Vh$.
Throughout this paper, we will parameterize the new closed surface 
$\Gamma^{m+1}$ over $\Gamma^m$, with the help of a parameterization
$\vec{X}^{m+1} \in \Vh$, i.e.\ $\Gamma^{m+1} = \vec{X}^{m+1}(\Gamma^m)$.
Moreover, for $m \geq 0$, we will use the notation 
$\vec{X}^m = \vec{\rm id}\!\mid_{\Gamma^m} \in \Vh$. 

We also introduce the $L^2$--inner 
product $\langle\cdot,\cdot\rangle_{\Gamma^m}$ over
the current polyhedral surface $\Gamma^m$, as well as the 
the mass lumped inner product
$\langle\cdot,\cdot\rangle_{\Gamma^m}^h$.
Given $\Gamma^m$, we 
let $\Omega^m_+$ denote the exterior of $\Gamma^m$ and let
$\Omega^m_-$ denote the interior of $\Gamma^m$, so that
$\Gamma^m = \partial \Omega^m_- = \overline{\Omega^m_-} \cap 
\overline{\Omega^m_+}$. 
We then partition the elements of the bulk mesh 
$\mathcal{T}^m$ into interior, exterior and interfacial elements as before, and
we introduce  
$\rho^m,\,\mu^m \in S^m_0$, for $m\geq 0$, as 
\begin{equation} \label{eq:rhoma}
\rho^m\!\mid_{o^m} = \begin{cases}
\rho_- & o^m \in \mathcal{T}^m_-\,, \\
\rho_+ & o^m \in \mathcal{T}^m_+\,, \\
\tfrac12\,(\rho_- + \rho_+) & o^m \in \mathcal{T}^m_{\Gamma^m}\,,
\end{cases}
\quad\text{and}\quad
\mu^m\!\mid_{o^m} = \begin{cases}
\mu_- & o^m \in \mathcal{T}^m_-\,, \\
\mu_+ & o^m \in \mathcal{T}^m_+\,, \\
\tfrac12\,(\mu_- + \mu_+) & o^m \in \mathcal{T}^m_{\Gamma^m}\,.
\end{cases}
\end{equation}
We also set $\rho^{-1} := \rho^0$.

Our proposed fully discrete equivalent of (\ref{eq:sdGDa}--e) is then given as
follows.
Let $\Gamma^0$, an approximation to $\Gamma(0)$, 
and $\vec U^0\in \uspace^0$, as well as $\vec\kappa^0 \in
\underline{V}(\Gamma^0)$ and $\Psi^0 \in W(\Gamma^0)$ be given.
For $m=0,\ldots, M-1$, find $\vec U^{m+1} \in \uspace^m$, 
$P^{m+1} \in \widehat\pspace^m$, $\vec{X}^{m+1}\in\Vh$ and
$\vec\kappa^{m+1} \in \Vh$ such that
\begin{subequations}
\begin{align}
&
\tfrac12 \left( \frac{\rho^m\,\vec U^{m+1} - (I^m_0\,\rho^{m-1})
\,\vec I^m_2\,\vec U^m}{\tau}
+(I^m_0\,\rho^{m-1}) \,\frac{\vec U^{m+1}- \vec I^m_2\,\vec{U}^m}{\tau}, 
\vec \xi \right)
 \nonumber \\ & \qquad
+ 2\left(\mu^m\,\mat D(\vec U^{m+1}), \mat D(\vec \xi) \right)
+ \tfrac12\left(\rho^m, 
 [(\vec I^m_2\,\vec U^m\,.\,\nabla)\,\vec U^{m+1}]\,.\,\vec \xi
- [(\vec I^m_2\,\vec U^m\,.\,\nabla)\,\vec \xi]\,.\,\vec U^{m+1} \right)
\nonumber \\ & \qquad
- \left(P^{m+1}, \nabla\,.\,\vec \xi\right)
= \left(\rho^m\,\vec f^{m+1}_1 + \vec f^{m+1}_2, \vec \xi\right)
 + \left\langle \gamma(\Psi^m)\,\vec \kappa^m
 + \nabs\,\pi^m\,[\gamma(\Psi^m)], 
   \vec\xi\right\rangle_{\Gamma^m}^h
\nonumber \\ & \hspace{11cm}
\qquad \forall\ \vec\xi \in \uspace^m \,, \label{eq:GDa}\\
& \left(\nabla\,.\,\vec U^{m+1}, \varphi\right) = 0 
\qquad \forall\ \varphi \in \widehat\pspace^m\,,
\label{eq:GDb} \\
& \left\langle \frac{\vec X^{m+1} - \vec X^m}{\tau} ,
\vec\chi \right\rangle_{\Gamma^m}^h
= \left\langle \vec U^{m+1}, \vec\chi \right\rangle_{\Gamma^m}^h
 \qquad\forall\ \vec\chi \in \Vh\,,
\label{eq:GDc} \\
& \left\langle \vec\kappa^{m+1} , \vec\eta \right\rangle_{\Gamma^m}^h
+ \left\langle \nabs\,\vec X^{m+1}, \nabs\,\vec \eta \right\rangle_{\Gamma^m}
 = 0  \qquad\forall\ \vec\eta \in \Vh\label{eq:GDd} \\
\intertext{and set $\Gamma^{m+1} = \vec{X}^{m+1}(\Gamma^m)$. 
We note that in (\ref{eq:GDa}), as no confusion can arise, for $m\geq1$
we denote by $\vec\kappa^m$ the function $\vec z \in \Vh$, defined by 
$\vec z(\vec{q}^m_k) = \vec\kappa^m(\vec{q}^{m-1}_k)$, $k=1, \ldots, K_\Gamma$, 
where $\vec\kappa^m \in \underline{V}(\Gamma^{m-1})$ is given.
Then find $\Psi^{m+1} \in \Whp$ such that}
& \frac1{\tau}
\left\langle \Psi^{m+1}, \chi^{m+1}_k \right\rangle_{\Gamma^{m+1}}^h
+ \Ds\left\langle \nabs\, \Psi^{m+1}, \nabs\, \chi^{m+1}_k
\right\rangle_{\Gamma^{m+1}}
= \frac1{\tau}
\left\langle \Psi^{m}, \chi^{m}_k \right\rangle_{\Gamma^m}^h 
\nonumber \\ & \hspace{10cm}
\quad\forall\ k \in \{1,\ldots,K_\Gamma\}\,.
\label{eq:GDe}
\end{align}
\end{subequations}
Here we have defined $\vec f^{m+1}_i := \vec I^m_2\,\vec
f_i(\cdot,t_{m+1})$, $i=1,2$.
We observe that (\ref{eq:GDa}--e) is a linear scheme in that
it leads to a linear system of equations for the unknowns 
$(\vec U^{m+1}, P^{m+1}, \vec{X}^{m+1}, \vec\kappa^{m+1},$ $ \Psi^{m+1})$ 
at each time level. In particular, the system (\ref{eq:GDa}--e) clearly
decouples into (\ref{eq:GDa},b) for $(\vec U^{m+1}, P^{m+1})$, then
(\ref{eq:GDc},d) for $(\vec{X}^{m+1}, \vec\kappa^{m+1})$ and finally 
(\ref{eq:GDe}) for $\Psi^{m+1}$.

\begin{rem} \label{rem:newGD}
Of course, the natural analogue of {\rm (\ref{eq:GDa}--e)} that is based on the
semidiscrete scheme from {\rm Theorem~\ref{thm:stabGD2}}, is given by:
Find $\vec U^{m+1} \in \uspace^m$, 
$P^{m+1} \in \widehat\pspace^m$, $\vec{X}^{m+1}\in\Vh$
and $\Psi^{m+1} \in \Whp$ such that {\rm (\ref{eq:GDa}--c,e)} hold with
$\langle \gamma(\Psi^m)\,\vec \kappa^{m} + \nabs\,\pi^m\,[\gamma(\Psi^m)],$ $
\vec\xi\rangle_{\Gamma^m}^h$ in {\rm (\ref{eq:GDa})} replaced by
$-\langle \gamma(\Psi^m), \nabs\,.\,\vec\pi^m\,\vec \xi \rangle_{\Gamma^m}^h$.
\end{rem}

When the velocity/pressure space pair $(\uspace^m,\widehat\pspace^m)$ does not
satisfy (\ref{eq:LBB}), we need to consider the following reduced version of
(\ref{eq:GDa},b), where the pressure $P^{m+1}$ is eliminated, 
in order to prove existence of a solution.
Let 
$$\uspace^m_0 := 
\{ \vec U \in \uspace^m : (\nabla\,.\,\vec U, \varphi) = 0 \ \
\forall\ \varphi \in \widehat\pspace^m \} \,.$$ 
Then any solution $(\vec U^{m+1}, P^{m+1}) \in 
\uspace^m\times\widehat\pspace^m$ to {\rm (\ref{eq:GDa},b)}
is such that $\vec U^{m+1} \in \uspace^m_0$ satisfies
(\ref{eq:GDa}) with $\uspace^m$ replaced by
$\uspace^m_0$.
In addition, 
we make the following very mild well-posedness assumption.

\begin{itemize}
\item[$(\mathcal{A})$]
We assume for $m=0,\ldots, M-1$ that $\mathcal{H}^{d-1}(\sigma^m_j) > 0$ 
for all $j=1,\ldots, J_\Gamma$,
and that $\Gamma^m \subset \Omega$.
\end{itemize}

Moreover, and similarly to (\ref{eq:chij}), we note that the assumption
\begin{equation} \label{eq:chijm}
\int_{\sigma^{m+1}_j} \nabs \chi^{m+1}_i \,.\,\nabs \chi^{m+1}_k 
\dH{d-1} \leq 0 
\quad \forall\ i \neq k\,,\qquad j = 1,\ldots,J_\Gamma
\end{equation}
is always satisfied for $d=2$, and for $d=3$ if all the triangles 
$\sigma^{m+1}$ of $\Gamma^{m+1}$ have no obtuse angles.

\begin{thm} \label{thm:GD}
Let the assumption $(\mathcal{A})$ hold.
If the LBB condition {\rm (\ref{eq:LBB})} holds, then there exists a unique
solution $(\vec U^{m+1}, P^{m+1}) \in \uspace^m\times\widehat\pspace^m$ 
to {\rm (\ref{eq:GDa},b)}. In all other
cases there exists a unique solution $\vec U^{m+1} \in \uspace^m_0$ to the
reduced equation {\rm (\ref{eq:GDa})} with $\uspace^m$ replaced by
$\uspace^m_0$.
In either case, there exists a unique solution
$(\vec{X}^{m+1}, \vec\kappa^{m+1}) \in \Vh \times \Vh$ to 
{\rm (\ref{eq:GDc},d)}.
Finally, there exists a unique solution $\Psi^{m+1}\in\Whp$ to 
{\rm (\ref{eq:GDe})} that satisfies 
\begin{subequations}
\begin{equation} \label{eq:consm}
\left\langle \Psi^{m+1}, 1 \right\rangle_{\Gamma^{m+1}} =
\left\langle \Psi^{m}, 1 \right\rangle_{\Gamma^{m}}
\end{equation}
and, if $\Ds=0$ or if the assumption {\rm (\ref{eq:chijm})} holds,
\begin{equation} \label{eq:dmp}
\Psi^{m+1} \geq 0 \qquad \text{if}\quad \Psi^m \geq 0\,.
\end{equation}
\end{subequations}
\end{thm}
\begin{proof}
As all the systems are linear, existence follows from uniqueness.
In order to establish the latter, we will consider the homogeneous 
system in each case. We begin with:
Find $(\vec U, P) \in \uspace^m\times\widehat\pspace^m$ such that
\begin{subequations}
\begin{align}
&
\tfrac1{2\,\tau} \left( (\rho^m+I^m_0\,\rho^{m-1})\,\vec U, \vec \xi \right)
+ 2\left(\mu^m\,\mat D(\vec U), \mat D(\vec \xi) \right)
- \left(P, \nabla\,.\,\vec \xi\right)
\nonumber \\ & \qquad
+ \tfrac12\left(\rho^m, [(\vec I^m_2\,\vec U^m\,.\,\nabla)\,\vec U]\,.\,\vec \xi
- [(\vec I^m_2\,\vec U^m\,.\,\nabla)\,\vec \xi]\,.\,\vec U \right)
= 0 
 \qquad \forall\ \vec\xi \in \uspace^m \,, \label{eq:proofa}\\
& \left(\nabla\,.\,\vec U, \varphi\right)  = 0 
\qquad \forall\ \varphi \in \widehat\pspace^m\,.
\label{eq:proofb} 
\end{align}
\end{subequations}
Choosing $\vec\xi=\vec U$ in (\ref{eq:proofa}) and
$\varphi =  P$ in (\ref{eq:proofb})
yields that
\begin{align}
& \tfrac12\left((\rho^m + I^m_0\,\rho^{m-1})\,\vec U, \vec U \right) + 
2\,\tau\left(\mu^m\,\mat D(\vec U), \mat D(\vec U) \right)
=0\,. \label{eq:proof2GD}
\end{align}
It immediately follows from (\ref{eq:proof2GD}), on recalling
$\rho_\pm > 0$, 
that $\vec U = \vec 0 \in \uspace^m$.
Moreover, (\ref{eq:proofa}) with $\vec U = \vec 0$ implies,
together with (\ref{eq:LBB}), that $P = 0 \in \widehat\pspace^m$. This shows
existence and uniqueness of 
$(\vec U^{m+1}, P^{m+1}) \in \uspace^m\times\widehat\pspace^m$.
The proof for the reduced equation is very similar. The homogeneous system to
consider is (\ref{eq:proofa}) with $\uspace^m$ replaced by
$\uspace^m_0$, where we note that the latter is a linear subspace of
$\uspace^m$. As before, (\ref{eq:proof2GD}) 
yields that $\vec U = \vec 0 \in \uspace^m_0$, and so the existence of a unique
solution $\vec U^{m+1} \in \uspace^m_0$ to the reduced equation.

Next we consider: Find $(\vec X, \vec\kappa) \in \Vh \times\Vh$ such that
\begin{align*}
&  \left\langle \vec X, \vec\chi \right\rangle_{\Gamma^m}^h
 = 0
 \qquad\forall\ \vec\chi \in \Vh\,, 
\\
& \left\langle \vec\kappa, \vec\eta \right\rangle_{\Gamma^m}^h
+ \left\langle \nabs\,\vec X, \nabs\,\vec \eta \right\rangle_{\Gamma^m}
 = 0  \qquad\forall\ \vec\eta \in \Vh\,, 
\end{align*}
which immediately implies that $\vec X = \vec 0$ and hence $\vec\kappa=\vec0$.
Finally, (\ref{eq:GDe}) is clearly a symmetric, positive definite linear
system with a unique solution $\Psi^{m+1} \in \Whp$. The desired result
(\ref{eq:consm}) follows on summing (\ref{eq:GDe}) for 
$k = 1, \ldots, K_\Gamma$. 
In order to prove (\ref{eq:dmp}) we assume that $\Psi^m\geq0$ and observe from 
(\ref{eq:GDe}) that this implies that
\begin{equation} \label{eq:dmp1}
\left\langle \Psi^{m+1}, [\Psi^{m+1}]_- \right\rangle_{\Gamma^{m+1}}^h
+ \tau\,\Ds\left\langle \nabs\, \Psi^{m+1}, \nabs\, \pi^{m+1}\,[\Psi^{m+1}]_-
\right\rangle_{\Gamma^{m+1}} \leq 0\,.
\end{equation}
Similarly to (\ref{eq:LG}) it follows that under our assumptions the
second term in (\ref{eq:dmp1}) is nonnegative,
which yields that 
$$\left\langle [\Psi^{m+1}]_-, [\Psi^{m+1}]_- 
\right\rangle_{\Gamma^{m+1}}^h 
= \left\langle \Psi^{m+1}, [\Psi^{m+1}]_- \right\rangle_{\Gamma^{m+1}}^h 
\leq0\,,$$
i.e.\ $\Psi^{m+1} \geq 0$.
\end{proof}

Our proposed fully discrete equivalent of (\ref{eq:sdHGa}--e) is given as
follows, where we recall the regularization parameter $\epsilon>0$ and the
definitions (\ref{eq:Freg},b).
Let $\Gamma^0$, an approximation to $\Gamma(0)$, 
and $\vec U^0\in \uspace^0$, as well as $\kappa^0 \in W(\Gamma^0)$
and $\Psi^0 \in W(\Gamma^0)$ be given.
For $m=0,\ldots, M-1$, find $\vec U^{m+1} \in \uspace^m$, 
$P^{m+1} \in \widehat\pspace^m$, $\vec{X}^{m+1}\in\Vh$ and 
$\kappa^{m+1} \in \Wh$ such that
\begin{subequations}
\begin{align}
&
\tfrac12 \left( \frac{\rho^m\,\vec U^{m+1} - (I^m_0\,\rho^{m-1})
\,\vec I^m_2\,\vec U^m}{\tau}
+(I^m_0\,\rho^{m-1}) \,\frac{\vec U^{m+1}- \vec I^m_2\,\vec{U}^m}{\tau}, 
\vec \xi \right)
 \nonumber \\ & \qquad
+ 2\left(\mu^m\,\mat D(\vec U^{m+1}), \mat D(\vec \xi) \right)
+ \tfrac12\left(\rho^m, 
 [(\vec I^m_2\,\vec U^m\,.\,\nabla)\,\vec U^{m+1}]\,.\,\vec \xi
- [(\vec I^m_2\,\vec U^m\,.\,\nabla)\,\vec \xi]\,.\,\vec U^{m+1} \right)
\nonumber \\ & \qquad
- \left(P^{m+1}, \nabla\,.\,\vec \xi\right)
= \left(\rho^m\,\vec f^{m+1}_1 + \vec f^{m+1}_2, \vec \xi\right)
 + \left\langle \pi^m\,[\gamma_\epsilon(\Psi^m)\,\kappa^{m}]\,\vec\nu^m,
   \vec\xi\right\rangle_{\Gamma^m}
\nonumber \\ & \qquad\qquad\qquad\qquad\qquad\qquad
 + \left\langle \nabs\,\pi^m\,[\gamma_\epsilon(\Psi^m)], 
   \vec\xi\right\rangle_{\Gamma^m}^h
\qquad \forall\ \vec\xi \in \uspace^m 
\,, \label{eq:HGa}\\
& \left(\nabla\,.\,\vec U^{m+1}, \varphi\right)  = 0 
\quad \forall\ \varphi \in \widehat\pspace^m\,,
\label{eq:HGb} \\
&  \left\langle \frac{\vec X^{m+1} - \vec X^m}{\tau} ,
\chi\,\vec\nu^m \right\rangle_{\Gamma^m}^h
= \left\langle \vec U^{m+1}, 
\chi\,\vec\nu^m \right\rangle_{\Gamma^m} 
 \quad\forall\ \chi \in \Wh\,,
\label{eq:HGc} \\
& \left\langle \kappa^{m+1}\,\vec\nu^m, \vec\eta \right\rangle_{\Gamma^m}^h
+ \left\langle \nabs\,\vec X^{m+1}, \nabs\,\vec \eta \right\rangle_{\Gamma^m}
 = 0  \quad\forall\ \vec\eta \in \Vh \label{eq:HGd} \\
\intertext{and set $\Gamma^{m+1} = \vec{X}^{m+1}(\Gamma^m)$. 
We note that in (\ref{eq:HGa}), similarly to $\vec\kappa^m$ in (\ref{eq:GDa}), 
for $m\geq 1$ 
we denote by $\kappa^m$ the function $z \in \Wh$, defined by 
$z(\vec{q}^m_k) = \kappa^m(\vec{q}^{m-1}_k)$, $k=1,\ldots, K_\Gamma$, 
where $\kappa^m \in W(\Gamma^{m-1})$ is given.
Then find $\Psi^{m+1} \in \Whp$ such that}
& \frac1{\tau}
\left\langle \Psi^{m+1}, \chi^{m+1}_k \right\rangle_{\Gamma^{m+1}}^h 
+ \Ds\left\langle \nabs\, \Psi^{m+1}, \nabs\, \chi^{m+1}_k
\right\rangle_{\Gamma^{m+1}}
\nonumber \\ & \quad
= \frac1{\tau}
 \left\langle \Psi^{m}, \chi^{m}_k \right\rangle_{\Gamma^m}^h 
 - \left\langle \Psi^m_{\star,\epsilon} \left( 
\frac{\vec X^{m+1} - \vec X^m}{\tau} - \vec U^{m+1} \right) ,
\nabs\,\chi^{m}_k \right\rangle_{\Gamma^m}^h
\quad\forall\ k \in \{1,\ldots,K_\Gamma\}\,,
\label{eq:HGe}
\end{align}
\end{subequations}
where $\Psi^m_{\star,\epsilon} = \Psi^m$ for $d=3$ and, 
similarly to (\ref{eq:Psihstar}),
\begin{equation*} 
\Psi^m_{\star,\epsilon} = \begin{cases}
- \frac{\gamma_\epsilon(\Psi^m_k) - \gamma_\epsilon(\Psi^m_{k-1})}
{F'_\epsilon(\Psi^m_k) - F'_\epsilon(\Psi^m_{k-1})} &
F'_\epsilon(\Psi^m_{k-1}) \not= F'_\epsilon(\Psi^m_k)\,, \\
\frac12\,(\Psi^m_{k-1} + \Psi^m_k) 
& F'_\epsilon(\Psi^m_{k-1}) = F'_\epsilon(\Psi^m_k)\,,
\end{cases}
\quad\text{on}\quad [\vec q^m_{k-1}, \vec q^m_{k}]
\quad\forall\ k \in \{1,\ldots,K_\Gamma\}
\end{equation*}
for $d=2$, where $\Psi^m = \sum_{k=1}^{K_\Gamma} \Psi^m_k\,\chi^m_k$.
We observe that (\ref{eq:HGa}--e) is a linear scheme in that
it leads to a linear system of equations for the unknowns 
$(\vec U^{m+1}, P^{m+1}, \vec{X}^{m+1}, \kappa^{m+1}, \Psi^{m+1})$ 
at each time level. In particular, the system (\ref{eq:HGa}--e) clearly
decouples into (\ref{eq:HGa},b) for $(\vec U^{m+1}, P^{m+1})$, then
(\ref{eq:HGc},d) for $(\vec{X}^{m+1}, \kappa^{m+1})$ and finally (\ref{eq:HGe}) 
for $\Psi^{m+1}$. 

In order to prove the existence of a unique solution to (\ref{eq:HGc},d) we
need to make the following very mild additional assumption.

\begin{itemize}
\item[$(\mathcal{B})$]
For $k= 1 , \ldots, K_\Gamma$, let
$\Xi_k^m:= \{\sigma^m_j : \vec{q}^m_k \in \overline{\sigma^m_j}\}$
and set
\begin{equation*}
\Lambda_k^m := \bigcup_{\sigma^m_j \in \Xi_k^m} \overline{\sigma^m_j}
 \qquad \mbox{and} \qquad
\vec\omega^m_k := \frac{1}{\mathcal{H}^{d-1}(\Lambda^m_k)}
\sum_{\sigma^m_j\in \Xi_k^m} \mathcal{H}^{d-1}(\sigma^m_j)
\;\vec{\nu}^m_j\,. 
\end{equation*}
Then we further assume that 
$\dim \spa\{\vec{\omega}^m_k\}_{k=1}^{K_\Gamma} = d$, $m=0,\ldots, M-1$.
\end{itemize}
We refer to \cite{triplej} and \cite{gflows3d} for more details and for an
interpretation of this assumption, but we note that
$(\mathcal{B})$ is always satisfied if $\Gamma^m$ has no self-intersections.
Given the above definitions, we introduce the piecewise linear 
vertex normal function
\begin{equation*} 
\vec\omega^m := \sum_{k=1}^{K_\Gamma} \chi^m_k\,\vec\omega^m_k \in \Vh \,,
\end{equation*}
and note that 
\begin{equation} 
\left\langle \vec{v}, w\,\vec\nu^m\right\rangle_{\Gamma^m}^h =
\left\langle \vec{v}, w\,\vec\omega^m\right\rangle_{\Gamma^m}^h 
\qquad \forall\ \vec{v} \in \Vh\,,\ w \in \Wh \,.
\label{eq:NI}
\end{equation}

\begin{thm} \label{thm:BGN}
Let the assumption $(\mathcal{A})$ hold.
If the LBB condition {\rm (\ref{eq:LBB})} holds, then there exists a unique
solution $(\vec U^{m+1}, P^{m+1}) \in \uspace^m\times\widehat\pspace^m$ 
to {\rm (\ref{eq:HGa},b)}. In all other
cases there exists a unique solution $\vec U^{m+1} \in \uspace^m_0$ to the
reduced equation {\rm (\ref{eq:HGa})} with $\uspace^m$ replaced by
$\uspace^m_0$.
If the assumption ($\mathcal{B}$) holds,
then there exists a unique solution
$(\vec{X}^{m+1}, \kappa^{m+1}) \in \Vh \times \Wh$ to 
{\rm (\ref{eq:HGc},d)}.
Finally, there exists a unique solution $\Psi^{m+1}\in\Whp$ to 
{\rm (\ref{eq:HGe})} that satisfies {\rm (\ref{eq:consm})}. 
\end{thm}
\begin{proof}
The results for $\vec U^{m+1}$, $P^{m+1}$ and $\Psi^{m+1}$ can be shown exactly
as in the proof of Theorem~\ref{thm:GD}. For the remaining result we consider:
Find $(\vec{X}, \kappa) \in \Vh \times \Wh$ such that
\begin{subequations}
\begin{align}
&  \left\langle \vec X ,
\chi\,\vec\nu^m \right\rangle_{\Gamma^m}^h
 = 0  \qquad\forall\ \chi \in \Wh\,,
\label{eq:proofc} \\
& \left\langle \kappa\,\vec\nu^m, \vec\eta \right\rangle_{\Gamma^m}^h
+ \left\langle \nabs\,\vec X, \nabs\,\vec \eta \right\rangle_{\Gamma^m}
 = 0  \qquad\forall\ \vec\eta \in \Vh\,.
\label{eq:proofd}
\end{align}
\end{subequations}
Choosing $\chi=\kappa$ in (\ref{eq:proofc}) and
$\vec\eta = \vec X$ in (\ref{eq:proofd}) yields that
\begin{align}
\left\langle \nabs\,\vec{X}, \nabs\,\vec{X} \right\rangle_{\Gamma^m} 
=0\,. \label{eq:proof2}
\end{align}
It immediately follows from (\ref{eq:proof2}) 
that $\vec{X} = \vec{X}_c \in \R^d$. 
Together with (\ref{eq:proofc}), (\ref{eq:NI}) and the assumption 
$(\mathcal{B})$ this yields that $\vec{X} = \vec0$.
Now (\ref{eq:proofd}) with $\vec\eta=\vec\pi^m[\kappa\,\vec\omega^m]$, 
recall (\ref{eq:NI}), implies that $\kappa = 0$.
\end{proof}

\begin{rem} \label{rem:fluidfbp}
On replacing $\kappa^m$ in {\rm (\ref{eq:HGa})} with $\kappa^{m+1}$
the subsystem {\rm (\ref{eq:HGa}--d)} no longer decouples. 
However, this system, for the special case of constant surface tension, as in
{\rm (\ref{eq:Fconst})}, 
i.e.\ for a two-phase flow problem without surfactants,
has been considered by the authors in \cite{fluidfbp}. For this finite element
approximation of two-phase flow, the authors proved the existence 
of a unique solution 
$(\vec U^{m+1}, \vec{X}^{m+1}, \kappa^{m+1}) \in \uspace^m_0\times\Vh\times\Wh$
to the reduced system {\rm (\ref{eq:HGa},c,d)},
with $\uspace^m$ replaced by $\uspace^m_0$, and with $\kappa^m$ in 
{\rm (\ref{eq:HGa})} replaced by $\kappa^{m+1}$,
which in addition satisfies the following stability bound:
\begin{align*} 
& \tfrac12\,(\rho^m\,\vec U^{m+1}, \vec U^{m+1}) + 
\gamma_0\, \mathcal{H}^{d-1}(\Gamma^{m+1})
+ \tfrac12\left((I^m_0\rho^{m-1})\,(\vec U^{m+1} - \vec I^m_2\,\vec U^m), 
\vec U^{m+1} - \vec I^m_2\,\vec U^m \right) \nonumber \\ & \hspace{5cm}
+ 2\,\tau\left(\mu^m\,\mat D(U^{m+1}), \mat D(U^{m+1}) \right)
\nonumber \\ & \hspace{2cm}
\leq \tfrac12\,(I^m_0\,\rho^{m-1}\,\vec I^m_2\,\vec U^{m}, 
\vec I^m_2\,\vec U^{m}) + \gamma_0\, \mathcal{H}^{d-1}(\Gamma^{m})
+ \tau\left( \rho^m\,\vec f^{m+1}_1 + \vec f^{m+1}_2, 
\vec U^{m+1} \right)\,.
\end{align*}
The same stability result, in the case {\rm (\ref{eq:Fconst})},
can be shown for the scheme {\rm (\ref{eq:GDa}--e)},
once again on replacing $\vec\kappa^m$ in {\rm (\ref{eq:GDa})} with
$\vec\kappa^{m+1}$.
\end{rem}

The above remark motivates our choice of time discretizations in 
(\ref{eq:HGa}--d). As it does not appear possible to prove a stability result
similar to (\ref{eq:stabHG}) 
for the fully discrete scheme (\ref{eq:HGa}--e) 
for general choices of $\gamma$ such as
(\ref{eq:gamma1},b), we prefer to use $\kappa^m$ in (\ref{eq:HGa}) rather than
$\kappa^{m+1}$, which simplifies the existence and uniqueness proof, as well as
the solution procedure.

\begin{rem} \label{rem:ffbp}
For ease of presentation we have assumed so far that the number of vertices,
$K_\Gamma$, and the number of elements, $J_\Gamma$, of the discrete interface
$\Gamma^m$ remain constant over time. However, it is a simple matter to allow
for a localized refinement procedure as employed in \cite{fluidfbp}. Here any
newly introduced basis function for $\Gamma^{m+1}$, say, needs to be traced
back to $\Gamma^m$ so that {\rm (\ref{eq:HGe})}, and similarly 
{\rm (\ref{eq:GDe})}, remain well-defined. 
\end{rem}

\setcounter{equation}{0}
\section{Numerical results}  \label{sec:6}

For details on the assembly of the linear system arising at each time step
of (\ref{eq:HGa}--e), as well as
details on the adaptive mesh refinement algorithm and the solution procedure,
we refer to \cite{fluidfbp}. The main new ingredient is (\ref{eq:HGe}), which
decouples from (\ref{eq:HGa}--d) and so is straightforward to solve. An
analogous comment holds for the scheme (\ref{eq:GDa}--e).
We recall from \cite{fluidfbp} that for the bulk mesh adaptation we
use a strategy that results in a fine mesh size
$h_f$ around $\Gamma^m$ and a coarse mesh size $h_c$ 
further away from it. 
Here
$h_{f} = \frac{2\,\min\{H_1,H_2\}}{N_{f}}$ and 
$h_{c} =  \frac{2\,\min\{H_1,H_2\}}{N_{c}}$
are given by two integer numbers $N_f >  N_c$, where we assume from now on that
$\Omega$ is given by $\times_{i=1}^d (-H_i,H_i)$. 
We remark that we implemented our scheme with the help of 
the finite element toolbox ALBERTA, see \cite{Alberta}.

For the scheme (\ref{eq:HGa}--e) we fix $\epsilon = 10^{-5}$, and in all our
numerical experiments presented in this section
the discrete surfactant concentration $\Psi^m$ remained above 
$\epsilon$ throughout the evolution, so that $\gamma_\epsilon(\Psi^m) =
\gamma(\Psi^m)$, recall (\ref{eq:geps}). 
Unless otherwise stated we use the linear equation of state (\ref{eq:gamma1}) 
for the surface tension, and for the numerical simulations without surfactant 
we set $\beta = 0$ in (\ref{eq:gamma1}). We set $\Psi^0 = \psi_0 = 1$, unless
stated otherwise.
In addition, we employ the lowest order
Taylor--Hood element P2--P1 in all computations and
set $\vec U^0 = \vec I^0_2\,\vec u_0$, where $\vec u_0 = \vec 0$
unless stated otherwise.
For the initial interface we always choose a circle/sphere of radius $R_0$ 
and set
$\kappa^0 = -\frac{d-1}{R_0}$ for the scheme (\ref{eq:HGa}--e). For the scheme
(\ref{eq:GDa}--e) we let $\vec\kappa^0 \in \Vhz$ be the solution of 
(\ref{eq:GDd}) with $m$ and $m+1$ replaced by zero.
To summarize the discretization parameters we use the shorthand notation
$n\,{\rm adapt}_{k,l}$ from \cite{fluidfbp}. 
The subscripts refer to the fineness of the spatial discretizations, i.e.\
for the set $n\,{\rm adapt}_{k, l}$ it holds that 
$N_f = 2^k$ and $N_c = 2^l$. For the case $d=2$ we have in addition that 
$K_\Gamma = J_\Gamma = 2^k$, while for $d=3$ it holds that
$(K_\Gamma, J_\Gamma) = (770, 1536), (1538, 3072), (3074, 6144)$ for 
$k = 5,6,7$. 
Finally, the uniform time step size 
for the set $n\,{\rm adapt}_{k,l}$ is given by $\tau = 10^{-3} / n$,
and if $n=1$ we write ${\rm adapt}_{k, l}$.

\subsection{Convergence experiments for convection diffusion equation}
In this subsection we test the two approximations
(\ref{eq:GDc},e) and (\ref{eq:HGc}--e) for the convection diffusion equation
(\ref{eq:1surf}), in a situation where the evolution of the surface 
$\Gamma(t)$ is given. In particular, we perform convergence experiments for 
the true solution from the Appendix; that is, 
$\psi(\vec z, t) = e^{-6\,t}\,z_1\,z_2$ is fixed on the moving ellipsoid 
$\Gamma(t)$ with time dependent $x_1$-axis.
To this end, we replace $\vec U^{m+1}$ in (\ref{eq:GDc}) and (\ref{eq:HGc},e) 
with $\vec u(\cdot, t_{m+1})$ as defined in (\ref{eq:trueu}), and set $\Ds=1$.
In addition, we add the term
\begin{equation*} 
\left\langle f^{m+1}_\Gamma, \chi^{m+1}_k\right\rangle_{\Gamma^{m+1}}^h
\end{equation*}
to the right hand sides of (\ref{eq:GDe}) and (\ref{eq:HGe}), where
$f^{m+1}_\Gamma \in \Whp$ is defined such that
$$
f^{m+1}_\Gamma(\vec q^{m+1}_k) = f_\Gamma ( 
\vec\Pi_{\Gamma(t_{m+1})}\,\vec q^{m+1}_k, t_{m+1}) 
\qquad k = 1,\ldots,K_\Gamma\,,
$$
with $f_\Gamma$ given as in (\ref{eq:truef}), and with 
$\vec\Pi_{\Gamma(t)} : \R^d \to \Gamma(t)$ denoting the orthogonal projection
onto $\Gamma(t)$ for $t \in [0,T]$. 
In practice this projection can be computed with the help of
a Newton iteration. 
In Tables~\ref{tab:conv2dL2} and \ref{tab:conv3dL2} we report
on the error
$$
\LerrorPsipsi := \left[ \sum_{m=1}^{M} \tau
\left\langle [\Psi^m - \psi (\cdot, t_m) \circ \vec\Pi_{\Gamma(t_m)}]^2,
1 \right\rangle_{\Gamma^{m}}^h \right]^\frac12
$$
for convergence experiments for $d=2$ and $d=3$, respectively. Here we choose
the time interval $[0,T]$ with $T=1$, 
and for the uniform time step size we take 
$\tau = h_0^2$, where $h_0$ denotes the maximal element diameter of $\Gamma^0$.
Of course, for the last time step we use the time step size 
$T - t_{M-1} = T - (M-1)\,\tau$.
\begin{table}
\center
\begin{tabular}{|c|c|c|}
\hline
$h_0$ & (\ref{eq:GDc},e) & (\ref{eq:HGc}--e) \\
\hline 
3.9018e-01 & 5.9569e-03 & 6.1760e-03 \\ 
1.9603e-01 & 2.4356e-04 & 2.4544e-04 \\
9.8135e-02 & 2.1006e-04 & 2.1197e-04 \\
4.9082e-02 & 9.0700e-06 & 9.1626e-06 \\
2.4543e-02 & 1.3328e-06 & 1.3469e-06 \\
\hline
\end{tabular}
\caption{The errors $\LerrorPsipsi$ for the convergence experiment for $d=2$.}
\label{tab:conv2dL2}
\end{table}%
\begin{table}
\center
\begin{tabular}{|c|c|c|}
\hline
$h_0$ & (\ref{eq:GDc},e) & (\ref{eq:HGc}--e) \\
\hline 
7.6537e-01 & 3.1233e-02 & 1.7760e-02 \\ 
4.0994e-01 & 2.6612e-03 & 3.1695e-03 \\ 
2.0854e-01 & 4.1570e-04 & 4.2492e-04 \\ 
1.0472e-01 & 2.1768e-05 & 2.1966e-05 \\ 
5.2416e-02 & 6.0305e-06 & 6.0785e-06 \\ 
\hline
\end{tabular}
\caption{The errors $\LerrorPsipsi$ for the convergence experiment for $d=3$.}
\label{tab:conv3dL2}
\end{table}%
We observe that both schemes show very similar errors, indicating a convergence
order of at least $\mathcal{O}(h_0^2)$.

\subsection{Numerical simulations in 2d} \label{sec:61}

In this section we consider some numerical simulations for two-phase flow with
insoluble surfactant in two space dimensions. We begin with a comparison
between the schemes (\ref{eq:GDa}--e) and (\ref{eq:HGa}--e) for a rising bubble
experiment that is motivated by the benchmark problems in \cite{HysingTKPBGT09}
for two-phase Navier--Stokes flow.

\subsubsection{Rising bubble benchmark problem 1} \label{sec:612}

We use the setup described in \cite{HysingTKPBGT09}, see Figure~2 there;
i.e.\ $\Omega = (0,1) \times (0,2)$ with 
$\partial_1\Omega = [0,1] \times \{0,2\}$ and 
$\partial_2\Omega = \{0,1\} \times (0,2)$.
Moreover, $\Gamma_0 = \{\vec z \in \R^2 : |\vec z - (\frac12, \frac12)^T| =
\frac14\}$.
The physical parameters from the test case 1 in 
\citet[Table~I]{HysingTKPBGT09}, in the absence of surfactant, 
are given by
\begin{equation} \label{eq:Hysing1}
\rho_+ = 1000\,,\quad \rho_- = 100\,,\quad \mu_+ = 10\,,\quad \mu_- = 1\,,\quad
\gamma_0 = 24.5\,,\quad \vec f_1 = -0.98\,\vec\ek_d\,,\quad 
\vec f_2 = \vec 0\,,
\end{equation}
where, here and throughout, 
$\{\vec \ek_j\}_{j=1}^d$ denotes the standard basis in $\R^d$.
The time interval chosen for the simulation is $[0,T]$ with $T=3$.
For the surfactant problem we choose the parameters $\Ds = 0.1$ and 
(\ref{eq:gamma1}) with $\beta = 0.5$.

We start with a simulation for the scheme (\ref{eq:GDa}--e), using 
the discretization parameters adapt$_{7, 3}$.
The results can be seen on the left of
Figure~\ref{fig:dziuk}. Two things are immediately evident. Firstly, the area
of the inner phase is not conserved. In fact, in this computation the relative
area loss for the inner phase is 62\%. 
And secondly, we see that the vertices of the approximation
$\Gamma^m$ are transported, similarly to the surfactant, with the fluid flow.
This means that many vertices can
be found at the bottom of the bubble, with hardly any vertices left at the top.
The second behaviour can be improved by allowing local mesh refinements on
$\Gamma^m$, recall Remark~\ref{rem:ffbp}.
In particular, we refine an element $\sigma^m$ on $\Gamma^m$ 
whenever $\mathcal{H}^{d-1}(\sigma^m) > \frac74\,\max_{j=1,\ldots, J_\Gamma}
\mathcal{H}^{d-1}(\sigma^0_j)$.
Then the interface remains well resolved, and the final number of elements is
$J^M_\Gamma = 252 > 128 = J^0_\Gamma$. However, coalescence of
vertices can still be observed at the bottom of the bubble, see the plot on the
right of Figure~\ref{fig:dziuk}.
\begin{figure}
\center
\includegraphics[angle=-90,width=0.4\textwidth]{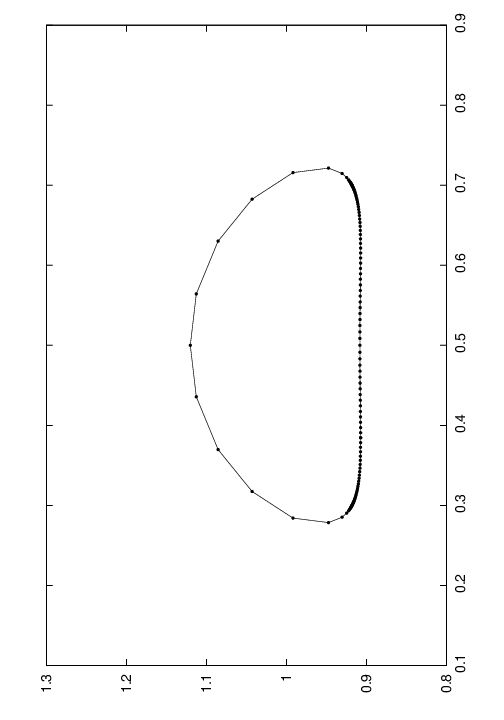}
\includegraphics[angle=-90,width=0.4\textwidth]{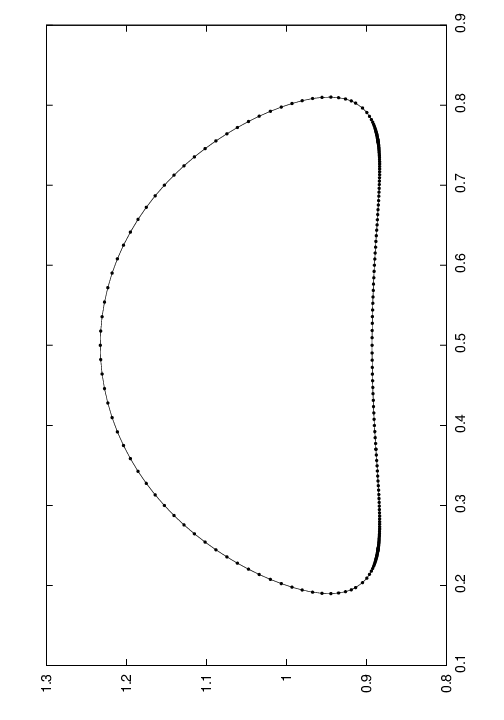}
\caption{(adapt$_{7,3}$) 
Vertex distributions for the final bubbles for the benchmark problem 1 at time 
$T=3$ for the scheme (\ref{eq:GDa}--e) without local refinement on 
$\Gamma^m$ (left), and with local refinement (right).}
\label{fig:dziuk}
\end{figure}%
We remark that for the latter computation the area of the inner phase 
decreases by $14\%$. For completeness we note that this dramatic area loss is
connected to mass lumping being employed on the right hand side of 
(\ref{eq:GDc}). To visualize this effect, we repeat the above computations now
for $\langle \vec U^{m+1}, \vec\chi \rangle_{\Gamma^m}^h$ in (\ref{eq:GDc}) 
replaced by $\langle \vec U^{m+1}, \vec\chi \rangle_{\Gamma^m}$. 
The semidiscrete variant of this new approximation then no longer satisfies 
the stability result in Theorem~\ref{thm:stabGD}. However, in practice this
approximation appears to perform much better, with the relative area loss of
the inner phase now down to 1.4\% for the simulation without local refinement.
The simulation with local refinement leads to coalescence of vertices and 
a clear loss of symmetry, which is of course unphysical, 
see Figure~\ref{fig:dziuknoni}.
\begin{figure}
\center
\includegraphics[angle=-90,width=0.4\textwidth]{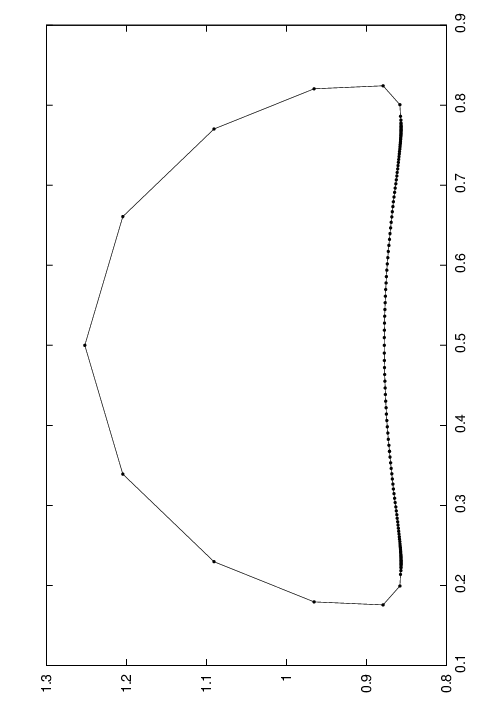}
\includegraphics[angle=-90,width=0.4\textwidth]{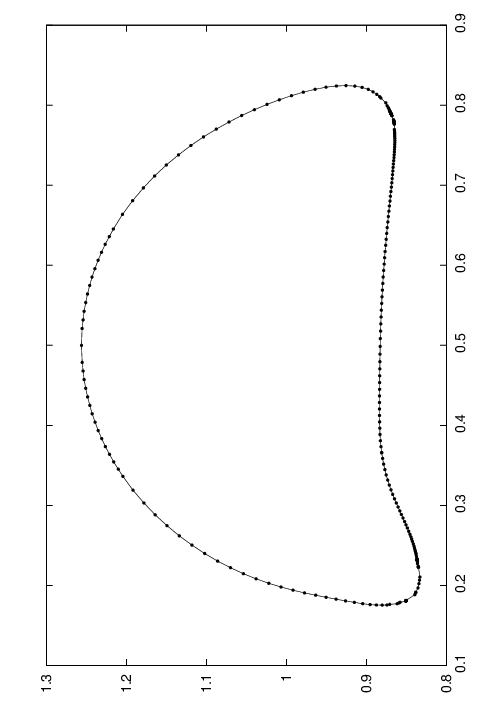}
\caption{(adapt$_{7,3}$) 
Vertex distributions for the 
final bubbles for the benchmark problem 1 at time $T=3$ for 
a variant of the scheme (\ref{eq:GDa}--e) 
without local refinement on $\Gamma^m$ (left)
and with local refinement (right). The loss of symmetry is caused by
coalescence of vertices.}
\label{fig:dziuknoni}
\end{figure}%

The same computation for our preferred scheme (\ref{eq:HGa}--e), where no local
refinements need to be performed because the tangential movement of vertices
yields an almost equidistributed approximation of $\Gamma^m$, can be seen in
Figure~\ref{fig:bgn}, where we compare the run with $\beta = 0.5$ also to the 
case of constant surface tension, i.e.\ $\beta = 0$.
We remark that for these computations the areas of the two phases,
as well as the total surfactant mass on $\Gamma^m$, were conserved. 
\begin{figure}
\center
\includegraphics[angle=-90,width=0.4\textwidth]{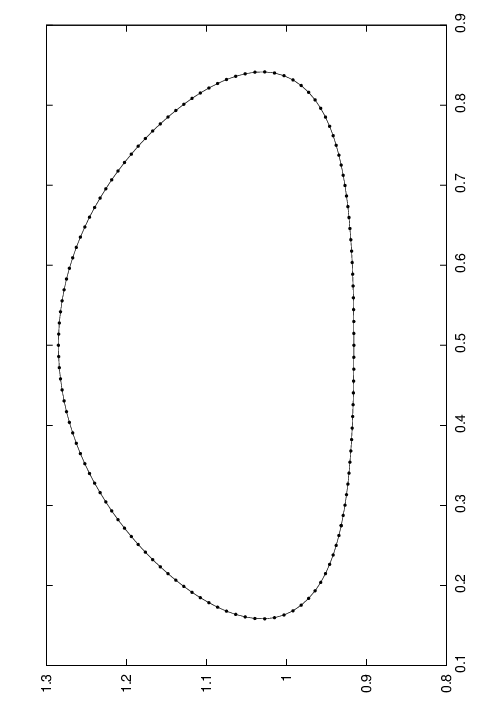}
\includegraphics[angle=-90,width=0.4\textwidth]{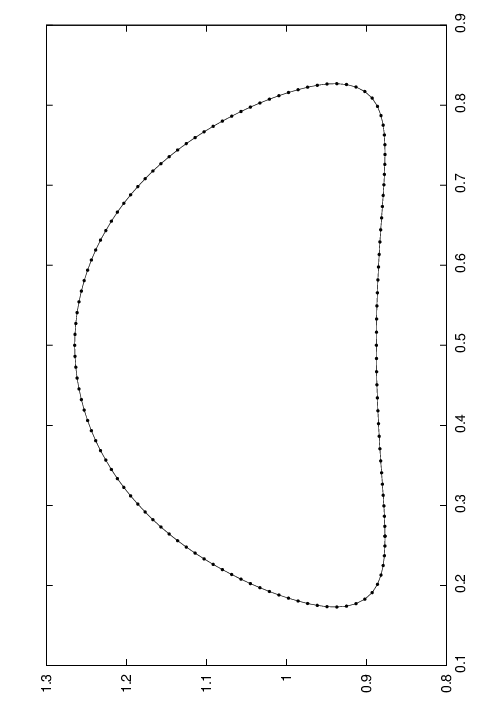}
\caption{(adapt$_{7,3}$)
Vertex distributions for the final bubble for the benchmark problem 1 at 
time $T=3$ for the scheme (\ref{eq:HGa}--e). On the left the computation with 
$\beta=0$, on the right with $\beta = 0.5$.}
\label{fig:bgn}
\end{figure}%

In Figure~\ref{fig:dziukbgn} we show the surfactant concentrations $\Psi^M$
on the final bubble for the two schemes (\ref{eq:GDa}--e) and 
(\ref{eq:HGa}--e), where in the computation for the former scheme we employ
local mesh refinements. We observe that the numerical results
are in rough agreement, apart from the smaller bubble for the scheme
(\ref{eq:GDa}--e) because of the loss of area for the inner phase. 
We also show a plot of the discrete surface energy
$\langle F(\Psi^m), 1 \rangle_{\Gamma^m}^h$, where for (\ref{eq:HGa}--e) 
it holds that $\langle F_\epsilon(\Psi^m), 1 \rangle_{\Gamma^m}^h
= \langle F(\Psi^m), 1 \rangle_{\Gamma^m}^h$ throughout the evolution. 
Here it can be seen that the
plots for the two approximations differ significantly, most probably because of
the area loss for the scheme (\ref{eq:GDa}--e).
\begin{figure}
\center
\includegraphics[angle=-90,width=0.4\textwidth]{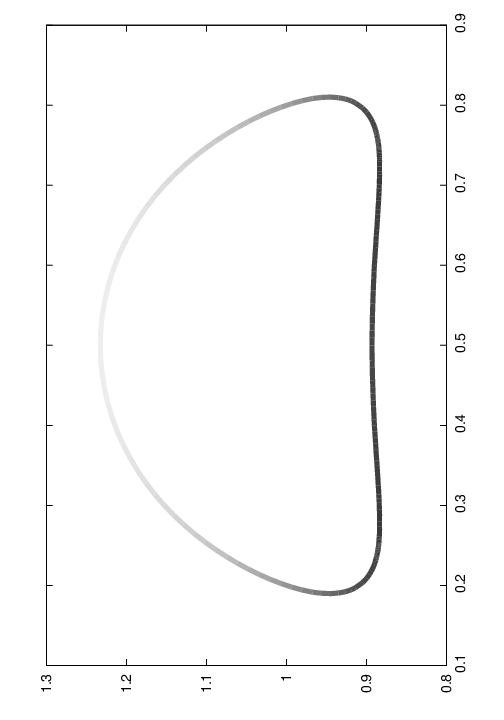}
\includegraphics[angle=-90,width=0.4\textwidth]{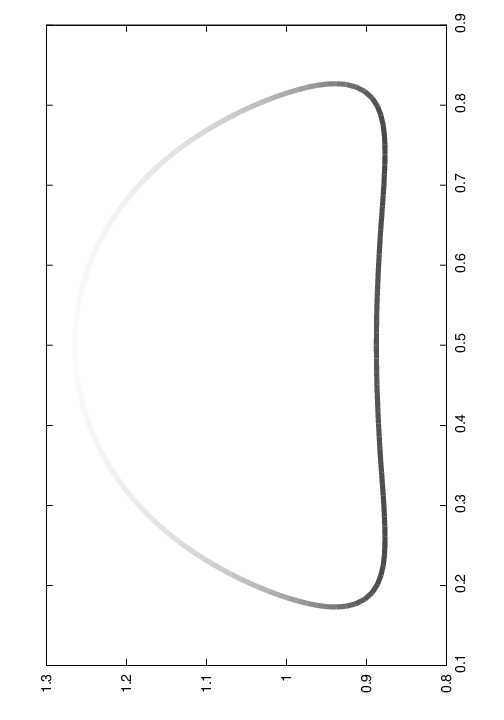}
\includegraphics[angle=-90,width=0.4\textwidth]{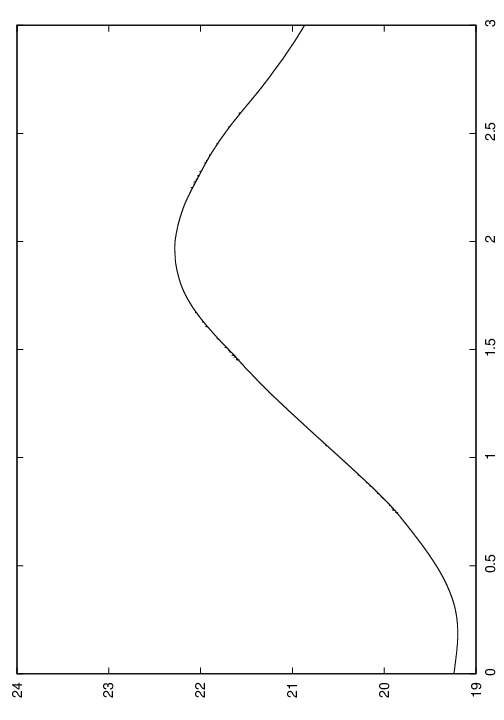}
\includegraphics[angle=-90,width=0.4\textwidth]{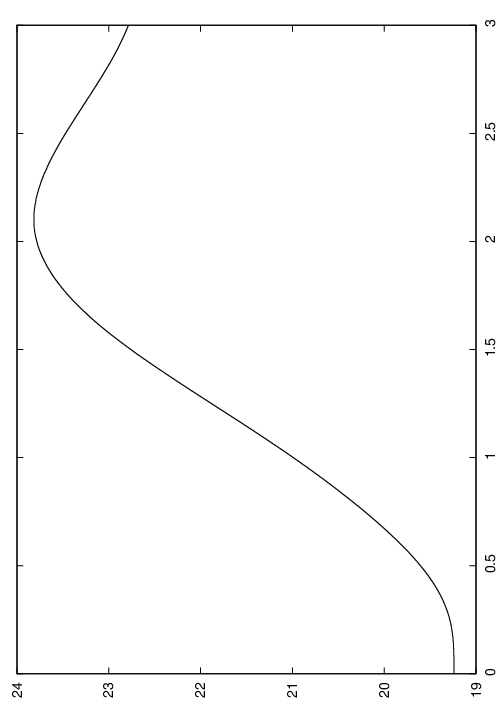}
\caption{(adapt$_{7,3}$) 
The surfactant concentration on the final bubbles for the benchmark problem 1 
at time $T=3$ for the schemes (\ref{eq:GDa}--e) and (\ref{eq:HGa}--e). 
The grey scales linearly with the surfactant
concentration ranging from 0.4 (white) to 1.4 (black).
Below we present plots of $\langle F(\Psi^m), 1 \rangle_{\Gamma^m}^h$ over time
for the two schemes.
}
\label{fig:dziukbgn}
\end{figure}%

The poor mesh properties of the scheme (\ref{eq:GDa}--e), together with the
fact that the volume of the two phases is in general not conserved, mean that
this scheme is not very practical. Of course, the same applies to the scheme
from Remark~\ref{rem:newGD}. It is for this reason that from now on we only
consider numerical experiments for the scheme (\ref{eq:HGa}--e). 

The parameters in (\ref{eq:Hysing1}) were proposed in 
\citet[Table~I]{HysingTKPBGT09} in order to define a test case for two-phase
flow, in the absence of surfactant, for which benchmark computations can be 
performed. We now report on these benchmark quantities also in the presence of
surfactant. To this end, we recall from \cite{fluidfbp} our fully discrete
approximations for
the $x_2$-component of the bubble's centre of mass, 
the bubble's ``degree of circularity'' and the rise velocity:
\begin{equation} \label{eq:benchmarkm}
y_c^m = \frac1{\mathcal{L}^2(\Omega_-^m)}\,\int_{\Omega_-^m} x_2 \dL2\,,
\quad 
\strikec^m = 2\,[\pi\,\mathcal{L}^2(\Omega_-^m)]^\frac12\,
[\mathcal{H}^{1}(\Gamma^m)]^{-1}\,, \quad 
V^m_c = \frac{(\rho^m_-\,\vec U^m, \vec \ek_d)}{(\rho^m_-,1)}\,,
\end{equation}
where $\rho^m_-\in S^m_0$ is defined as in (\ref{eq:rhoma}) 
but with $\rho_+$ replaced by zero.
Finally, we also define the relative overall area/volume loss as
$$
\Mloss = 
\frac{\mathcal{L}^d(\Omega^0_-) - \mathcal{L}^d(\Omega^M_-)}
{\mathcal{L}^d(\Omega^0_-)}\,.
$$
In Table~\ref{tab:Hysing1} we report on these quantities for simulations with
and without surfactant for our preferred scheme
(\ref{eq:HGa}--e). 
Here we note that the numbers for the simulations without surfactant differ
slightly from the ones in \citet[Table~2]{fluidfbp}, because the finite element
approximations employed here is different, recall Remark~\ref{rem:fluidfbp}.
\begin{table}
\center
\begin{tabular}{l|r|r|r|r}
\hline
& adapt$_{5,2}$ & adapt$_{7,3}$ & 
2\,adapt$_{9,4}$ & 5\,adapt$_{11,5}$\\
\hline 
$\Mloss$ & 
  0.0\% &  0.0\% &  0.0\% &  0.0\% \\
$\strikec_{\min}$ & 
 0.9135 & 0.9069 & 0.9034 & 0.9022 \\
$t_{\strikec = \strikec_{\min}}$ & 
 2.0760 & 1.9420 & 1.9105 & 1.9028 \\
$V_{c,\max}$ & 
 0.2477 & 0.2415 & 0.2413 & 0.2420 \\
$t_{V_c = V_{c,\max}}$ & 
 0.9470 & 0.9360 & 0.9255 & 0.9698 \\
$y_c(t=3)$ & 
 1.0906 & 1.0822 & 1.0814 & 1.0815 \\
\hline
\end{tabular}
\begin{tabular}{l|r|r|r|r}
\hline
& adapt$_{5,2}$ & adapt$_{7,3}$ & 
2\,adapt$_{9,4}$ & 5\,adapt$_{11,5}$\\ 
\hline 
$\Mloss$ & 
  0.0\% &  0.0\% &  0.0\% &  0.0\% \\
$\strikec_{\min}$ & 
 0.8779 & 0.8715 & 0.8681 & 0.8669 \\
$t_{\strikec = \strikec_{\min}}$ & 
 2.1330 & 2.0710 & 2.0550 & 2.0500 \\
$V_{c,\max}$ & 
 0.2279 & 0.2243 & 0.2236 & 0.2237 \\
$t_{V_c = V_{c,\max}}$ & 
 1.0070 & 0.9040 & 0.9010 & 0.8710 \\
$y_c(t=3)$ & 
 1.0423 & 1.0449 & 1.0467 & 1.0473 \\
\hline
\end{tabular}
\caption{Some quantitative results for the benchmark problem 1. 
Without surfactant (top) and with surfactant (bottom).}
\label{tab:Hysing1}
\end{table}%
In what follows we
present some visualizations of the numerical results for the runs with 
the discretization parameters 5\,adapt$_{11,5}$.
A plot of $\Gamma^M$ can be seen in Figure~\ref{fig:bubble}, while the time
evolution of the circularity, the centre of mass and the rise velocity
are shown in Figures~\ref{fig:circularity} and \ref{fig:comrise}. 
\begin{figure}
\center
\includegraphics[angle=-90,width=0.6\textwidth]{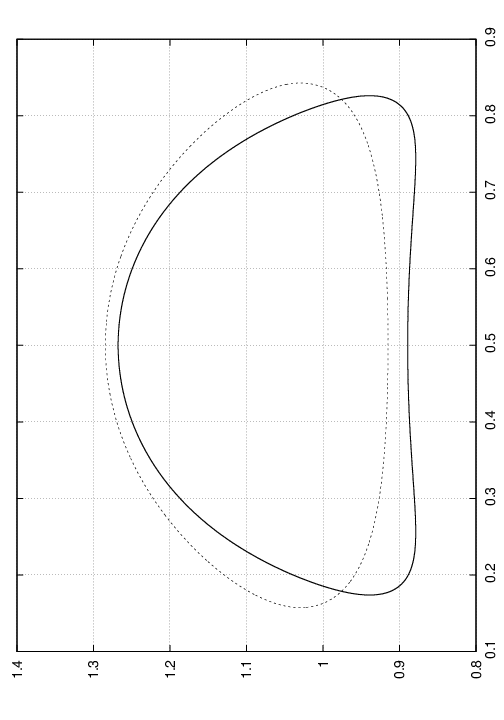}
\caption{(5\,adapt$_{11,5}$)
The final bubble with surfactant for the benchmark problem 1 at time $T=3$. 
The clean bubble is shown dashed.}
\label{fig:bubble}
\end{figure}%
\begin{figure}
\center
\includegraphics[angle=-90,width=0.45\textwidth]{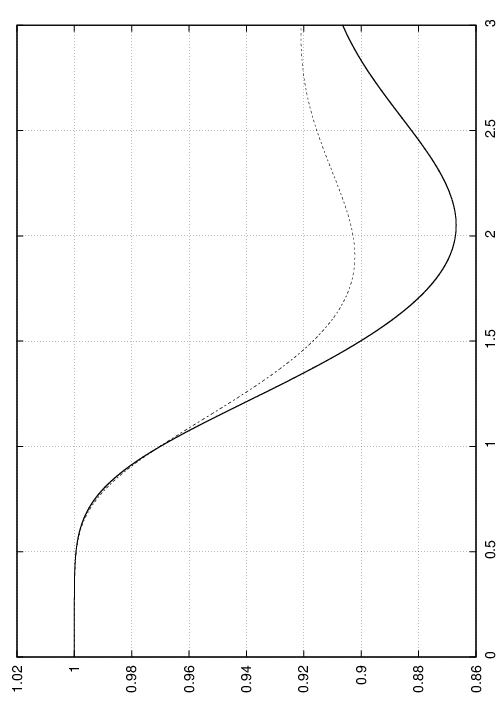}
\caption{(5\,adapt$_{11,5}$)
Circularity of the surfactant bubble for the benchmark problem 1. 
The dashed line is for the clean bubble.}
\label{fig:circularity}
\end{figure}%
\begin{figure}
\center
\includegraphics[angle=-90,width=0.45\textwidth]{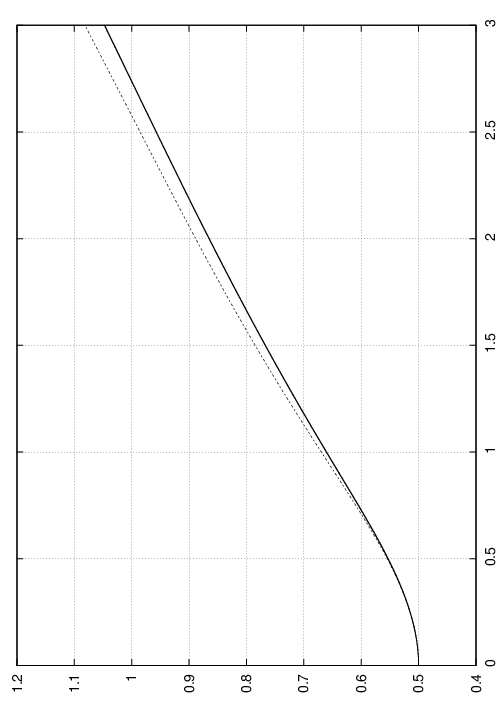}
\includegraphics[angle=-90,width=0.45\textwidth]{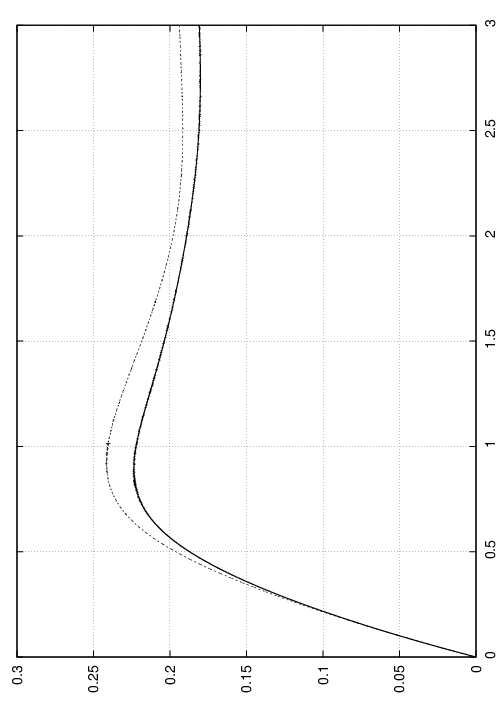}
\caption{(5\,adapt$_{11,5}$)
Centre of mass and rise velocity for the surfactant bubble for the benchmark
problem 1. The dashed lines are for the clean bubble.}
\label{fig:comrise}
\end{figure}%

\subsubsection{Rising bubble benchmark problem 2} \label{sec:613}

In a second set of benchmark computations, we fix
\begin{equation} \label{eq:Hysing2}
\rho_+ = 1000\,,\quad \rho_- = 1\,,\quad \mu_+ = 10\,,\quad \mu_- = 0.1\,,\quad
\gamma_0 = 1.96\,,\quad \vec f_1 = -0.98\,\vec\ek_d\,,\quad \vec f_2 = \vec 0\,,
\end{equation}
as in test case 2 in \citet[Table~I]{HysingTKPBGT09}.
For the surfactant problem we again let $\Ds = 0.1$ and let $\beta = 0.5$
in (\ref{eq:gamma1}).
In Table~\ref{tab:Hysing2} we report on some benchmark quantities for 
simulations with and without surfactant for our preferred scheme
\mbox{(\ref{eq:HGa}--e)}. 
Here we note that 
in contrast to the experiments in \S\ref{sec:612}, there is little difference
between the numbers for the runs with and without surfactant. This is because
in the simulations for (\ref{eq:Hysing2}) 
the large values of $\frac{\rho_+}{\rho_-}$ and
$\frac{\mu_+}{\mu_-}$ dominate the evolution. In particular, they lead to 
elongated fingers developing at the bottom of the rising bubble which means
that there is a significant growth in the overall interface length. In order 
to account for this growth, 
we locally refine $\Gamma^m$ in all the
simulations for the parameters as in (\ref{eq:Hysing2}). Here, similarly to the
experiment on the right of Figure~\ref{fig:dziuk},
we refine an element $\sigma^m$ on $\Gamma^m$ 
whenever $\mathcal{H}^{d-1}(\sigma^m) > \frac74\,\max_{j=1,\ldots, J_\Gamma}
\mathcal{H}^{d-1}(\sigma^0_j)$.
\begin{table}
\center
\begin{tabular}{l|r|r|r}
\hline
& adapt$_{5,2}$ & adapt$_{7,3}$ & 2\,adapt$_{9,4}$ \\
\hline 
$\Mloss$ & 
  0.0\% &  0.0\% &  0.0\% \\
$\strikec_{\min}$ & 
 0.5890 & 0.5198 & 0.5165 \\
$t_{\strikec = \strikec_{\min}}$ & 
 3.0000 & 3.0000 & 3.0000 \\
$V_{c,\max1}$ & 
 0.2584 & 0.2480 & 0.2489 \\
$t_{V_c = V_{c,\max1}}$ & 
 0.8800 & 0.7610 & 0.7295 \\
$V_{c,\max2}$ &
 0.2283 & 0.2305 & 0.2357 \\
$t_{V_c = V_{c,\max2}}$ & 
 2.0000 & 1.9510 & 2.0485 \\
$y_c(t=3)$ & 
 1.1275 & 1.1239 & 1.1319 \\
\hline
\end{tabular}
\begin{tabular}{l|r|r|r}
\hline
& adapt$_{5,2}$ & adapt$_{7,3}$ & 2\,adapt$_{9,4}$ \\
\hline 
$\Mloss$ & 
  0.0\% &  0.0\% &  0.0\% \\
$\strikec_{\min}$ & 
 0.5449 & 0.4996 & 0.4891 \\
$t_{\strikec = \strikec_{\min}}$ & 
 3.0000 & 3.0000 & 3.0000 \\
$V_{c,\max1}$ & 
 0.2565 & 0.2467 & 0.2476 \\
$t_{V_c = V_{c,\max1}}$ & 
 0.8830 & 0.7370 & 0.7395 \\
$V_{c,\max2}$ & 
 0.2283 & 0.2326 & 0.2391 \\
$t_{V_c = V_{c,\max2}}$ & 
 2.0070 & 2.0330 & 2.0830 \\
$y_c(t=3)$ & 
 1.1217 & 1.1197 & 1.1294 \\
\hline
\end{tabular}
\caption{Some quantitative results for the benchmark problem 2. 
Without surfactant (top) and with surfactant (bottom).}
\label{tab:Hysing2}
\end{table}%
In what follows we
present some visualizations of the numerical results for the runs with 
the discretization parameters 2\,adapt$_{9,4}$.
A plot of $\Gamma^M$ can be seen in Figure~\ref{fig:bubble2}, where we also
show the final surfactant concentration $\Psi^M$. Here we observe that most of
the surfactant has accumulated at the inner side walls of the lower part of
the bubble. It is worth pointing out that our numerical method has no
difficulties in computing the evolution of the advection-diffusion equation
on a highly deformed interface as seen in Figure~\ref{fig:bubble2}.
The time evolution of the circularity, the centre of mass and the rise velocity
are shown in Figures~\ref{fig:circularity2} and \ref{fig:comrise2}.
\begin{figure}
\center
\mbox{
\includegraphics[angle=-90,width=0.5\textwidth]{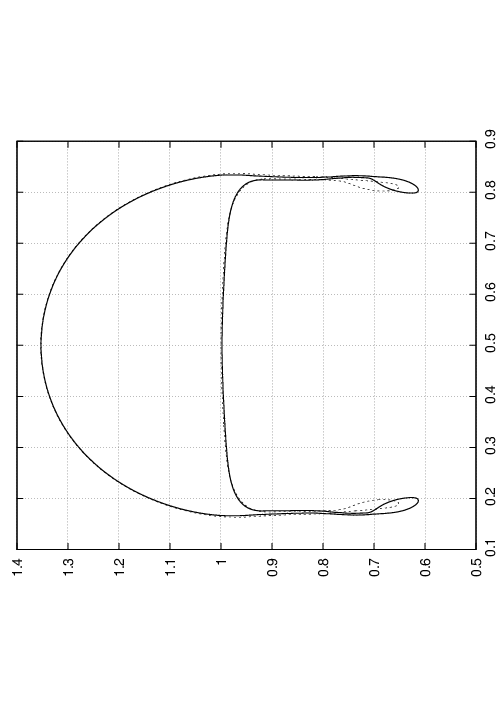}
\includegraphics[angle=-90,width=0.5\textwidth]{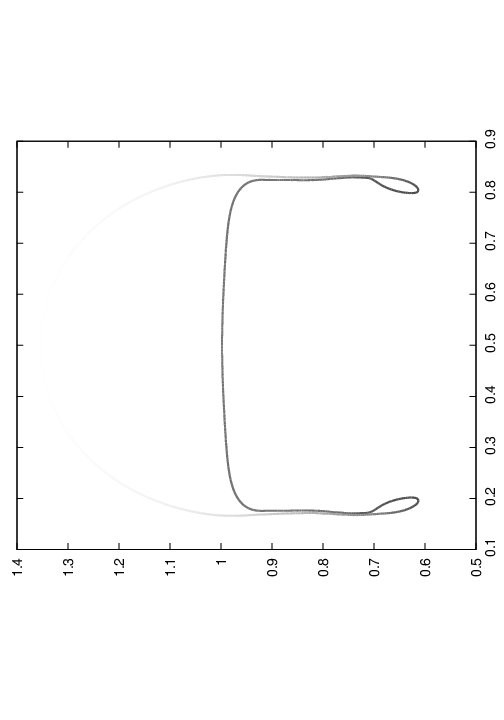}}
\caption{(2\,adapt$_{9,4}$)
The final bubble with surfactant for the benchmark problem 2 at time $T=3$, with the 
surfactant concentration on the right. 
The grey scales linearly with the surfactant concentration ranging from 
0.1 (white) to 0.9 (black). The dashed curve on the left represents the final
shape of the clean bubble.
}
\label{fig:bubble2}
\end{figure}%
\begin{figure}
\center
\includegraphics[angle=-90,width=0.45\textwidth]{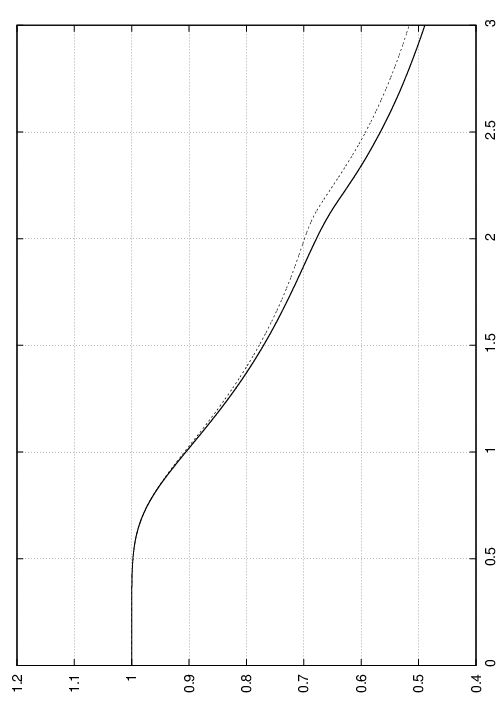}
\caption{(2\,adapt$_{9,4}$)
Circularity of the surfactant bubble for the benchmark problem 2. The dashed line is
for the clean bubble.}
\label{fig:circularity2}
\end{figure}%
\begin{figure}
\center
\includegraphics[angle=-90,width=0.45\textwidth]{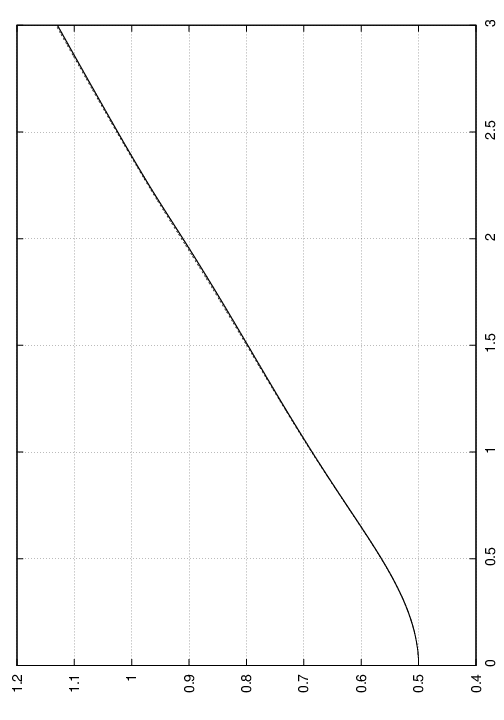}
\includegraphics[angle=-90,width=0.45\textwidth]{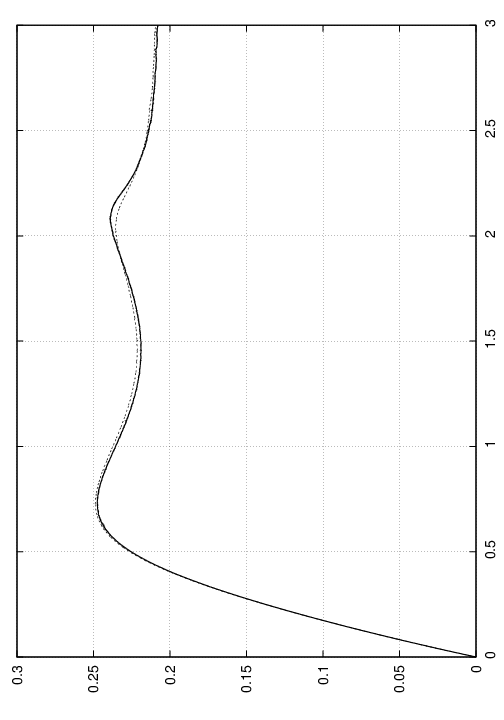}
\caption{(2\,adapt$_{9,4}$)
Centre of mass and rise velocity for the surfactant bubble for the benchmark
problem 2.
The dashed lines are for the clean bubble.}
\label{fig:comrise2}
\end{figure}%

\subsubsection{Bubble in shear flow} \label{sec:611}

In the literature on numerical methods for two-phase flow with insoluble
surfactant it is often common to consider shear flow experiments for an
initially circular bubble in order to study the effect of surfactants and of
different equations of state. In this subsection, we will perform such
simulations for our preferred scheme (\ref{eq:HGa}--e).
Here we consider the setup from \citet[Fig.~1]{LaiTH08}. In particular,
we let $\Omega = (-5,5) \times (-2,2)$ and
prescribe the inhomogeneous Dirichlet boundary condition
$\vec g(\vec z) = (\frac12\,z_2, 0)^T$ on $\partial\Omega = \partial_1\Omega$.
Moreover, $\Gamma_0 = \{\vec z \in \R^2 : |\vec z | = 1 \}$. The physical
parameters are given by
\begin{equation} \label{eq:Lai}
\rho_+ = \rho_- = 1\,, \quad \mu_+ = \mu_- = 0.1\,,\quad \gamma_0 = 0.2\,,
\quad \Ds = 0.1\,,\quad
\vec f = \vec 0\,,\quad \vec u_0 = \vec g\,.
\end{equation}
First we compare the evolutions for the linear equation of state 
(\ref{eq:gamma1}) for
(i) $\beta = 0$, (ii) $\beta = 0.25$ and (iii) $\beta = 0.5$.
Our numerical results appear to agree very well with the ones in
\citet[Fig.~1]{LaiTH08}; see Figure~\ref{fig:2dLai_nr} for more details.
On recalling (\ref{eq:benchmarkm}), we note that the ``circularities'' 
$\strikec^M$ of the final bubbles 
are given by 0.68, 0.59 and 0.51, respectively.
Moreover, we remark that for these simulations the relative overall area 
loss satisfies $|\Mloss| < 0.02\%$, and the same holds true for all of the
remaining numerical experiments in this subsection.
\begin{figure}
\center
\newcommand\localwidth{0.24\textwidth}
\includegraphics[angle=-90,width=\localwidth]{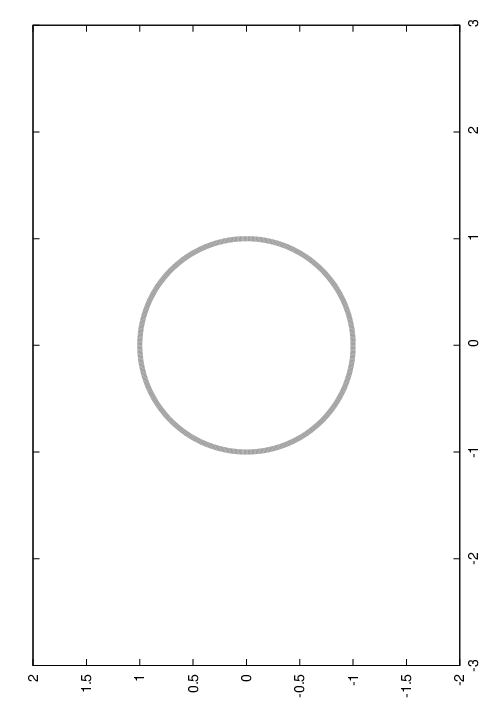}
\includegraphics[angle=-90,width=\localwidth]{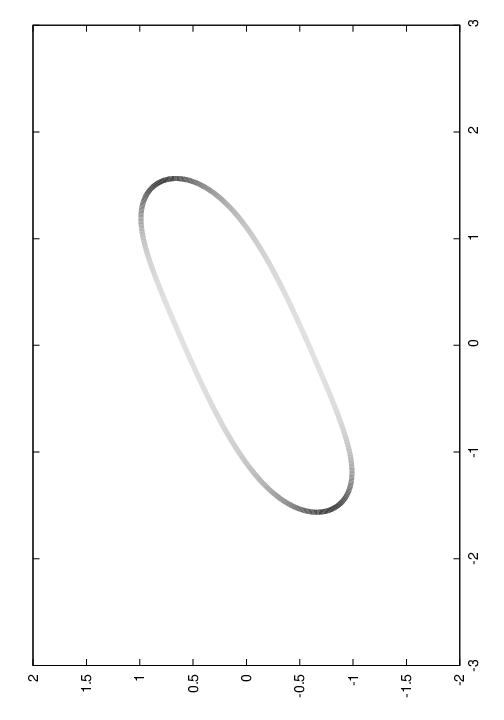}
\includegraphics[angle=-90,width=\localwidth]{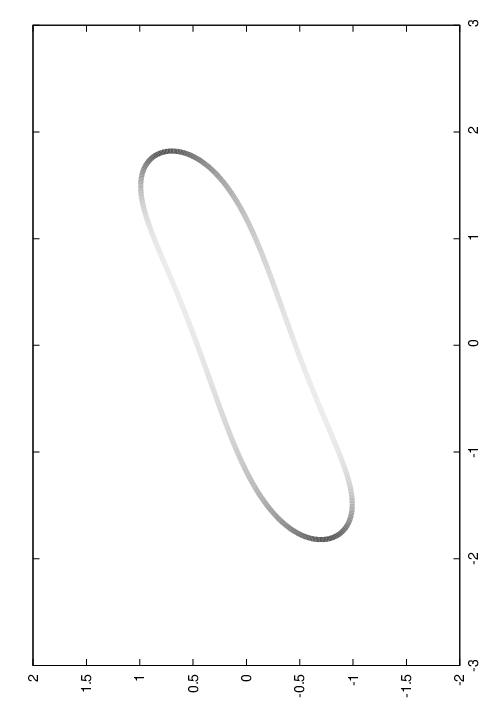}
\includegraphics[angle=-90,width=\localwidth]{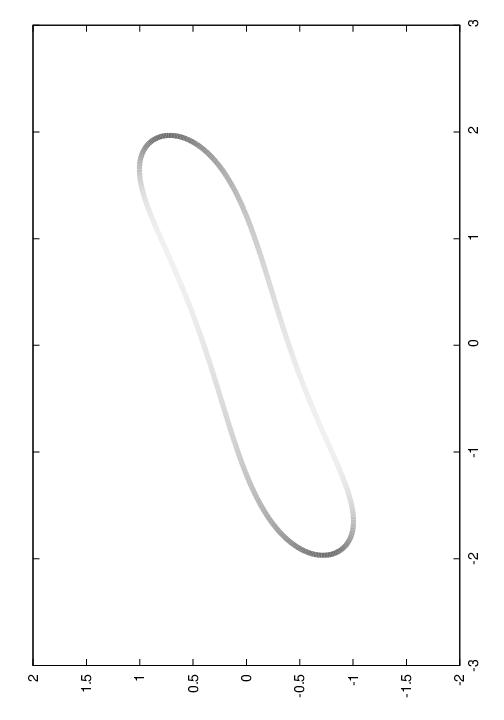}
\includegraphics[angle=-90,width=\localwidth]{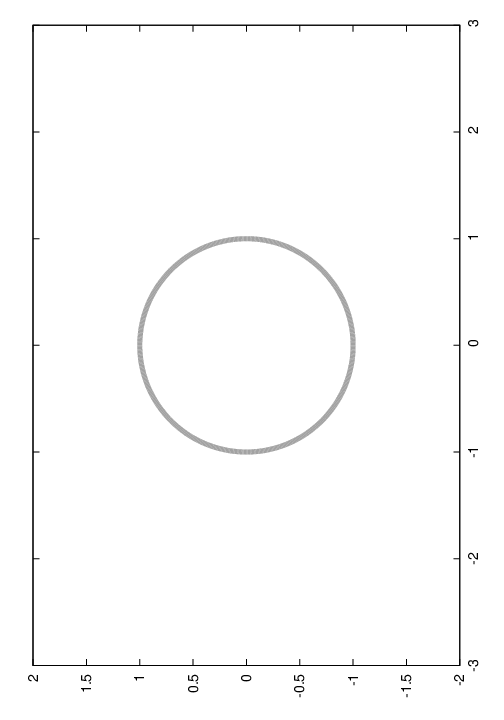}
\includegraphics[angle=-90,width=\localwidth]{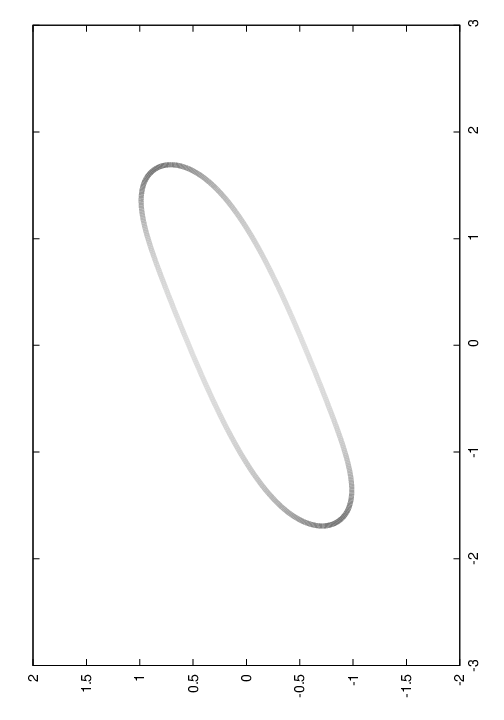}
\includegraphics[angle=-90,width=\localwidth]{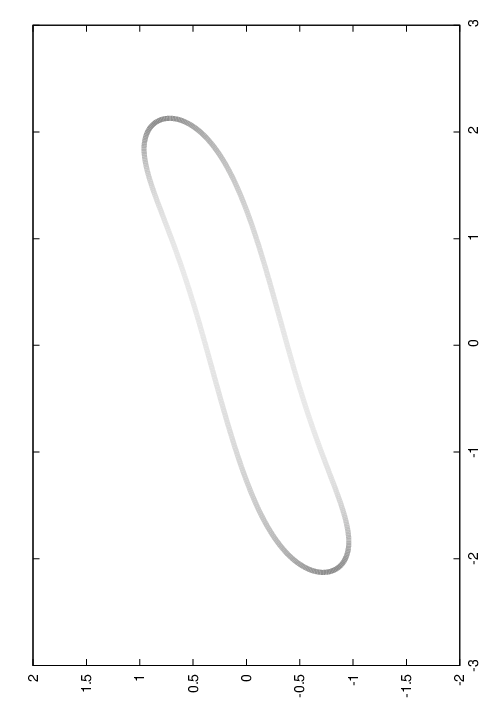}
\includegraphics[angle=-90,width=\localwidth]{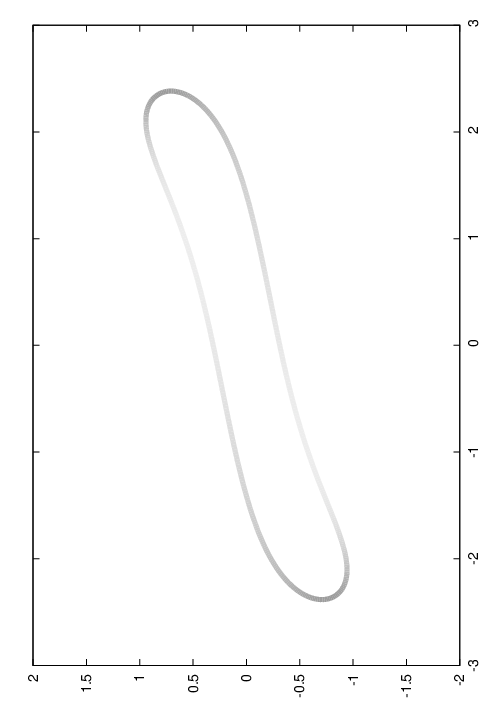}
\includegraphics[angle=-90,width=\localwidth]{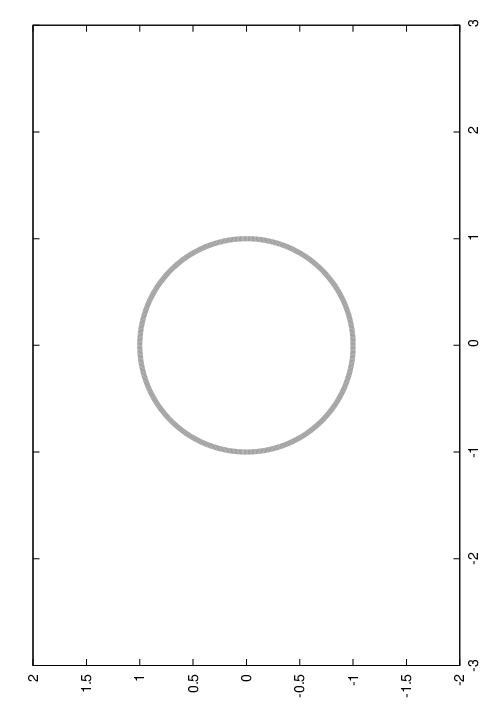}
\includegraphics[angle=-90,width=\localwidth]{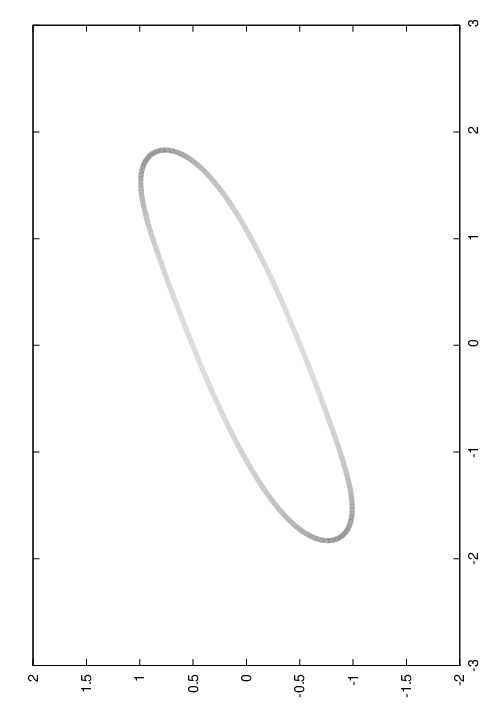}
\includegraphics[angle=-90,width=\localwidth]{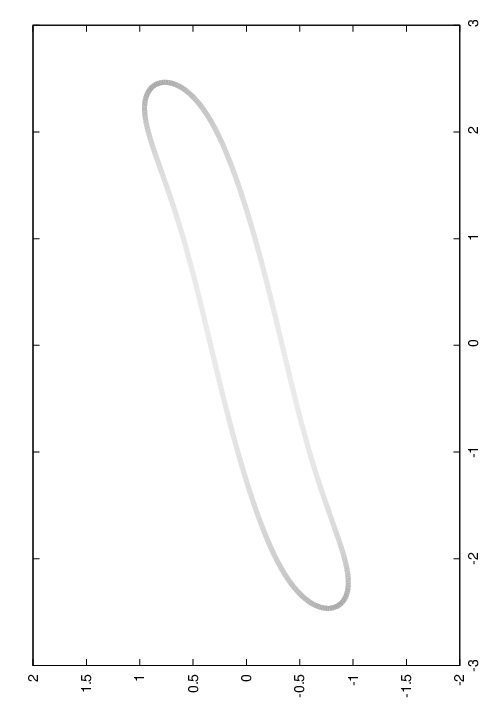}
\includegraphics[angle=-90,width=\localwidth]{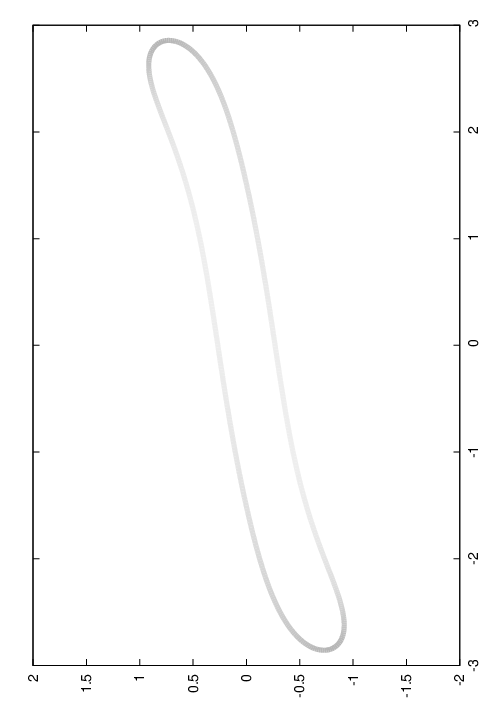}
\caption{(2\,adapt$_{9,4}$)
The time evolution of a drop in shear flow for (\ref{eq:gamma1}) with 
$\beta = 0$ (top), $\beta = 0.25$ (middle) and $\beta = 0.5$ (bottom). 
Plots are at times $t=0,\,4,\,8,\,12$. 
The grey scales linearly with the surfactant concentration ranging from 
0.2 (white) to 1.6 (black).
}
\label{fig:2dLai_nr}
\end{figure}%

In the next experiment we choose the nonlinear equation of state
(\ref{eq:gamma2}) with $\psi_\infty = \frac1\beta$; see also 
\citet[Fig.\ 6]{LaiTH08}. 
We show the evolutions of the drop for $\beta = 0.25$ and for $\beta = 0.5$
in
Figure~\ref{fig:2dLain_nr}. A detailed comparison of the final drop shapes for
the two equations of state (\ref{eq:gamma1},b) can be seen in 
Figure~\ref{fig:2dLaiLain_nr}. As expected, the difference between the 
simulations for the two equations of state are more pronounced for the larger
value of $\beta$.
\begin{figure}
\center
\newcommand\localwidth{0.24\textwidth}
\includegraphics[angle=-90,width=\localwidth]{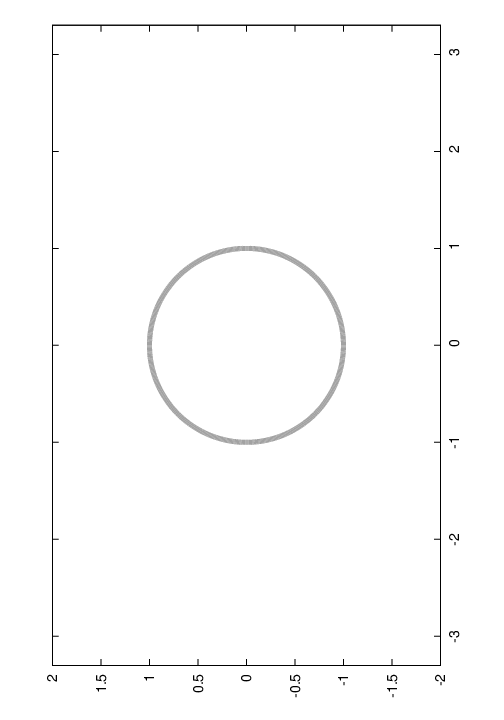}
\includegraphics[angle=-90,width=\localwidth]{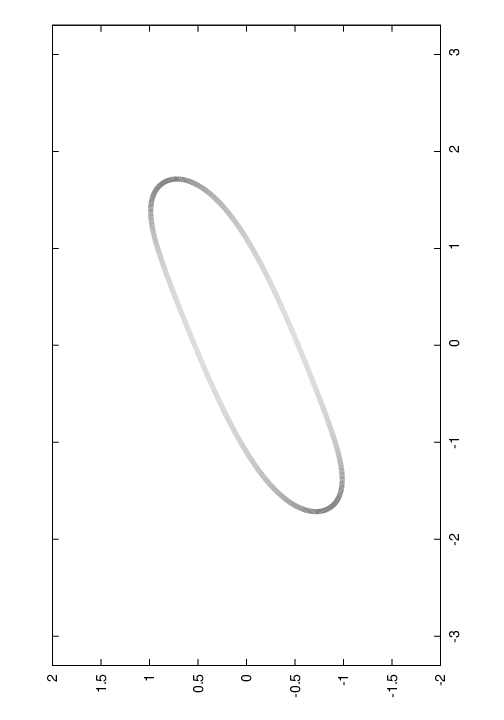}
\includegraphics[angle=-90,width=\localwidth]{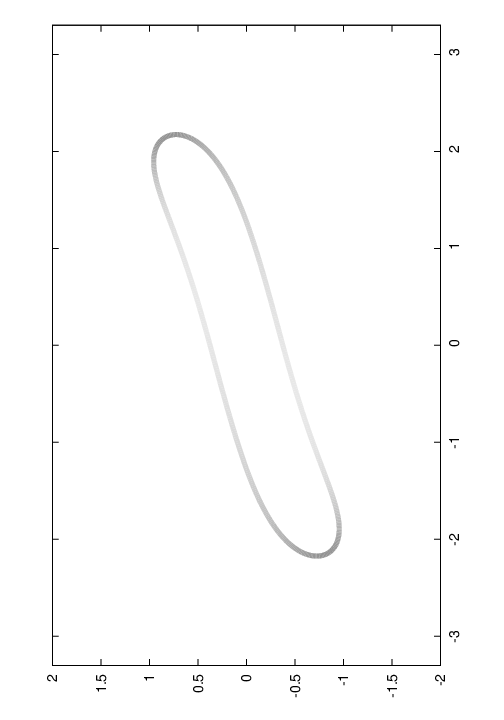}
\includegraphics[angle=-90,width=\localwidth]{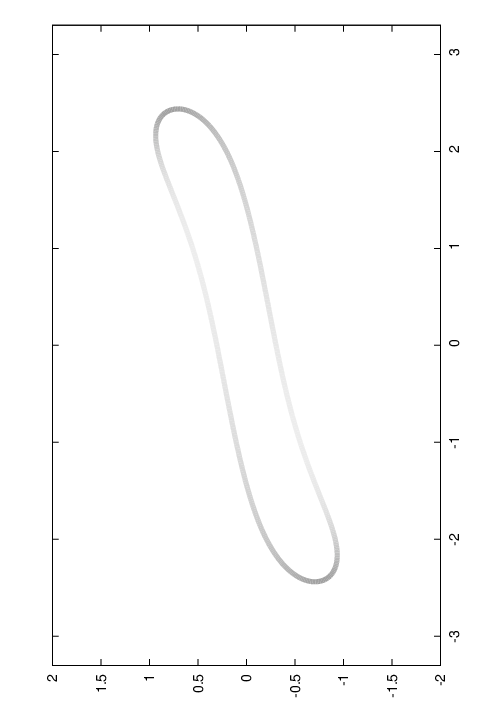}
\includegraphics[angle=-90,width=\localwidth]{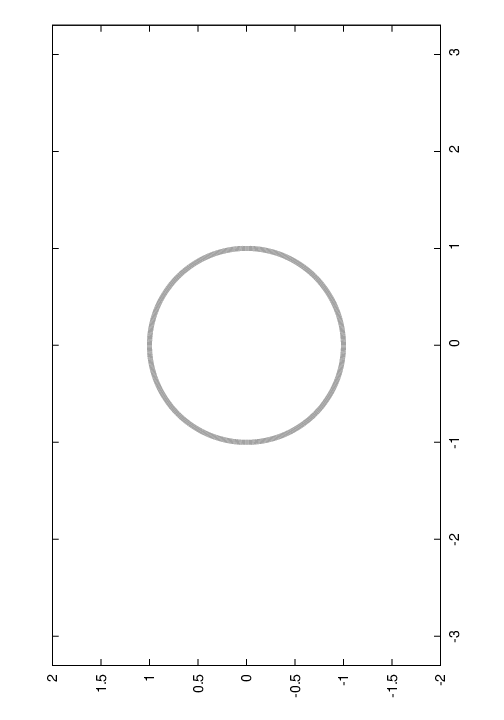}
\includegraphics[angle=-90,width=\localwidth]{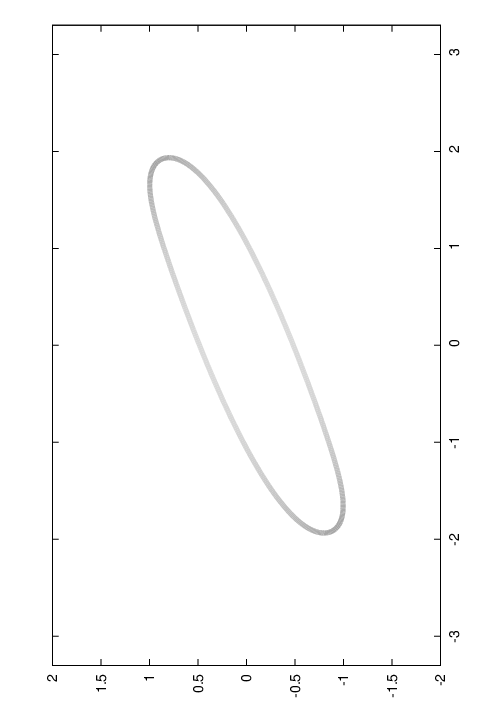}
\includegraphics[angle=-90,width=\localwidth]{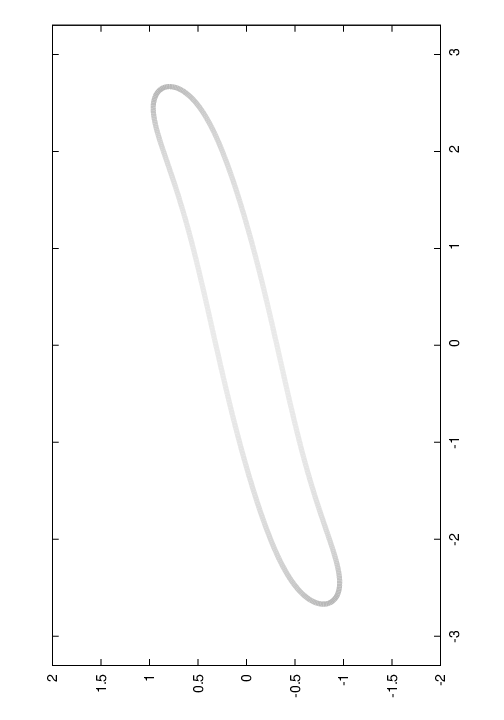}
\includegraphics[angle=-90,width=\localwidth]{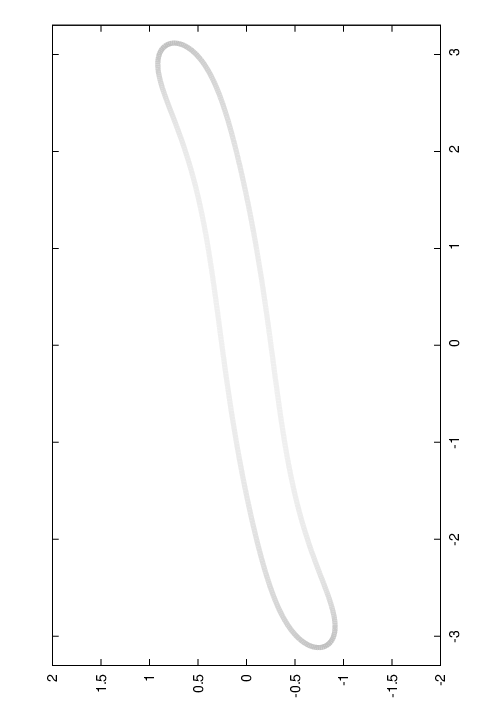}
\caption{(2\,adapt$_{9,4}$)
The time evolution of a drop in shear flow for (\ref{eq:gamma2}) with
$\beta = 0.25$ (top) and $\beta = 0.5$ (bottom). 
Plots are at times $t=0,\,4,\,8,\,12$. 
The grey scales linearly with the surfactant concentration ranging from 
0.2 (white) to 1.6 (black).
}
\label{fig:2dLain_nr}
\end{figure}%
\begin{figure}
\center
\newcommand\localwidth{0.48\textwidth}
\includegraphics[angle=-90,width=\localwidth]{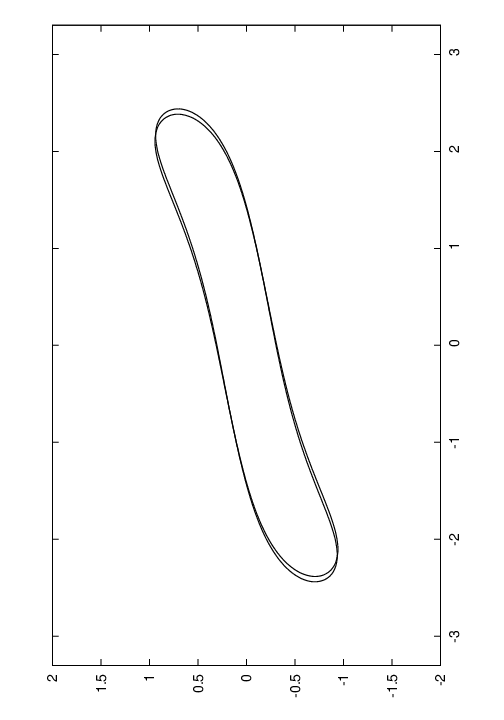}
\includegraphics[angle=-90,width=\localwidth]{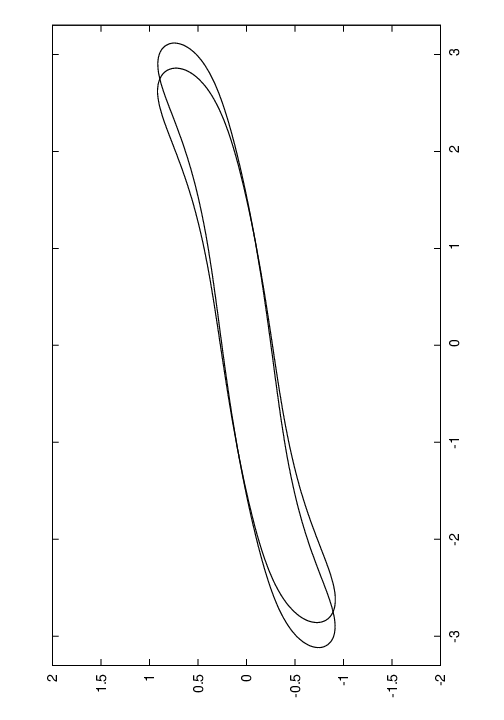}
\caption{(2\,adapt$_{9,4}$)
Comparison of the final drop shapes in shear flow for a linear
(\ref{eq:gamma1}) and a nonlinear (\ref{eq:gamma2}) equation of state with
$\beta = 0.25$ (left) and $\beta = 0.5$ (right). 
In each case the shape for (\ref{eq:gamma2}) is more elongated.
}
\label{fig:2dLaiLain_nr}
\end{figure}%

In Figure~\ref{fig:Fs} we compare the previously used (\ref{eq:gamma1}) 
with $\beta = 0.5$ and (\ref{eq:gamma2}) with $\psi_\infty^{-1} = \beta = 0.5$ 
to (\ref{eq:gamma2}) with $\beta = 0.5$ and $\psi_\infty = 1.3$. This
indicates that the initial drop should now be more unstable.
However, the evolution is not very different to what we saw before, see
Figure~\ref{fig:2dLain1.3}. This is despite the maximum discrete surfactant
concentration being $\approx 1.08$, which means that the discrete surface
tension $\gamma(\Psi^m)$ at times is negative. In fact, the observed minimum
value is $ < -0.03$, compare with Figure~\ref{fig:Fs}, but this posed no
problem for our numerical method.
\begin{figure}
\center
\newcommand\localwidth{0.48\textwidth}
\includegraphics[angle=-90,width=\localwidth]{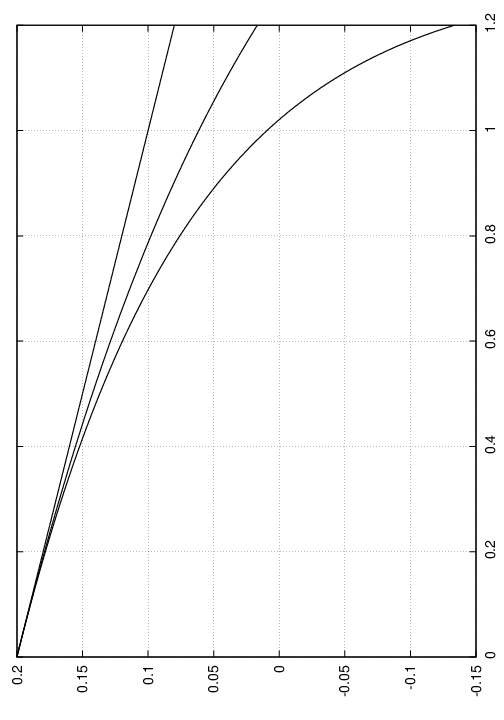}
\caption{($\beta=0.5$)
Plots of $\gamma(r)$ for the linear equation of state (\ref{eq:gamma1}) and
the nonlinear equation of state (\ref{eq:gamma2}) with $\psi_\infty=2$ and 
$1.3$.}
\label{fig:Fs}
\end{figure}%
\begin{figure}
\center
\newcommand\localwidth{0.24\textwidth}
\includegraphics[angle=-90,width=\localwidth]{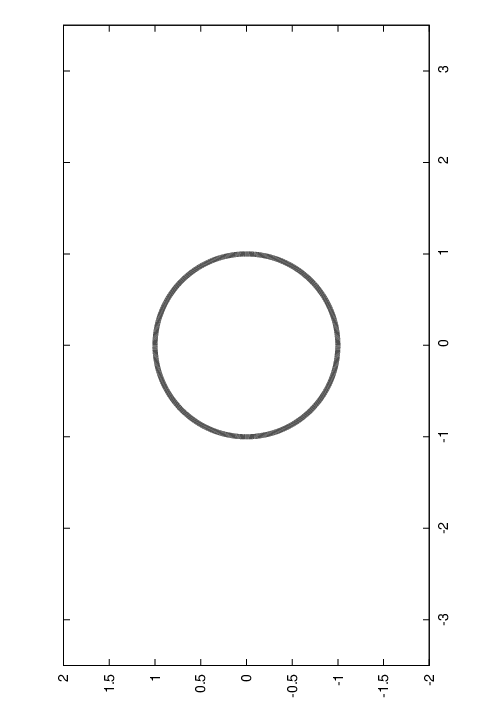}
\includegraphics[angle=-90,width=\localwidth]{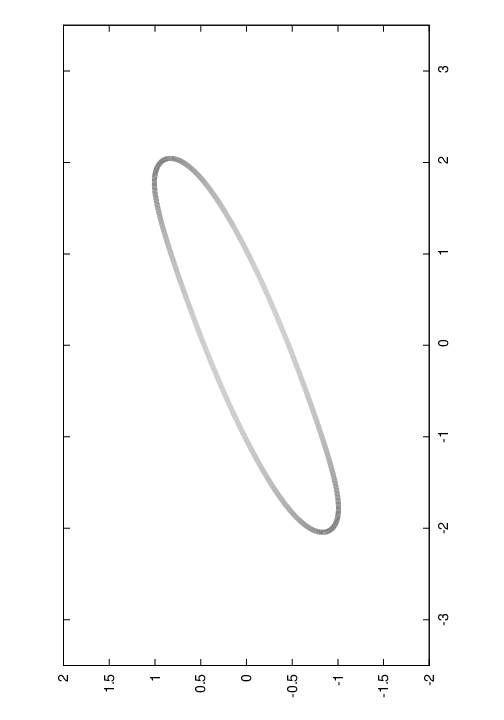}
\includegraphics[angle=-90,width=\localwidth]{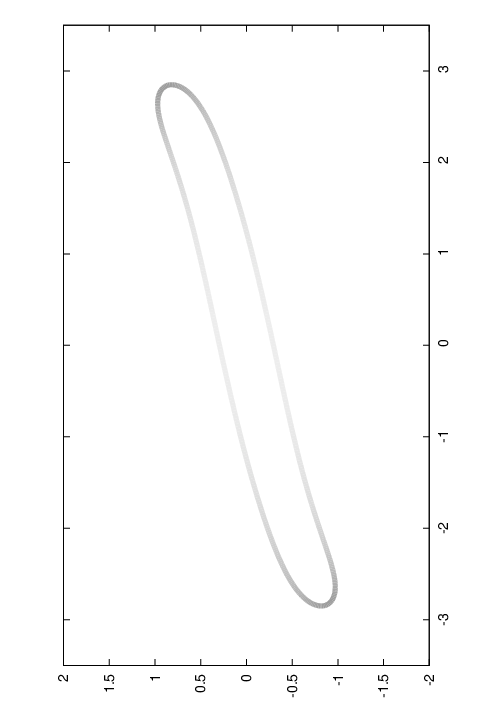}
\includegraphics[angle=-90,width=\localwidth]{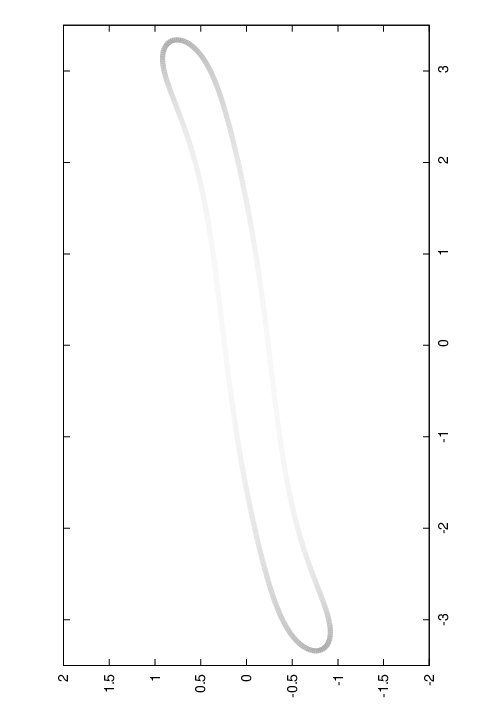}
\caption{(2\,adapt$_{9,4}$)
The time evolution of a drop in shear flow for (\ref{eq:gamma2}) with
$\beta=0.5$ 
and $\psi_\infty = 1.3$. 
Plots are at times $t=0,\,4,\,8,\,12$. 
The grey scales linearly with the surfactant concentration ranging from 
0.3 (white) to 1.1 (black).
}
\label{fig:2dLain1.3}
\end{figure}%

On returning back to the linear equation of state (\ref{eq:gamma1}), 
we also present a numerical simulation for 
different densities and viscosities. In particular, we leave all the parameters 
as in (\ref{eq:Lai}), but now choose
\begin{equation*}
\rho_+ = 10\,,\quad \rho_- = 1\,,\quad \mu_+ = 1\,,\quad \mu_- = 0.1\,.
\end{equation*}
We show the evolution of the drop in Figure~\ref{fig:2dLai_rhomu} for $\beta =
0$, $0.25$ and $0.5$. In contrast to Figure~\ref{fig:2dLai_nr}, the presence of
surfactant has very little impact on the shape of the drop
here. However, the interfaces in Figure~\ref{fig:2dLai_rhomu} are
more distorted and have higher curvatures at the ends, which is a
well-known fact when the viscosity of the drop is much less than the one of
the surrounding fluid, see \cite{RenardyRC02}.

\begin{figure}
\center
\newcommand\localwidth{0.24\textwidth}
\includegraphics[angle=-90,width=\localwidth]{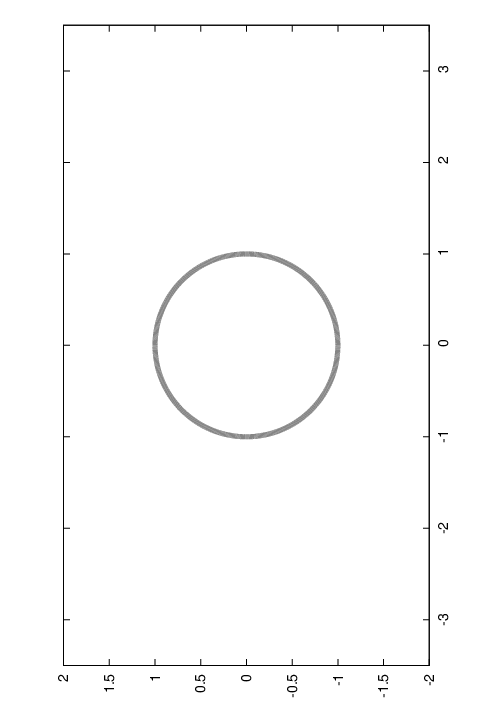}
\includegraphics[angle=-90,width=\localwidth]{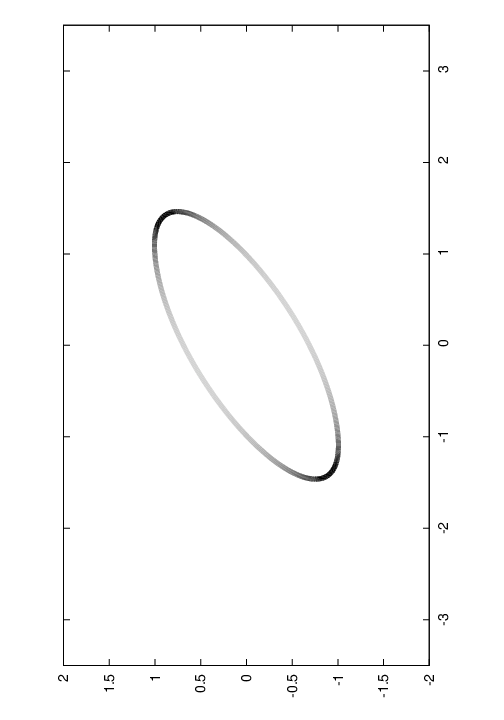}
\includegraphics[angle=-90,width=\localwidth]{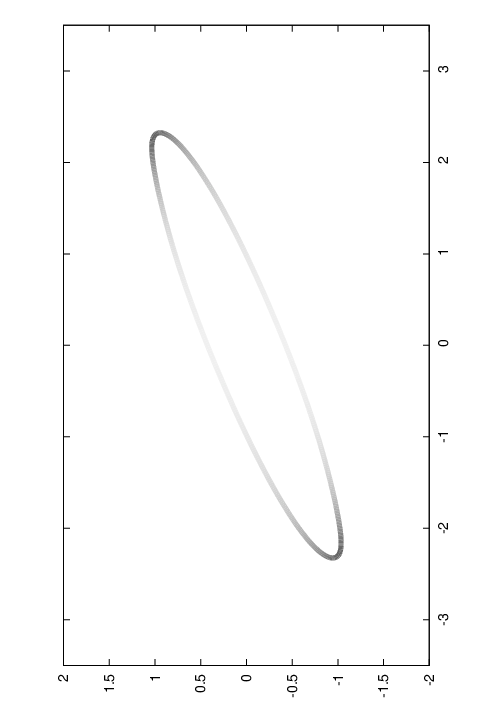}
\includegraphics[angle=-90,width=\localwidth]{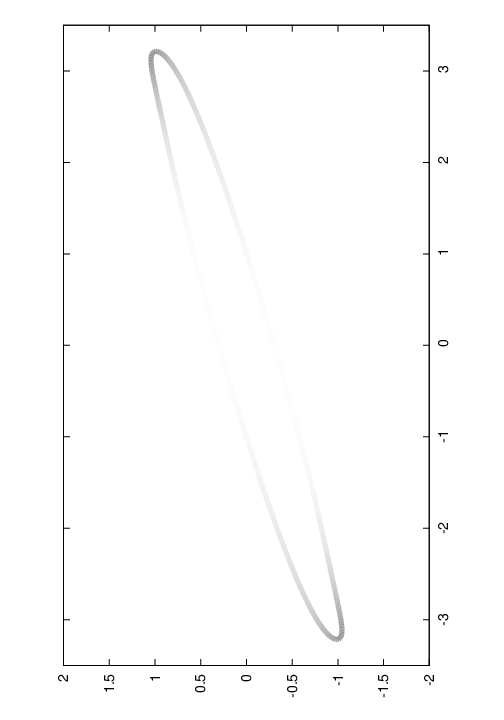}
\includegraphics[angle=-90,width=\localwidth]{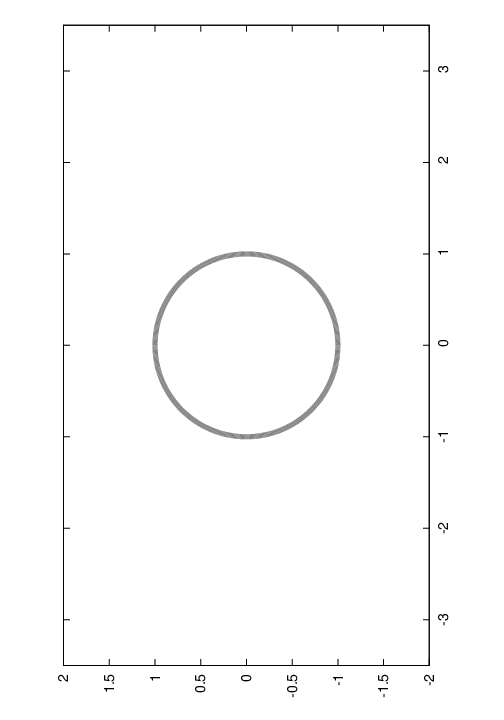}
\includegraphics[angle=-90,width=\localwidth]{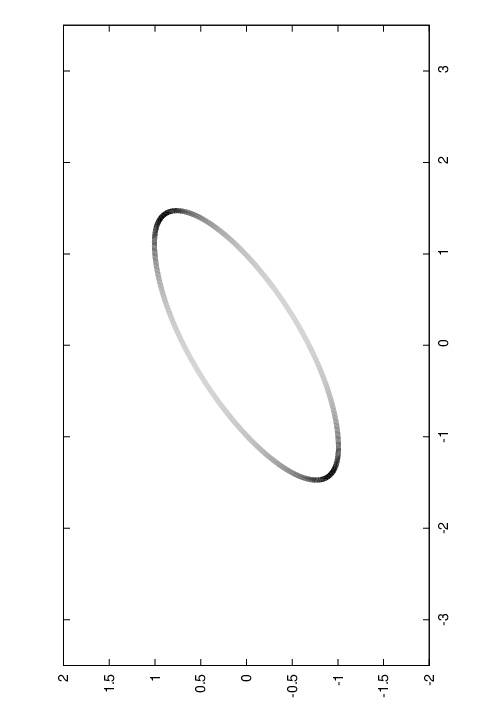}
\includegraphics[angle=-90,width=\localwidth]{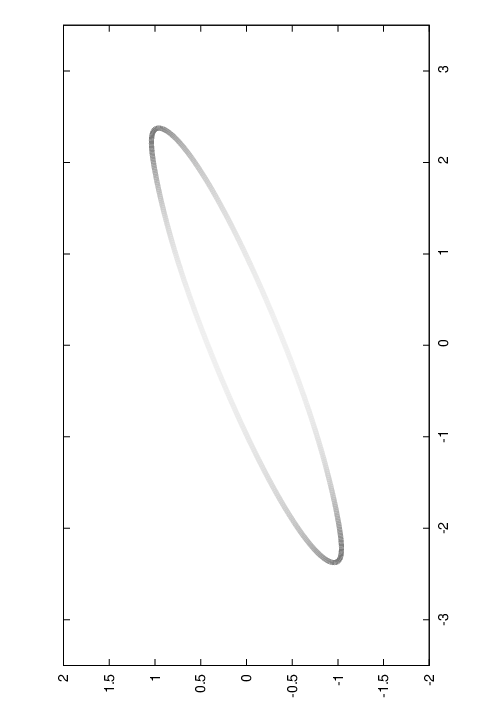}
\includegraphics[angle=-90,width=\localwidth]{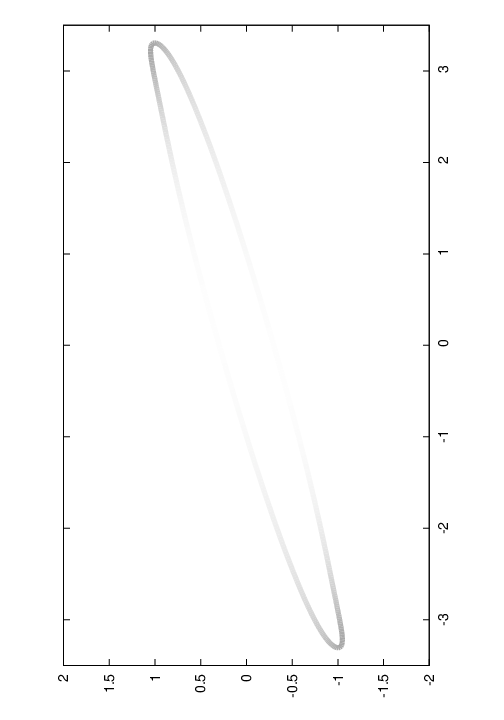}
\includegraphics[angle=-90,width=\localwidth]{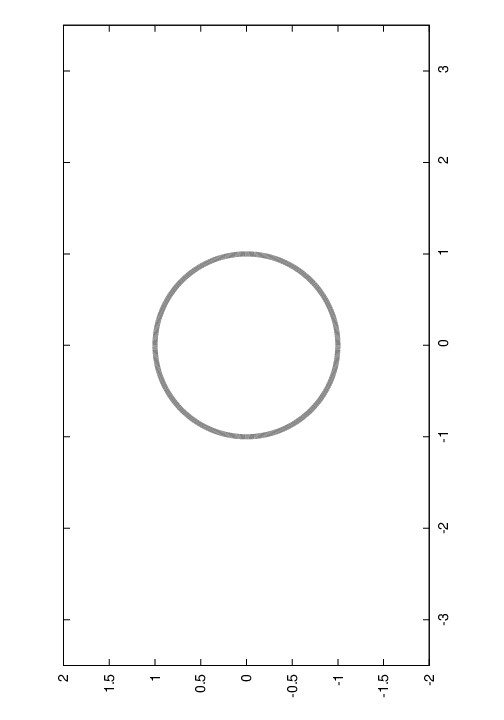}
\includegraphics[angle=-90,width=\localwidth]{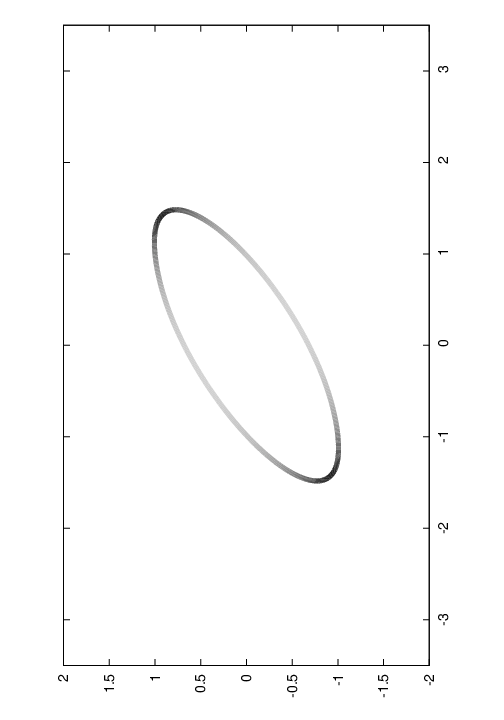}
\includegraphics[angle=-90,width=\localwidth]{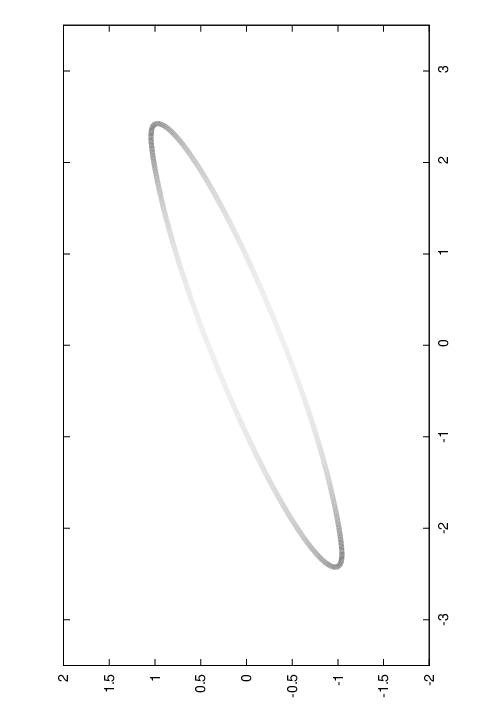}
\includegraphics[angle=-90,width=\localwidth]{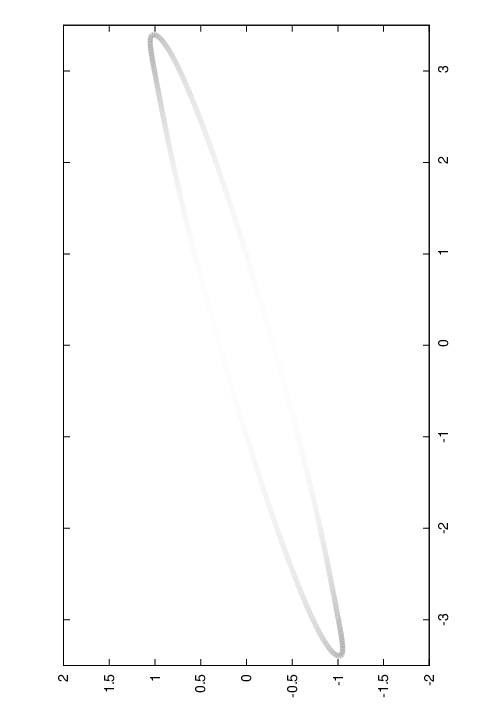}
\caption{(2\,adapt$_{9,4}$)
The time evolution of a drop in shear flow for (\ref{eq:gamma1}) with 
$\beta = 0$ (top), $\beta = 0.25$ (middle) and $\beta = 0.5$ (bottom). 
Here $\rho_+ = 10$, $\rho_- = 1$, $\mu_+ = 1$, $\mu_- = 0.1$.
Plots are at times $t=0,\,2,\,4,\,6$. 
The grey scales linearly with the surfactant concentration ranging from 
0.3 (white) to 1.3 (black).
}
\label{fig:2dLai_rhomu}
\end{figure}%

\subsection{Numerical simulations in 3d} \label{sec:62}

In this section we consider some numerical simulations for two-phase flow with
insoluble surfactant in three space dimensions. Here we will always report on
simulations for our preferred scheme (\ref{eq:HGa}--e).

\subsubsection{Rising bubble benchmark problem 1} \label{sec:332}
Here we consider the natural 3d analogue of the problem in \S\ref{sec:612}.
To this end, we let $\Omega =
(0,1) \times (0,1) \times (0.2)$ with 
$\partial_1\Omega = [0,1] \times [0,1] \times \{0,2\}$ and 
$\partial_2\Omega = \partial\Omega \setminus \partial_1\Omega$.
Moreover, we set $T=3$, $\Gamma_0 = \{ \vec z \in \R^3 : |\vec z -
(\frac12, \frac12, \frac12)^T| = \frac14\}$, and choose the physical 
parameters as in (\ref{eq:Hysing1}).
The time interval chosen for the simulation is again $[0,T]$ with $T=3$.
For the surfactant problem we choose the parameters $\Ds = 0.1$ and 
(\ref{eq:gamma1}) with $\beta = 0.5$.

Some quantitative values for the evolution are given in 
Table~\ref{tab:3ddataa}, where 
we have introduced the natural extensions of the 
quantities defined in (\ref{eq:benchmarkm}). In particular,
the discrete approximations of the
$x_3$-component of the bubble's centre of mass and the ``degree of sphericity''
are defined by
\begin{align*}
z_c^m & = \frac1{\mathcal{L}^3(\Omega_-^m)}\,\int_{\Omega_-^m} x_3 \dL3
= \frac3{\int_{\Gamma^m} \vec X^m\,.\,\vec \nu^m \dH{2}}
\int_{\Gamma^m} \tfrac12\,(\vec X^m \,.\,\vec\ek_3)^2\,
(\vec\nu^m \,.\,\vec\ek_3) \dH{2} \,,
\quad \nonumber \\ 
\strikes^m & = \pi^\frac13\,[6\,\mathcal{L}^3(\Omega_-^m)]^\frac23\,
[\mathcal{H}^{2}(\Gamma^m)]^{-1} 
\,. 
\end{align*}
\begin{table}
\center
\begin{tabular}{l|r|r}
\hline
& adapt$_{5,2}$ & adapt$_{6,3}$ \\
\hline 
$\Mloss$ & 
  0.0\% & 0.0\% \\
$\strikes_{\min}$ & 
 0.9570 & 0.9508 \\
$t_{\strikes = \strikes_{\min}}$ & 
 3.0000 & 3.0000 \\
$V_{c,\max}$ & 
 0.3822 & 0.3845 \\
$t_{V_c = V_{c,\max}}$ & 
 1.1930 & 1.0790 \\
$z_c(t=3)$ & 
 1.5515 & 1.5555 \\
\hline
\end{tabular} \qquad\qquad
\begin{tabular}{l|r|r}
\hline
& adapt$_{5,2}$ & adapt$_{6,3}$ \\
\hline 
$\Mloss$ & 
  0.0\% & 0.0\% \\
$\strikes_{\min}$ & 
 0.9348 & 0.9297 \\
$t_{\strikes = \strikes_{\min}}$ & 
 2.9300 & 2.9970 \\
$V_{c,\max}$ & 
 0.3252 & 0.3296 \\
$t_{V_c = V_{c,\max}}$ & 
 0.8160 & 0.8960 \\
$z_c(t=3)$ & 
 1.3807 & 1.3902 \\
\hline
\end{tabular}
\caption{Some quantitative results for the 3d benchmark problem 1.
Without surfactant (left) and with surfactant (right).}
\label{tab:3ddataa}
\end{table}%

In what follows we present some visualizations of the numerical results for 
the runs with adapt$_{6,3}$. A comparison of the final meshes for the runs with
and without surfactant can be seen in Figure~\ref{fig:3dcomp05}, while
the discrete surfactant concentration for the run with surfactant
can be seen in Figure~\ref{fig:3dbubble05}.

\begin{figure}
\center
\hspace*{-2.1cm}
\includegraphics[angle=-90,width=0.5\textwidth]{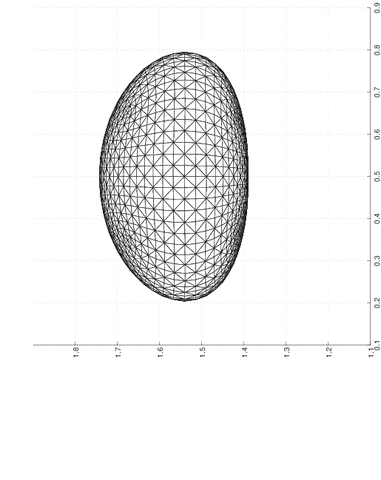}
\includegraphics[angle=-90,width=0.5\textwidth]{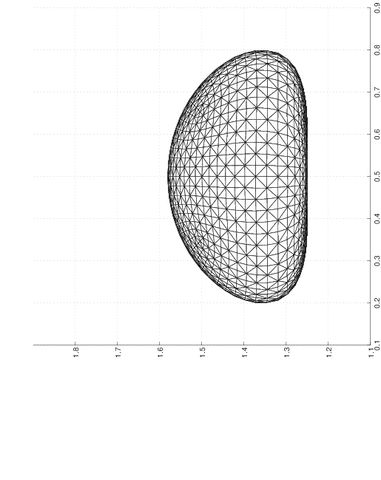}
\caption{(adapt$_{6,3}$)
Side view of the final bubble for the 3d benchmark problem 1 at time $T=3$. 
Without surfactant (left) and with surfactant (right).}
\label{fig:3dcomp05}
\end{figure}%
\begin{figure}
\center
\includegraphics[angle=-0,width=0.4\textwidth]{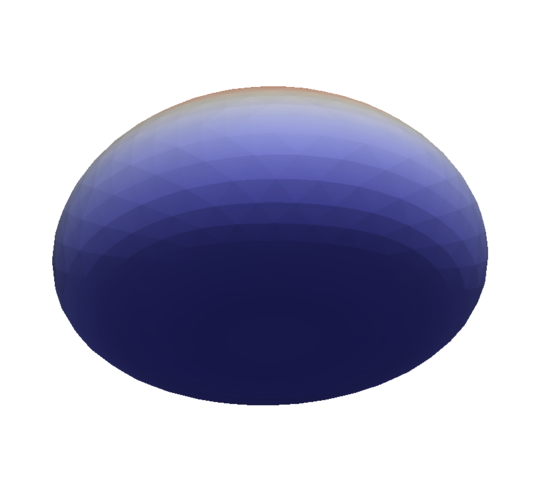}
\qquad\qquad
\includegraphics[angle=-0,width=0.4\textwidth]{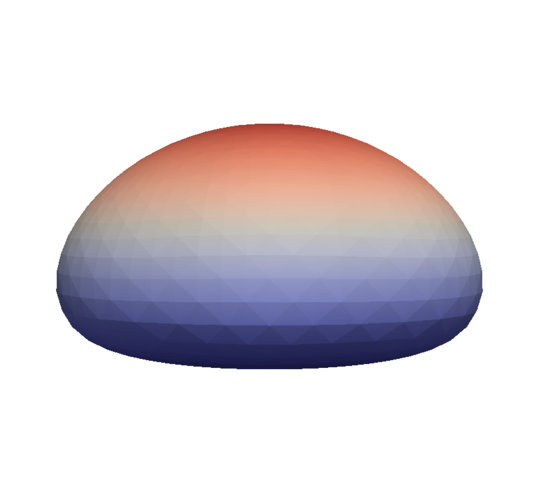}
\caption{(adapt$_{6,3}$)
The final surfactant concentration $\Psi^M$ on $\Gamma^M$.
Here the colour ranges from red (0.5) to blue (1.2).
}
\label{fig:3dbubble05}
\end{figure}%

\subsubsection{Bubble in shear flow}
In this subsection we report on the 3d analogues of the computations shown in
Figure~\ref{fig:2dLai_nr}.
In particular, in Figure~\ref{fig:3dLai_nr} we show shear flow experiments
on the domain $\Omega = (-5,5)\times (-2,2)^2$ with
$\partial\Omega=\partial_1\Omega$ and $\vec g(\vec z) = (\frac12\,z_3,0,0)^T$.
The physical parameters are as in (\ref{eq:Lai}), and we compare the 
evolutions for the linear equation of state (\ref{eq:gamma1}) for
(i) $\beta = 0$, (ii) $\beta = 0.25$ and (iii) $\beta = 0.5$.
As the
discretization parameters we choose adapt$_{5,2}^\star$, which are the
same as for ${\rm adapt}_{5,2}$, apart from $\tau = 0.01$ and
$(K_\Gamma,J_\Gamma) = (1538, 3072)$, i.e.\ adapt$_{5,2}^\star$ uses a larger
time step size and a finer interface mesh compared to ${\rm adapt}_{5,2}$.
Our three dimensional results turn out to be very similar to the two
dimensional results in Figure~\ref{fig:2dLai_nr};
see Figure~\ref{fig:3dLai_nr} for more details.
\begin{figure}
\center
\newcommand\localwidth{0.24\textwidth}
\includegraphics[angle=-0,width=\localwidth]{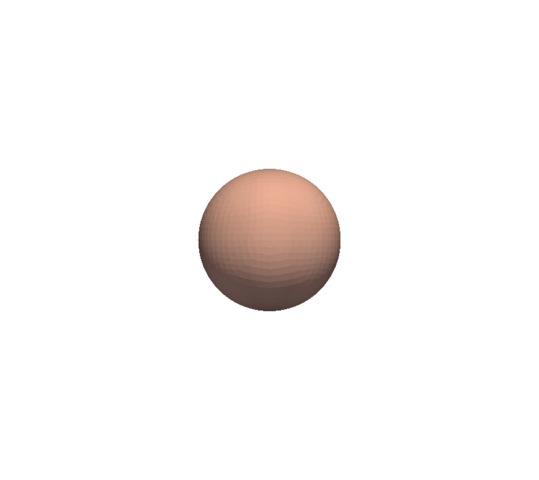}
\includegraphics[angle=-0,width=\localwidth]{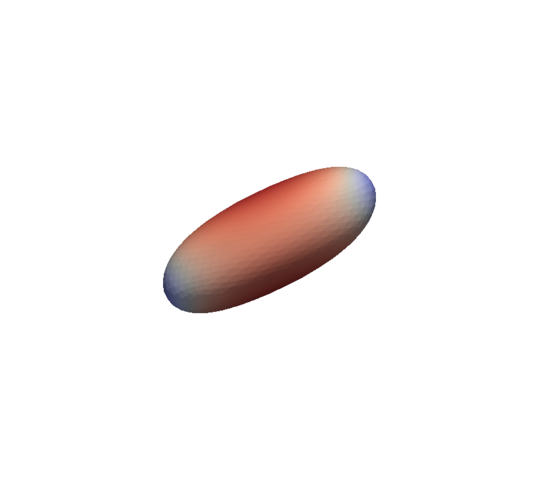}
\includegraphics[angle=-0,width=\localwidth]{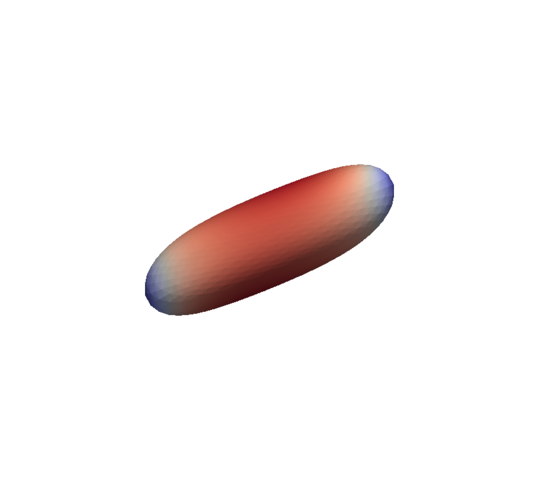}
\includegraphics[angle=-0,width=\localwidth]{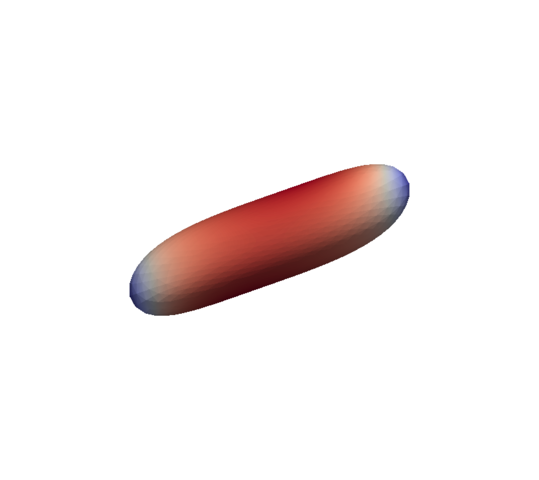}
\includegraphics[angle=-0,width=\localwidth]{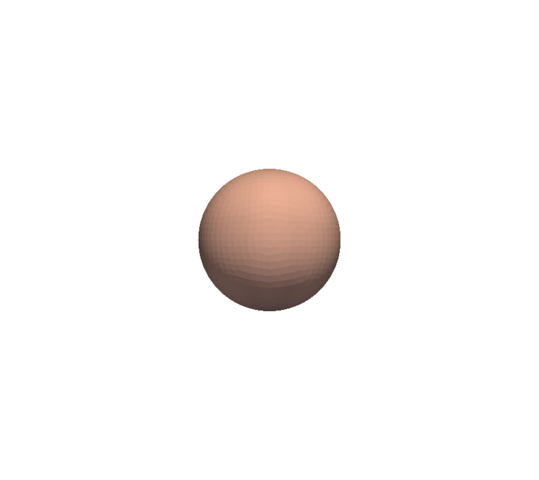}
\includegraphics[angle=-0,width=\localwidth]{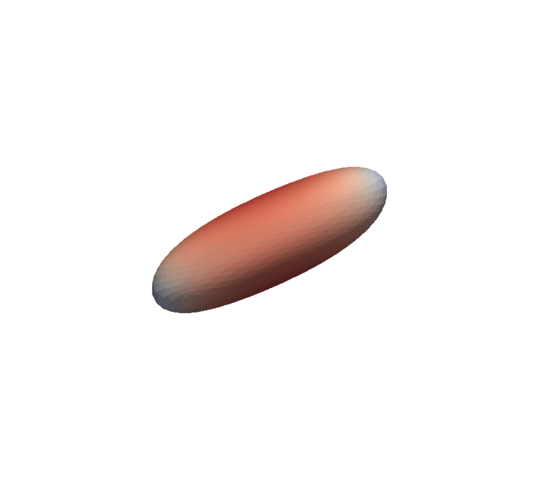}
\includegraphics[angle=-0,width=\localwidth]{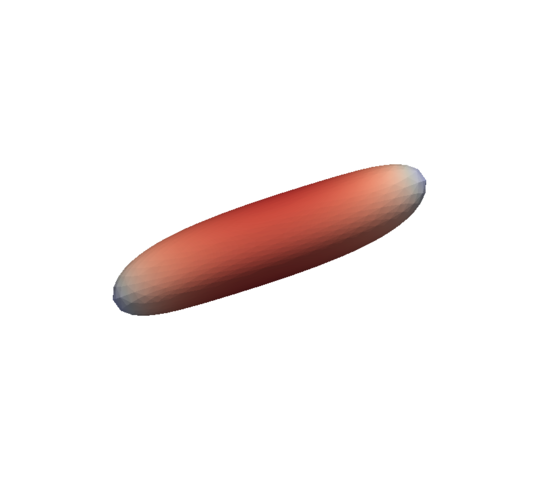}
\includegraphics[angle=-0,width=\localwidth]{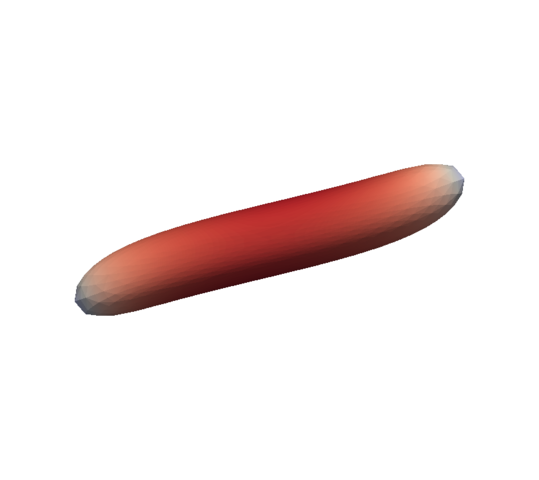}
\includegraphics[angle=-0,width=\localwidth]{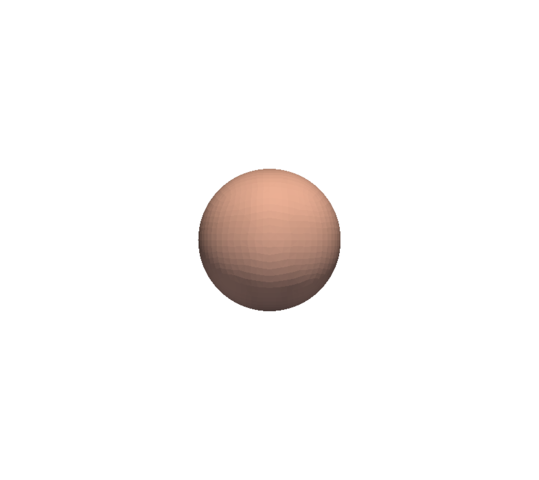}
\includegraphics[angle=-0,width=\localwidth]{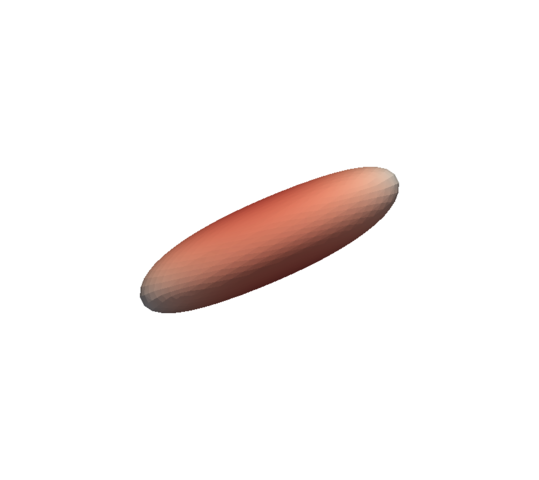}
\includegraphics[angle=-0,width=\localwidth]{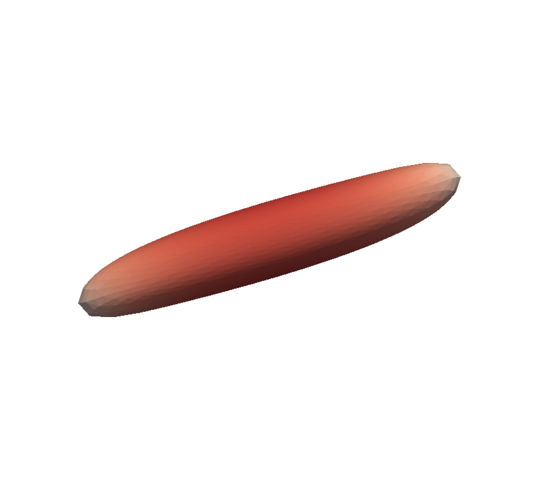}
\includegraphics[angle=-0,width=\localwidth]{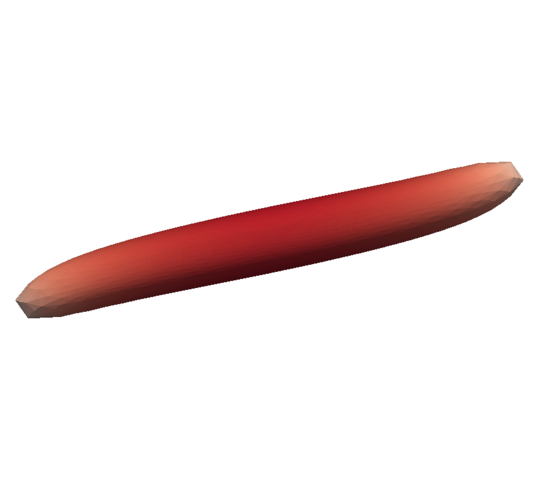}
\caption{(adapt$_{5,2}^\star$)
The discrete surfactant concentrations $\Psi^m$ at times $t=0,4,8,12$
for $\beta=0$ (top), $\beta = 0.25$ (middle) and $\beta = 0.5$ (bottom).
The colour ranges from red (0.5) to blue (1.9).}
\label{fig:3dLai_nr}
\end{figure}%

\begin{appendix}
\setcounter{equation}{0}
\renewcommand{\theequation}{\Alph{section}.\arabic{equation}}
\section{Exact solution for the advection diffusion equation} \label{sec:A}
Following \citet[Example~7.3]{DziukE07}, we present a true solution to the 
inhomogeneous advection diffusion equation 
\begin{equation} \label{eq:1surff}
\matpartu\,\psi + \psi\,\nabs\,.\vec u - \Delta_s\,\psi =
f_\Gamma \qquad \mbox{on } \Gamma(t)\,,
\end{equation}
recall (\ref{eq:1surf}), in a situation where 
the fluid velocity $\vec u$, and hence the evolution of
$\Gamma(t)$, is given.
The surface is given by 
$\Gamma(t) = \{ \vec z \in \R^d : \phi(\vec z, t) = 1 \}$, where
\begin{equation*} 
\phi(\vec z, t) = [a(t)]^{-1}\,z_1^2 + \sum_{i=2}^d z_i^2 \,,
\end{equation*}
so that the moving surface $\Gamma(t)$
is an ellipsoid with time dependent $x_1$-axis. Here we choose 
$$a(t) = 1 + \sin(\pi\,t)\,,$$ 
and as the parameterization $\vec x (\cdot, t) :
\mathbb{S}^{d-1} \to \Gamma(t)$, where 
$\mathbb{S}^{d-1} := \{ \vec q \in \R^d : |\vec q| = 1\}$, we choose
$$
\vec x ( \vec q, t) = [a(t)]^{\frac12}\,q_1\,\vec\ek_1 +
\sum_{i=2}^d q_i\,\vec\ek_i \qquad\forall\ \vec q \in \mathbb{S}^{d-1}\,,\quad
t \in \Rgeq \,.
$$
On recalling (\ref{eq:V}), for the fluid velocity we naturally choose
\begin{equation} \label{eq:trueu}
\vec u (\vec z, t) = \tfrac12\,[a(t)]^{-1}\,a'(t)\,z_1\,\vec\ek_1 \qquad\ 
\vec z \in \Omega\,,
\end{equation}
so that
\begin{equation*} 
\vec u (\vec z, t) \!\mid_{\Gamma(t)} = \vec{\mathcal{V}} (\vec z, t) 
\qquad\ \vec z \in \Gamma(t)\,.
\end{equation*}
As an exact solution we choose $\psi(\vec z, t) = e^{-6\,t}\,z_1\,z_2$,
and hence it remains to calculate the right hand side $f_\Gamma$ in 
(\ref{eq:1surff}) for our chosen $\psi$ and $\vec u$. To this end we note that
\begin{equation} \label{eq:truef}
f_\Gamma =
\matpartu\,\psi + \psi\,\nabs\,.\vec u - \Delta_s\,\psi\,,
\end{equation}
with
\begin{align*} 
\matpartu\,\psi 
& = ( \tfrac12\,[a(t)]^{-1}\,a'(t) - 6 )\,\psi \,, \\
\psi\,\nabs\,.\vec u & = \tfrac12\,[a(t)]^{-1}\,a'(t)\,(1 - \nu_1^2)\,\psi\,,\\
-\Delta_s\, \psi (\vec z, t) & =
e^{-6\,t} \left[ 
2\,\nu_1\,\nu_2 - (\nu_1\,z_2 + \nu_2\,z_1)\,\varkappa(\vec z, t) \right],
\end{align*}
where 
$\vec \nu (\vec z, t) = \frac{\nabla \phi(\vec z, t)}{|\nabla \phi(\vec z, t|}
\in \R^d$ denotes the normal to $\Gamma(t)$ at $\vec z \in \Gamma(t)$, and
where
\begin{equation} \label{eq:k1}
\varkappa = -\nabs\,.\,\vec \nu = -\nabla\,.\,\vec \nu 
= - |\nabla\,\phi|^{-1}\, \sum_{i=1}^d \left[ 
\left(1 - |\nabla\,\phi|^{-2}\, 
\left(\frac{\partial \phi}{\partial z_i}\right)^2 \right)
\frac{\partial^2 \phi}{\partial z_i^2}\right]
\end{equation}
denotes the mean curvature of $\Gamma(t)$. Of course, for our example we have
that $\nabla\,\phi (\vec z, t) = 2\,[a(t)]^{-1}\,z_1\,\vec\ek_1 +
2\sum_{i=2}^d z_i\,\vec\ek_i$, and so (\ref{eq:k1}) reduces to
\begin{equation*} 
\varkappa = - 2\,|\nabla\,\phi|^{-1}\, [a(t)]^{-1}\,
\left(1 - 4\,|\nabla\,\phi|^{-2}\, [a(t)]^{-2}\,z_1^2 \right)
- 2\,|\nabla\,\phi|^{-1} \sum_{i=2}^d 
\left(1 - 4\,|\nabla\,\phi|^{-2}\,z_i^2 \right) .
\end{equation*}

\end{appendix}

\def\soft#1{\leavevmode\setbox0=\hbox{h}\dimen7=\ht0\advance \dimen7
  by-1ex\relax\if t#1\relax\rlap{\raise.6\dimen7
  \hbox{\kern.3ex\char'47}}#1\relax\else\if T#1\relax
  \rlap{\raise.5\dimen7\hbox{\kern1.3ex\char'47}}#1\relax \else\if
  d#1\relax\rlap{\raise.5\dimen7\hbox{\kern.9ex \char'47}}#1\relax\else\if
  D#1\relax\rlap{\raise.5\dimen7 \hbox{\kern1.4ex\char'47}}#1\relax\else\if
  l#1\relax \rlap{\raise.5\dimen7\hbox{\kern.4ex\char'47}}#1\relax \else\if
  L#1\relax\rlap{\raise.5\dimen7\hbox{\kern.7ex
  \char'47}}#1\relax\else\message{accent \string\soft \space #1 not
  defined!}#1\relax\fi\fi\fi\fi\fi\fi}

\end{document}